\documentclass[aos,preprint]{imsart}

\RequirePackage{amsthm,amsmath,amsfonts,amssymb}
\RequirePackage[numbers]{natbib}
\RequirePackage[colorlinks,citecolor=blue,urlcolor=blue]{hyperref}

\usepackage{color, multirow, graphicx}
\usepackage{float}
\usepackage{dsfont}
\usepackage{subcaption}
 \usepackage{siunitx}
\startlocaldefs

\newtheorem{lemma}{Lemma}[section]

\newtheorem{theorem}[lemma]{Theorem}

\newtheorem{proposition}[lemma]{Proposition}

\newtheorem{definition}[lemma]{Definition}

\newtheorem{corollary}[lemma]{Corollary}

\newtheorem{example}[lemma]{Example}

\newtheorem{remark}[lemma]{Remark}

\providecommand{\cov}[2]{\mathrm{Cov} \rbraces{#1 , #2}} 
\providecommand{\Var}[1]{\mathrm{Var} \rbraces{#1}}	
\providecommand{\EQ}[2]{\mathbb{E}_{#2} \rbraces{#1}}
\providecommand{\E}[1]{\EQ{#1}{}} 	
\providecommand{\abs}[1]{\left\lvert#1 \right\rvert} 
\newcommand\norm[1]{\left\lVert#1\right\rVert}
\providecommand{\Pb}{\mathbb{P}}
\providecommand{\distConv}{\xrightarrow{\mathcal{D}}}	
\providecommand{\pConv}{\xrightarrow{\mathbb{P}}}	
\providecommand{\NoD}[2]{\mathcal{N} \rbraces{#1 , #2}}

 \newcommand{\rbraces}[1]{\left( #1 \right)}
\providecommand{\cov}[2]{\mathrm{Cov} \rbraces{#1 , #2}} 
\providecommand{\Var}[1]{\mathrm{Var} \rbraces{#1}}	
\providecommand{\EQ}[2]{\mathbb{E}_{#2} \rbraces{#1}}
\providecommand{\E}[1]{\EQ{#1}{}} 	
\providecommand{\abs}[1]{\left\lvert#1 \right\rvert}
\providecommand{\Pb}{\mathbb{P}}
\newcommand{\limn}{\lim\limits_{n \rightarrow \infty}}
\providecommand{\distConv}{\xrightarrow{\mathcal{D}}}	
\providecommand{\pConv}{\xrightarrow{\mathbb{P}}}	
\providecommand{\NoD}[2]{\mathcal{N} \rbraces{#1 , #2}}
\newcommand{\1}{\ensuremath{\mathds{1}}}
\newcommand{\N}{{\mathbb N}}
\newcommand{\R}{{\mathbb R}}
\newcommand{\elln}{{\tilde{\ell}_n}}
\newcommand{\bn}{{\tilde{b}_n}}

\endlocaldefs

\begin{document}

\begin{frontmatter}

\title{Detecting changes in the trend function of heteroscedastic time series}
\runtitle{Detecting changes in the trend function of heteroscedastic time series}

\begin{aug}

\author[A]{\fnms{Sara Kristin} \snm{Schmidt}\ead[label=e1]{sara.schmidt@ruhr-uni-bochum.de}}

\address[A]{Department of Mathematics, Ruhr-Universit\"at Bochum,
  Universit\"atsstra\ss e 150, 44780 Bochum, Germany, \printead{e1}}

\end{aug}

\begin{abstract}
We propose a new asymptotic test to assess the stationarity of a time series' mean that is applicable in the presence of both heteroscedasticity and short-range dependence.  Our test statistic is composed of Gini's mean difference of local sample means. To analyse its asymptotic behaviour, we develop new limit theory for U-statistics of strongly mixing triangular arrays under non-stationarity. Most importantly, we show asymptotic normality of the test statistic under the hypothesis of a constant mean and prove the test's consistency against a very general class of alternatives, including both smooth and abrupt changes in the mean. We propose estimators for all parameters involved, including an adapted subsampling estimator for the long run variance, and show their consistency. 
Our procedure is practically evaluated in an extensive simulation study and in two data examples. 
\end{abstract}



\end{frontmatter}

\section{Introduction}

Tests for the constancy of a time series' trend function have attracted vivid research interest over the past decades. 
 In their simplest form, they aim to test the hypothesis of stationarity against the alternative of a single change in the mean, considering 
 a sequence of observations $X_i=\mu(i/n)+Y_i$, $i=1, \ldots, n$, with  i.i.d. errors $Y_1, \ldots, Y_n$.
A vast literature has developed around this basic testing problem, proposing different methodologies, generalizing the formulation of the testing problem with respect to the mean function $\mu$, and weakening the assumptions on the error sequence $(Y_i)_{i\in\N}$.
In the classical framework of change point analysis, one tests against the alternative of a piecewise constant mean and if the hypothesis of stationarity is rejected, one oftentimes estimates the location of the change points as well 
(see, for example, the wild binary segmentation algorithm proposed in Fryzlewicz \citep{Fryzlewicz.2014}, the simultaneous multiscale estimator SMUCE introduced in Frick, Munk and Sieling \citep{Frick.2014}, or the monograph  by Cs\"org\H o and Horv\'ath \citep{Csorgo.1997} for a survey). 
More generally, Wu and Chu \citep{WuChu.1993}, Loader \citep{Loader.1996} and Wu and Zhao \citep{WuZhao.2007} 
 examine structural breaks in an otherwise smooth mean function, and  Wu, Woodroofe and Mentz \citep{Wu.2001} study the hypothesis of a constant trend function against the alternative of a monotonic one.
In these and other works, the error sequence is either assumed to be independent or to be short-range dependent, but it is always stationary. 
Relaxing the assumption of homoscedasticity, Horv\'ath \citep{Horvath.1993}
 tests for a structural break in the mean and variance of a sequence $X_i=\mu(i/n)+\sigma(i/n)Y_i$, $i=1, \ldots, n$, with i.i.d. standard Gaussian errors, while Pein, Sieling and Munk \citep{Pein.2017} study the SMUCE estimator for the detection of multiple change points in a similar model. 
The literature considering models with both short-range dependent errors and heteroscedasticity is more scarce. 
Zhao and Li \citep{ZhaoLi.2013} and Pe\v{s}ta and Wendler \citep{Pesta.2020} combine a self-normalized cumulative sum test with  wild bootstrap to test for a single change in the mean.
G\'orecki, Horv\'ath and Kokoszka \citep{Gorecki.2018} 
consider a more general class of alternatives, including mean functions with multiple change points. 
Zhou \citep{Zhou.2013}  proposes a bootstrap approach to test for stationarity
against the alternative of a piecewise Lipschitz-continuous trend function, and Dette and Wu \citep{DetteWu.2019} investigate relevant deviations of a smooth mean function from its original value.

We propose an asymptotic test for the constancy of a time series' trend function that is able to detect arbitrary changes in the mean, including smooth and abrupt ones, and that allows for both short-range dependence and heteroscedasticity. 
More precisely, we develop our  test for time series data $X_1, \ldots, X_n$  stemming from a triangular array of the form

\begin{equation}
\label{Eq: Triangular array X}
  X_i:=X_{i,n}=\mu({i}/{n})+ \sigma({i}/{n}) Y_i, \quad i=1, \ldots, n,
\end{equation}
where $(Y_i)_{i\in\N}$ is a stationary, $\alpha$-mixing process with mean zero and variance one.
The non-stationary mean and variance are described by the mean and the variance function, $\mu:[0,1]\rightarrow \R$ and $\sigma^2:[0,1]\rightarrow (0,\infty)$, respectively. 
In case $\mu$ and $\sigma$ are both continuous, model \eqref{Eq: Triangular array X} is a special case of local stationarity as introduced by Dahlhaus \citep{Dahlhaus.1997}. 
The hypothesis we want to test now reads  
\begin{equation}
\label{Eq: Hypothesis mu constant}
\mathbb{H}: \mu(x)=\mu_H \quad \text{for almost all } x\in [0,1].
\end{equation}

For the construction of the test statistic, we divide the observations $X_1, \ldots, X_n$ into $b_n$ blocks of length $\ell_n$, such that the $j$-th block comprises all indices $\{(j-1)\ell_n+1, \ldots, j\ell_n\}$. On each block $j$, we compute the local sample mean as 
 $$\hat{\mu}_j:=\hat{\mu}_{j,n}:=\frac{1}{\ell_n} \sum_{i=(j-1) \ell_n+1}^{j\ell_n} X_i$$
 and compare these local estimates pairwise, giving rise to the statistic
\begin{equation}\label{teststatistic}
  U(n):=\frac{1}{b_n(b_n-1)} \sum_{1\leq j\neq k \leq b_n} \left|\hat{\mu}_j -\hat{\mu}_k \right|.
\end{equation}
Intuitively, a high value of $U(n)$ hints at large differences between the local block means and hence at some form of non-stationarity in the mean. From a more general view point, $U(n)$ constitutes a U-statistic $\frac{1}{b_n(b_n-1)}\sum_{1\leq j\neq k \leq b_n} \left| \nu_j-\nu_k\right|$, known as Gini's mean difference, whose arguments $\nu_j$ stem from the triangular array of local sample means $(\hat{\mu}_{j,n})_{1\leq j \leq b_n, n\in \N}$.  Due to the heteroscedasticity of the observations $X_1, \ldots, X_n$, 
 the entries $\hat{\mu}_{j,n}$, $j=1, \ldots, b_n$, are non-stationary even under the hypothesis. 
In order to analyse the limit behaviour of $U(n)$, we thus develop completely new asymptotic theory for U-statistics of row-wise non-stationary, $\alpha$-mixing triangular arrays. 
In particular, we establish a central limit theorem for the statistic $U(n)$ under the hypothesis \eqref{Eq: Hypothesis mu constant} of a constant mean, which is the key result for the construction of our asymptotic test. We prove the test's consistency against a very general class of alternatives, namely against all non-constant c\`adl\`ag  functions $\mu$. This class includes, as special cases, the piecewise constant mean functions as considered in change-point analysis as well as arbitrary smoothly varying ones.  
Note that we deliberately omit a standardization of the local sample means in \eqref{teststatistic}, that is, we do not use $\hat{\mu}_j/\hat{\sigma}_j$ with  $\hat{\sigma}_j^2=\frac{1}{\ell_n} \sum_{i=(j-1) \ell_n+1}^{j\ell_n}\left(X_i - \hat{\mu}_j\right)^2$, since we would otherwise not be able to discriminate between high values of $U(n)$ caused by a non-constant mean and those caused by a non-constant variance (see, for example, Theorem \ref{Thm: Behaviour under A} below). As a result, the non-constant variance function $\sigma$ influences the limit distribution of $U(n)$ under the hypothesis of a constant mean (see, Theorems \ref{THM: LLN} and \ref{Thm: Central limit theorem}). 

Our test statistic is inspired by the one recently proposed by Schmidt et al. \citep{Schmidt.2021}. The authors test for stationarity of the variance in an absolutely regular time series via a test statistic based on Gini's mean difference of logarithmic local sample variances. 
Wornowizki, Fried and Meintanis \citep{Wornowizki.2017} previously utilized the same statistic to test for the constancy of the variance in a sequence of i.i.d. observations via a permutation principle. 
Apart from these works, 
to some extent related problems have also been considered by Vogt \citep{Vogt.2015} and by Juhl and Xiao \citep{Juhl.2005}. 
Note that our procedure could be combined with that of Schmidt et al. \citep{Schmidt.2021} to yield a test for stationarity of the first two moments. Moreover, it can be employed to evaluate the fit of some estimated mean function, e.g., of some parametric model as considered in Zhang and Wu \citep{Zhang.2011}, 
by subtracting the fitted trend function from the data and 
applying our test to the residuals. We will discuss such extensions later on in Section \ref{Sec: Extensions}. 

The performance of our test is evaluated in an extensive simulation study. Examining a variety of data-generating processes, including ARMA- and GARCH-processes,  we find our test to have good finite sample properties both under the hypothesis and under several alternatives considered, among which are smoothly varying as well as piecewise constant mean functions. In particular, we find our test to be robust with respect to different types of variation in the variance. 
As opposed to that, procedures constructed to test against only smooth or only abrupt variation in the mean will seldom achieve consistency in such a general setting, whereas tests designed for a constant variance are likely to be oversized in the presence of heteroscedasticity.  Our test has a low computation time and does not involve any bootstrap techniques since the test statistic is based on a pairwise comparison of local estimates and since critical values are easily obtained from the 
normal limit distribution. The applicability of our test is demonstrated in 
two data examples. As a first example, we analyse the well-studied annual central England temperature record from 1659 to 2020 and find that the annual average temperatures can be described by either a piecewise constant or a by cubic trend function, the latter coinciding with the findings in Zhang and Wu \citep{Zhang.2011}. In a second data example, we use our test to detect mean changes in the annual increments of the monthly number of live births of girls in Germany from 1950 to 2020.

The rest of the paper is structured as follows: Section \ref{Sec: Main results} presents our main results and is divided into four parts.  After introducing the required preliminaries in \ref{Subsec: Defs and assumptions}, the asymptotic results for the U-statistic $U(n)$ and the estimation of the nuisance parameters are covered in \ref{Subsec: Limit theory for U(n)} and \ref{Subsec: Estimation of Parameters}, respectively. These results are afterwards combined in \ref{Subsec: The asymptotic test} to yield the asymptotic test. Section \ref{Sec: Outline of Proofs} outlines the proof ideas for the main theorems, while Section \ref{Sec: Extensions} proposes some extensions of our test. The results of the simulation study are presented in Section \ref{Sec: Simulations}, and the data examples are treated in Section \ref{Sec: Data Example}. All proof details as well as some additional results from the simulation study and the data examples are deferred to an appendix.

\section{Main results}
\label{Sec: Main results}
\subsection{Basic definitions and assumptions} 
\label{Subsec: Defs and assumptions}
Our asymptotic test is based on observations $X_1, \ldots, X_n$ stemming from a triangular array of type \eqref{Eq: Triangular array X}, that is, $X_i=\mu({i}/{n})+ \sigma({i}/{n}) Y_i$, where $(Y_i)_{i\in \N}$ is a stationary,  $\alpha$-mixing process with mean zero and variance one.

\begin{definition}
The $\alpha$-mixing coefficient of two $\sigma$-fields $\mathcal{A}$ and $\mathcal{B}$ is defined as
$$\alpha(\mathcal{A},\mathcal{B})= \sup \{\abs{\Pb(A)\Pb(B)-\Pb(A\cap B)}: A\in \mathcal{A}, B\in \mathcal{B}  \}$$
and a process $(Y_i)_{i\in\N}$ is called $\alpha$-mixing (or strongly mixing) if
$$\alpha_Y(k):=\sup_{m\in\N} \alpha(\sigma\rbraces{Y_1, \dots, Y_m}, \sigma\rbraces{Y_{k+m}, Y_{k+m+1}, \dots })\longrightarrow 0 \quad \text{  as } k \rightarrow \infty.$$
\end{definition} 

The assumption of $\alpha$-mixing is met by a wide range of time series models $(Y_i)_{i\in \N}$; very general examples can be found in Doukhan \citep{Doukhan.1994}. In light of the conditions of our asymptotic test and its practical applicability, we focus on three important examples, all of which are $\alpha$-mixing with a geometric rate, that is, $\alpha_Y(k)=O(e^{-\xi k})$ for some $\xi>0$. 

\begin{example}{
\label{Example: ARMA, GARCH}
\begin{enumerate}
\item Let $(Y_i)_{i\in\N}$ be a strictly stationary, causal ARMA(p,q)-process, i.e. let 
$$Y_i= \varepsilon_i + \sum_{j=1}^p \alpha_jY_{i-j} +\sum_{m=1}^q \beta_m \varepsilon_{i-m},$$
with all roots of  the AR-polynomial 
being larger than one in absolute value and with the AR-and MA-polynomials having no common roots. If the innovations $(\varepsilon_i)_{i\in\N}$ are independent with an absolutely continuous distribution with respect to the Lebesgue measure, then $(Y_i)_{i\in\N}$ is $\alpha$-mixing with geometric rate as shown by Mokkadem \citep{Mokkadem.1988}. 
\item Let $(Y_i)_{i\in\N}$ be a strictly stationary GARCH(p,q)-process, i.e. a strictly stationary solution to the equations
$$Y_i=\sigma_i\varepsilon_i\quad \text{and} \quad\sigma_i^2=\alpha_0+ \sum_{j=1}^p\alpha_j Y_{i-j}^2+\sum_{m=1}^q \beta_m \sigma_{i-m}^2,$$
where $(\varepsilon_i)_{i\in \N}$ is an i.i.d. error sequence. Given $\E{\abs{\varepsilon_1}^r}<\infty$ for some $r\in(0,\infty)$ and given $\varepsilon_1$ is absolutely continuous with a Lebesgue density that is strictly positive in a neighbourhood of zero, $(Y_i)_{i\in\N}$ is $\alpha$-mixing with geometric rate (see, Lindner \citep{Lindner2009} and the references therein). 
\item Let $(Y_i)_{i\in\N}$ be a strictly stationary non-linear AR(1)-process, that is, let
$$Y_i=f(Y_{i-1})+\varepsilon_i,$$
where 
$(\varepsilon_i)_{i\in \N}$ is an i.i.d. 
error sequence and $f:\R\rightarrow \R$ is measurable. Assume the distribution of $\varepsilon_1$ is equivalent to the Lebesgue measure with $\E{\varepsilon_1}=0$ and $\E{\abs{\varepsilon_1}}<\infty$. 
Moreover, let there exist some $0\leq \lambda<1$ and some constant $C$ such that $\abs{f(x)}\leq \lambda
\abs{x}+C$.
Then $(Y_i)_{i\in\N}$ is $\alpha$-mixing with geometric rate (see, e.g., Theorem 3.2 in An and Huang \citep{An.1996} in combination with Theorem 3.7 in Bradley \citep{Bradley.2005}).
\end{enumerate}
}\end{example}

Coming back to model \eqref{Eq: Triangular array X}, we will subsequently assume that there exists a $\delta>0$ such that 
 \begin{align*}
\text{(A1)} & \quad \E{\abs{Y_1}^{2+\delta}}<\infty \\
 \text{(A2)} & \quad \sum_{k=1}^{\infty} \alpha_Y(k)^{\delta/(2+\delta)}<\infty
\end{align*}
hold,  in which case the long run variance 
\begin{equation*}
{\kappa}_Y^2:=\Var{Y_1} +2 \sum_{k=1}^\infty \cov{Y_1}{Y_{k+1}}
\end{equation*}
of the process $(Y_i)_{i\in \N}$ is finite. We will henceforth additionally assume that ${\kappa}_Y^2>0$. 
Moreover, we assume the mean $\mu:[0,1]\rightarrow \R$ to be a c\`adl\`ag function and $\sigma: [0,1]\rightarrow (0,\infty)$  to be piecewise Lipschitz-continuous (i.e. a Lipschitz-continuous function to which one added a piecewise constant function with a fixed number of $k_\sigma$ jumps of maximal absolute height $\Delta_\sigma$) with $\sigma(x)\geq \sigma_0$ for all $x\in [0,1]$ and some $\sigma_0>0$. Our intention is to test the hypothesis of a constant mean, i.e. $\mu=\mu_H$, against the alternative of a non-constant  c\`adl\`ag mean function. Under the hypothesis, it then holds $\E{X_i}=\mu_H$ and $\Var{X_i}=\sigma^2(i/n)$.

To construct our asymptotic test, we split the observations $X_1, \ldots, X_n$ into $b_n$ blocks of length $\ell_n$, where 
we denote $\ell_n=n^s$ and $b_n=n^{1-s}$. For simplicity, we assume that both $\ell_n$ and $b_n$ are integers. We will subsequently always assume that $s\in (0.5, 1)$, i.e. we assume the block length $\ell_n$ to increase faster than the number of blocks $b_n$. By comparing the local sample means $\hat{\mu}_j$ of the blocks $j=1, \ldots, b_n$ within a U-statistic, we arrive at the statistic $U(n)=\frac{1}{b_n(b_n-1)} \sum_{1\leq j\neq k \leq b_n} \left|\hat{\mu}_j -\hat{\mu}_k \right|$ from \eqref{teststatistic}, whose asymptotic behaviour is analysed in the next section.

\subsection{Limit theory for the U-statistic $U(n)$}
\label{Subsec: Limit theory for U(n)}
 Under mild assumptions, we obtain $L^2$-convergence of the statistic $U(n)$ towards a double Riemann-integral.
\begin{theorem}
\label{Thm: Behaviour under A}
Let the assumptions (A1) and (A2) hold. Then, it holds both under the hypothesis and under the alternative that
$$U(n) \overset{L^2}{\longrightarrow}  \int_0^1 \int_0^1 \abs{\mu(x)-\mu(y)} \mathrm{d}x \mathrm{d}y \quad \text{as } n\rightarrow \infty.$$
\end{theorem}

Clearly, the above double integral is zero if and only if $\mu$ is non-constant only on a Lebesgue-negligible set. Consequently, our test will achieve consistency against the large class of non-constant c\`adl\`ag mean functions. 
Under the hypothesis, $U(n)$ converges towards zero, such that we further investigate its behaviour for constant means. 
By using an additional scaling factor $\frac{\sqrt{\ell_n}}{\kappa_Y}$ comprising the block length $\ell_n$ and the long run variance $\kappa_Y^2$, we obtain a law of large numbers. 

\begin{theorem}
\label{THM: LLN}
Assume there is a $\rho>1$ and a $0< \delta \leq 1$ such that $\mathbb{E}\big(\abs{Y_1}^{2+\delta}\big)<\infty$ and $\alpha_Y(k)\leq Ck^{-\rho (2+\delta)(1+\delta)/\delta^2}$. Moreover, let  
$s>\big(1+\frac{\delta^2}{\rho (2+\delta)(1+\delta)}\big)/\big(2+\frac{\delta^2}{\rho (2+\delta)(1+\delta)}\big)$ for $\ell_n=n^s$. Then it holds under the hypothesis  
$$\frac{\sqrt{\ell_n}}{\kappa_Y}U(n)\overset{L^1}{\longrightarrow}  \int_0^1 \int_0^1 \sqrt{\sigma^2(x)+\sigma^2(y)}\mathrm{d}x\mathrm{d}y \cdot \sqrt{\frac{2}{\pi}} \quad \text{as } n\rightarrow \infty. $$
\end{theorem}

The limit in Theorem \ref{THM: LLN} can be used as centring term when deriving a central limit theorem for $U(n)$, for which we require an additional scaling $\sqrt{b_n}$ that depends on the number of blocks. 

\begin{theorem}
\label{Thm: Central limit theorem}
Assume there is a $\rho>1$ and a $0< \delta \leq 1$ such that $\mathbb{E}\big(\abs{Y_1}^{2+\delta}\big)<\infty$ and $\alpha_Y(k)\leq Ck^{-\rho (2+\delta)(1+\delta)/\delta^2}$. Let  $s>1/\big(1+\delta\frac{\rho-1}{\rho+1}\big)\vee \big(1+\frac{\delta^2}{\rho (2+\delta)(1+\delta)}\big)/\big(2+\frac{\delta^2}{\rho (2+\delta)(1+\delta)}\big)$ for $\ell_n=n^s$.  Then it holds under the hypothesis 
$$\sqrt{b_n}\bigg(\frac{\sqrt{\ell_n}}{\kappa_Y}U(n) -  \int_0^1 \int_0^1 \sqrt{\sigma^2(x)+\sigma^2(y)}\mathrm{d}x\mathrm{d}y \cdot \sqrt{\frac{2}{\pi}}\bigg) \distConv \NoD{0}{\psi^2} \quad \text{as } n\rightarrow \infty, $$
where  $$\psi^2=4\int_0^1 \mathbb{E}\bigg(\bigg |\int_0^1 \E{\abs{\sigma(x)Z-\sigma(y)Z'}|Z}-\E{\abs{\sigma(x)Z-\sigma(y)Z'}}\mathrm{d}y \bigg |^2\bigg) \mathrm{d}x$$ 
for two independent standard normal random variables $Z$ and $Z'$.
\end{theorem}

For certain process classes, the assumptions of Theorem \ref{Thm: Central limit theorem} can be simplified considerably.

\begin{example}{
If $(Y_i)_{i\in\N}$ is $\alpha$-mixing with a geometric rate, as is the case for the processes in Example \ref{Example: ARMA, GARCH}, Theorem \ref{Thm: Central limit theorem} only requires the existence of some $0<\delta\leq 1$ such that $\E{\abs{Y_1}^{2+\delta}}<\infty$ and $s>1/(1+\delta)$ since a geometric rate corresponds to the case $\rho\rightarrow \infty$. 
}\end{example}

Theorem \ref{Thm: Central limit theorem} is not only the key result to our asymptotic test in Section \ref{Subsec: The asymptotic test} 
but also gives several interesting insights into the underlying mathematical theory: Firstly, the double asymptotics of the statistic  $U(n)=\frac{1}{b_n(b_n-1)} \sum_{1\leq j\neq k \leq b_n} \left|\hat{\mu}_{j,n} -\hat{\mu}_{k,n} \right|$ become visible. They are due to the composed structure of $U(n)$, with the inner structure consisting of the triangular array of local sample means $\hat{\mu}_{j,n}=\frac{1}{\ell_n}\sum_{i=(j-1)\ell_n+1}^{j\ell_n} X_i$, $1\leq j\leq b_n$, and with the outer structure consisting of the U-statistic $\frac{1}{b_n(b_n-1)} \sum_{1\leq j\neq k\leq b_n} \abs{\nu_j-\nu_k}$ 
known as Gini's mean difference. Depending on the scaling, both components can separately exhibit a law of large numbers or a central limit theorem, where the latter requires the additional scaling factors $\sqrt{\ell_n}$ and $\sqrt{b_n}$ for the sample means and the U-statistic, respectively. Hence, Theorem \ref{THM: LLN} constitutes a law of large numbers for the U-statistic, with its arguments $\frac{\sqrt{\ell_n}}{\kappa_Y}\hat{\mu}_{j,n}$ converging towards normal laws, whereas Theorem \ref{Thm: Central limit theorem} constitutes a ``twofold'' central limit theorem.   
Secondly, the non-stationarity of the time series $(X_i)_{i\in\N}$ is reflected in Theorem \ref{Thm: Central limit theorem}  both in the centring term and in the limit variance $\psi^2$ via the double Riemann-integrals of the variance function $\sigma$. Most interestingly, the limit variance $\psi^2$ does not only depend on the Hoeffding-decomposition of the kernel $h(x,y)=\abs{x-y}$, but instead, the latter is interwoven with $\sigma$, hinting that the proofs will require an adaptation of U-statistic theory to triangular arrays under non-stationarity.

To practically employ Theorem \ref{Thm: Central limit theorem}, reliable estimators for the long run variance $\kappa_Y^2$, the variance function $\sigma$ and the parameter $\psi^2$ are essential.  

\subsection{Estimation of nuisance parameters}
\label{Subsec: Estimation of Parameters}
This section introduces estimators for the nuisance parameters in Theorem \ref{Thm: Central limit theorem}. The centring term as well as the variance $\psi^2$ of the limit distribution depend upon an integral of the variance function. Since $\sigma$ is usually unknown, we will approximate it with the help of the empirical block variances. Define the local sample variance of the $j$-th block as
\begin{align}
\label{Eq: sigmajhat Definition}
\hat{\sigma}^2_{j,n}:=& \frac{1}{\ell_n} \sum_{i=(j-1) \ell_n+1}^{j\ell_n}
\left(X_i - \hat{\mu}_j\right)^2
 =\frac{1}{\ell_n} \sum_{i=(j-1)\ell_n+1}^{j\ell_n} \bigg (X_i-\frac{1}{\ell_n} \sum_{r=(j-1)\ell_n+1}^{j\ell_n}X_r \bigg )^2.
\end{align}
In the following, we will replace integrals of the form $\int_0^1 \sigma(x)\mathrm{d}x$ by Riemann-type sums $\frac{1}{b_n} \sum_{j=1}^{b_n} \sigma\rbraces{j/b_n}$ 
 and argue that  the estimated block variances $\hat{\sigma}^2_{j,n}$ are asymptotically close to  the theoretical variances $\sigma\rbraces{j/b_n}$. 

\subsubsection{Estimation of the long run variance}
\label{Subsec: Estimation of LRV} 
The literature contains various approaches to estimate the long run variance $\kappa_Y^2$ from stationary observations $Y_1, \ldots, Y_n$. The main drawback of these estimators is their lack of consistency in the presence of a non-stationary mean, 
leading to an overestimation of $\kappa_Y$. In the light of Theorem \ref{Thm: Central limit theorem}, an overestimation of $\kappa_Y$ under the alternative will lower the value of the test statistic and will consequently result in a loss of the test's power. To mitigate this effect, we will follow the subsampling approach first introduced by Carlstein \citep{Carlstein.1986} and modify the estimator to allow for certain non-constant mean functions. 
We start by deriving the basic estimator for centred time series. It uses the relation 
$$\E{\abs{\frac{1}{\sqrt{n}}\sum_{i=1}^{n} Y_i}}\longrightarrow \kappa_Y \sqrt{\frac{2}{\pi}}$$
and estimates the expectation on the left hand side via subsampling: Dividing the observations into $\bn$ blocks of length $\elln$ with $\elln=n^q$  for some $q\in (0,1)$ and $\bn=n^{1-q}$ (both are assumed to be integers), replacing the expectation by the sample mean over the $\bn$ subsampling blocks, and rearranging leads to the estimator
\begin{equation*}
\hat{\kappa}_Y(n):=\frac{1}{\bn}\sqrt{\frac{\pi}{2}} \sum_{j=1}^{\bn} \frac{1}{\sqrt{\elln}}\abs{\sum_{i=(j-1)\elln+1}^{j\elln} Y_i}.
\end{equation*}
The consistency of such long run variance subsampling estimators has been established under several mixing conditions (see, e.g., Peligrad and Shao \citep{PeligradShao.1995},  Peligrad and Suresh \citep{PeligradSuresh.1995} or Dehling et al. \citep{Dehling.2013}). However, as we generally only observe data $X_1, \dots, X_n$ of type \eqref{Eq: Triangular array X}, the observations need to be centred in a suitable manner. 
Define the estimator
\begin{equation}
\label{eq: kappahat Xtilde}
\hat{\kappa}_{\tilde{X}}(n)
:=\sqrt{\frac{2c_0}{1+2c_0}}\frac{1}{\bn}\sqrt{\frac{\pi}{2}} \sum_{j=1}^{\bn} \frac{1}{\sqrt{\elln}}\abs{\sum_{i=(j-1)\elln+1}^{j\elln} X_i-\frac{1}{2c_0}\rbraces{\sum_{i=(j-1-c_0)\elln+1}^{(j-1)\elln}X_i+\sum_{i=j\elln+1}^{(j+c_0)\elln}X_i}}
\end{equation}
for some $c_0\in [1, \bn)$, where we set $X_i:=0$ whenever $i\not\in \{1, \ldots, n\}$ and adjust the weight $\frac{1}{2c_0}$ accordingly. Intuitively, the $j$-th subsampling block is centred by the sample mean of the observations in the $c_0$ blocks surrounding it to each side. One can show that given a sufficiently ``nice'' mean function, it holds $\hat{\kappa}_{\tilde{X}}(n) \rightarrow \int_0^1 \sigma(x) \mathrm{d}x \cdot \kappa_Y$ in $L^2$ (see, Proposition \ref{Prop: Consistency alpha-mix} in the appendix). Replacing the variance function $\sigma$ by the block estimates $\hat{\sigma}_{j,n}$ from \eqref{Eq: sigmajhat Definition} leads to an estimator for $\kappa_Y$ itself, 
\begin{equation}
\label{Eq: kappahat Definition}
\hat{\kappa}(n):= \frac{\hat{\kappa}_{\tilde{X}}(n)}{\frac{1}{b_n}\sum_{j=1}^{b_n}\hat{\sigma}_{j,n}}.
\end{equation}

\begin{proposition}
\label{Prop: convergence estimator kappahat kappaY under H}
Assume there is  a $0<\delta \leq 1$ and a $\rho>(\delta^2+4\delta)/(2+2\delta)\vee 1$ such that $\mathbb{E}\big(\abs{Y_1}^{2+\delta}\big)<\infty$ and $\alpha_Y(k)\leq Ck^{-\rho (2+\delta)(1+\delta)/\delta^2}$.  
Denote $\ell_n=n^s$ and $\elln=n^q$. 
\begin{enumerate}
\item Let $\mu$ be a constant mean function and let $\mathbb{E}\big(\abs{Y_1}^{4+2\delta}\big)<\infty$. Assume $s>0.5$, $q<s$ as well as $q>(1-s) \cdot \max\big(\delta^{-1}\frac{\rho+1}{\rho-1},\rbraces{1+\delta^2(\rho(1+\delta)(2+\delta))^{-1}}\big)$. 
Then it holds 
$$\sqrt{b_n}\rbraces{\hat{\kappa}(n)-\kappa_Y}\pConv 0 \quad \text{as } n \rightarrow \infty.$$
\item Let $\mu$ be a piecewise Lipschitz-continuous mean function.
Assume $q<2/3$ and let either $\mathbb{E}\big(\abs{Y_1}^{4+2\delta}\big)<\infty$ or $\rho>\frac{\delta(8+\delta)}{(2+\delta)(1+\delta)}$. 
Then it holds
$$\rbraces{\hat{\kappa}(n)-\kappa_Y}\pConv 0 \quad \text{as } n \rightarrow \infty.$$
\end{enumerate}
\end{proposition}

\begin{remark}
Despite the different conditions imposed on the parameters $s$ and $q$ in the first part of Proposition \ref{Prop: convergence estimator kappahat kappaY under H}, note that there always exists an admissible choice: One can find an admissible $q$ as long as 
$1>s>\max\big(1/(1+\delta\frac{\rho-1}{\rho+1}), 1/(1+\frac{(1+\delta)(2+\delta)\rho}{(1+\delta)(2+\delta)\rho+\delta^2}) \big)$ and since $\delta>0$ and $\rho>1$, the latter maximum is always strictly smaller than one. Moreover, note that the conditions imposed on $\rho$ are rather mild since $\frac{\delta^2+4\delta}{2+2\delta}\leq 1.25$ and $\frac{\delta(8+\delta)}{(2+\delta)(1+\delta)}\leq 1.5$ for all $\delta\in (0,1]$. 
\end{remark}

Under the hypothesis, we require the convergence in Proposition \ref{Prop: convergence estimator kappahat kappaY under H} to be faster than $\sqrt{b_n}$ in order to replace $\kappa_Y$ by $\hat{\kappa}$ in the central limit theorem \ref{Thm: Central limit theorem}. Under the alternative, it suffices to obtain divergence of $\hat{\kappa}$ at a rate slower than $\sqrt{\ell_n}$, although under the particular alternative of a piecewise Lipschitz-continuous mean function, we can still obtain convergence. Given a general c\`adl\`ag mean function, one can show (see, Lemma \ref{Lemma: convergence estimator kappahat kappaY under A} in the appendix) that 
$$\hat{\kappa}(n)= O_{\Pb}\left(\elln^{1/2}\right).$$
If we choose the subsampling block length $\elln$ to be of order $o(\ell_n)$, that is, if $q<s$,  the divergence of $\hat{\kappa}$ under the alternative is thus slow enough not to eliminate the divergence of the U-statistic $U(n)$ itself.

\subsubsection{Estimation of the centring term}
We estimate the centring term from the central limit theorem \ref{Thm: Central limit theorem} by  replacing the variance function $\sigma$ by the block sample variances $\hat{\sigma}_{j,n}$, $1\leq j\leq b_n$, from \eqref{Eq: sigmajhat Definition}. 
\begin{proposition}
\label{Prop: Convergence estimated centring term}
Let $\ell_n=n^s$ with $s>0.5$ and assume there exists a $\delta>0$ such that $\sum_{k=1}^{\infty}\alpha_Y(k)^{\delta/(2+\delta)}<\infty$ and $\mathbb{E}\big(\abs{Y_1}^{4+2\delta}\big)<\infty$. Then it holds under the hypothesis that 
$$\sqrt{b_n}\bigg |\frac{1}{b_n(b_n-1)}\sum_{1\leq j\neq k \leq b_n} \sqrt{\hat{\sigma}_{j,n}^2+\hat{\sigma}_{k,n}^2}-  \int_0^1 \int_0^1 \sqrt{\sigma^2(x)+\sigma^2(y)}\mathrm{d}x\mathrm{d}y\bigg |\overset{L^1}{\longrightarrow} \; 0 \quad \text{as } n \rightarrow \infty.$$
\end{proposition}

Once more, we require the convergence in Proposition \ref{Prop: Convergence estimated centring term} to be faster than $\sqrt{b_n}$, whereas under the alternative, we have to check that the estimated centring term does not diverge  
and obtain (see, Lemma \ref{Lemma: estimated centring term alternative} in the appendix) $$\frac{1}{b_n(b_n-1)}\sum_{1\leq j\neq k \leq b_n} \sqrt{\hat{\sigma}_{j,n}^2+\hat{\sigma}_{k,n}^2} \cdot \sqrt{\frac{2}{\pi}}=O_{\mathbb{P}}(1).$$

\subsubsection{Estimation of the limit variance} Lastly,  we estimate the variance $\psi^2$ of the limit distribution from Theorem \ref{Thm: Central limit theorem}.
Define the functions $f_n: \R_{\geq 0}^{b_n} \rightarrow \R_{\geq 0}$ via 
\begin{align*}
& f_n(x_1, \ldots, x_{b_n}) :=  \frac{1}{b_n} \sum_{j=1}^{b_n}\mathbb{E}\bigg (\bigg |\frac{1}{b_n-1} \sum_{\substack{1\leq k\leq b_n\\k\neq j}} \E{\abs{x_j Z-x_kZ'}\middle|Z}-\E{\abs{x_jZ-x_kZ'}}\bigg |^2\bigg )
\end{align*}
for two independent standard normal random variables $Z$ and $Z'$, and the estimator by 

\begin{equation}
\label{Eq: psinhat Definition}
\hat{\psi}_n^2:=4f_n(\hat{\sigma}_{1,n}, \ldots, \hat{\sigma}_{b_n,n}).
\end{equation}

\begin{proposition} \label{Prop: hatpsin psin approx H}
Let there exist a $0< \delta \leq 1$ and a $\rho> 1$ such that $\mathbb{E}\big(\abs{Y_1}^{2+\delta}\big)<\infty$ and $\alpha_Y(k)\leq Ck^{-\rho (2+\delta)(1+\delta)/\delta^2}$. Moreover, let either $\rho>\frac{\delta(8+\delta)}{(2+\delta)(1+\delta)}$ or  $\mathbb{E}\big(\abs{Y_1}^{4+2\delta}\big)<\infty$. Then it holds under the hypothesis that
$$\abs{\hat{\psi}_n^2-\psi^2}\overset{L^1}{\longrightarrow} \; 0 \quad \text{as } n \rightarrow \infty. $$
\end{proposition}

Under the alternative, it holds $ \hat{\psi}_n=O_{\mathbb{P}}(1)$ (see, Lemma \ref{Lemma: hatpsin under A} in the appendix).

\begin{remark} The estimator $\hat{\psi}_n$ solely depends upon the estimated sample variances $\hat{\sigma}_{j,n}$, $1\leq j\leq b_n$ and on two independent standard normal laws. In principle, it could thus be calculated. For a constant variance function, $\sigma(x)=\sigma$, the exact value of the limit variance $\psi^2$ has been explicitly derived as $\psi^2=\sigma^2 \rbraces{\frac{4}{3}+\frac{8}{\pi}(\sqrt{3}-2)}$ (see, Gerstenberger and Vogel \citep{Gerstenberger.2015}).  However, the derivation is rather cumbersome, and it is unclear whether it is possible to simplify $\psi^2$ for non-constant variance functions, such that we employ a Monte Carlo estimator in Section \ref{Sec: Simulations}. 
\end{remark}

\subsection{The asymptotic test}
\label{Subsec: The asymptotic test}
This section presents the actual test for constancy of the mean by combining the results of the former two sections. 
\begin{theorem}
\label{Cor: Asymptotic test H}
Assume there is a $0< \delta \leq 1$ and a  $\rho>(\delta^2+4\delta)/(2+2\delta)\vee 1$ such that $\mathbb{E}\big(\abs{Y_1}^{4+2\delta}\big)<\infty$ and $\alpha_Y(k)\leq Ck^{-\rho (2+\delta)(1+\delta)/\delta^2}$. For $\ell_n=n^s$ and $\elln=n^q$, let  $s>1/(1+\delta\frac{\rho-1}{\rho+1})\vee (1+\frac{\delta^2}{\rho (2+\delta)(1+\delta)})/(2+\frac{\delta^2}{\rho (2+\delta)(1+\delta)})$ as well as $s>q$ and $q>(1-s) \cdot\max \big(\delta^{-1}\frac{\rho+1}{\rho-1}, \rbraces{1+\delta^2(\rho(1+\delta)(2+\delta))^{-1}}\big)$. Then it holds under the hypothesis  
$$\sqrt{b_n}\frac{1}{\hat{\psi}_n}\bigg(\frac{\sqrt{\ell_n}}{\hat{\kappa}(n)}U(n) -  \frac{1}{b_n(b_n-1)}\sum_{1\leq j\neq k \leq b_n} \sqrt{\hat{\sigma}_{j,n}^2+\hat{\sigma}_{k,n}^2} \sqrt{\frac{2}{\pi}}\bigg) \distConv \NoD{0}{1} \quad \text{as } n\rightarrow \infty, $$
where the estimators $\hat{\sigma}^2_{j,n}$, $\hat{\kappa}(n)$ and $\hat{\psi}_n$ are defined by \eqref{Eq: sigmajhat Definition}, \eqref{Eq: kappahat Definition} and \eqref{Eq: psinhat Definition}, respectively.
\end{theorem}

 \begin{example} 
Recall that the processes from Example \ref{Example: ARMA, GARCH} correspond to the border case $\rho\rightarrow \infty$. The above conditions on the different block lengths $\ell_n=n^s$ and $\elln=n^q$ then read $s>1/(1+\delta)$, $q<s$ and $q>(1-s)/\delta$. 
\end{example}

Under the alternative, the test statistic 
diverges at a rate that depends on the form of the mean function.

\begin{theorem}
\label{Cor: Asymptotic test A}
 Assume there is  a $0< \delta \leq 1$ and a $\rho>(\delta^2+4\delta)/(2+2\delta)\vee 1$ such that $\mathbb{E}\big(\abs{Y_1}^{2+\delta}\big)<\infty$ and $\alpha_Y(k)\leq Ck^{-\rho (2+\delta)(1+\delta)/\delta^2}$.  Additionally, assume that  either $\mathbb{E}\big(\abs{Y_1}^{4+2\delta}\big)<\infty$ or $\rho>\frac{\delta(8+\delta)}{(2+\delta)(1+\delta)}$.  
\begin{enumerate}
\item Let $\mu$ be a piecewise Lipschitz-continuous mean function and $\elln=n^q$ with $q<2/3$. Then, 
$$\sqrt{n}\bigg(\sqrt{b_n}\frac{1}{\hat{\psi}_n}\bigg(\frac{\sqrt{\ell_n}}{\hat{\kappa}(n)}U(n) -  \frac{1}{b_n(b_n-1)}\sum_{1\leq j\neq k \leq b_n} \sqrt{\hat{\sigma}_{j,n}^2+\hat{\sigma}_{k,n}^2} \sqrt{\frac{2}{\pi}}\bigg)\bigg)^{-1}= O_{\Pb}\rbraces{1}.$$
\item Let $\mu$ be a general c\`adl\`ag mean function. Then it holds
$$\sqrt{\bn}\bigg(\sqrt{b_n}\frac{1}{\hat{\psi}_n}\bigg(\frac{\sqrt{\ell_n}}{\hat{\kappa}(n)}U(n) -  \frac{1}{b_n(b_n-1)}\sum_{1\leq j\neq k \leq b_n} \sqrt{\hat{\sigma}_{j,n}^2+\hat{\sigma}_{k,n}^2} \sqrt{\frac{2}{\pi}}\bigg)\bigg)^{-1}=O_{\Pb}\rbraces{1}.$$
\end{enumerate}
\end{theorem}

If previous knowledge suggests a stationary variance, that is, if $\sigma(x)=\sigma$ for all $x\in [0,1]$, the assumptions of Theorem \ref{Cor: Asymptotic test H} can be relaxed. Most importantly, since we  have to estimate neither the centring term nor the limit variance, the moment assumptions can be significantly weakened. 
\begin{theorem}
\label{Cor: Asymptotic test, sigma=c}
Assume the variance function $\sigma$ is constant. Let there be a~$0<~\delta~\leq 1$ and a $\rho>(\delta^2+4\delta)/(2+2\delta)\vee 1$ such that $\mathbb{E}\big(\abs{Y_1}^{2+\delta}\big)<\infty$ and $\alpha_Y(k)\leq Ck^{-\rho (2+\delta)(1+\delta)/\delta^2}$. For $\ell_n=n^s$ and $\elln=n^q$, let $s>1/(1+\delta\frac{\rho-1}{\rho+1})\vee (1+\frac{\delta^2}{\rho (2+\delta)(1+\delta)})/(2+\frac{\delta^2}{\rho (2+\delta)(1+\delta)})$, $q<s$ and $q>(1-s)\cdot \max\big(\delta^{-1}\frac{\rho+1}{\rho-1}, \rbraces{1+\delta^2(\rho(1+\delta)(2+\delta))^{-1}}\big)$. Then it holds under the hypothesis 
$$\sqrt{b_n}\rbraces{\frac{\sqrt{\ell_n}}{\hat{\kappa}_{\tilde{X}}(n)}U(n)-\frac{2}{\sqrt{\pi}}} \distConv \NoD{0}{\frac{4}{3}+\frac{8}{\pi}\rbraces{\sqrt{3}-2}},$$
 where the estimator $\hat{\kappa}_{\tilde{X}}(n)$ is defined in \eqref{eq: kappahat Xtilde}.
\end{theorem}

Theorem \ref{Cor: Asymptotic test H} (or Theorem \ref{Cor: Asymptotic test, sigma=c} in case of homoscedasticity) can now be applied to test for a stationary mean.  Given some data $x_1, \ldots, x_n$, the value of the test statistic  
 $\sqrt{b_n}\,\hat{\psi}_n^{-1}\rbraces{\frac{\sqrt{\ell_n}}{\hat{\kappa}(n)}U(n) -  \sqrt{{2}/{\pi}} \frac{1}{b_n(b_n-1)}\sum_{1\leq j\neq k \leq b_n} \sqrt{\hat{\sigma}_{j,n}^2+\hat{\sigma}_{k,n}^2} }$ can be explicitly computed and compared to the $(1-\alpha)$-quantile of the standard normal limit distribution. We reject the hypothesis of a constant mean if the value of the test statistic is larger than the respective quantile since Theorem \ref{Cor: Asymptotic test A} ensures the divergence of the test statistic under the alternative.

\section{Outline and main ideas of the proofs}
\label{Sec: Outline of Proofs}
This section gives an overview over the line of proof for the main theoretical results in Section \ref{Subsec: Limit theory for U(n)}. The proofs concerning the estimators in Section \ref{Subsec: Estimation of Parameters} as well as all technical details are deferred to an appendix. 

The proofs of Section \ref{Subsec: Limit theory for U(n)} consist of several approximation steps that reduce the statistic 
\begin{equation*}
  U(n)=\frac{1}{b_n(b_n-1)} \sum_{1\leq j\neq k \leq b_n} \left|\hat{\mu}_j -\hat{\mu}_k \right|
\end{equation*}
to a more manageable form, whose asymptotic behaviour is easier to analyse.  
In a first approximation step, we replace the variance of each observation $X_i=\mu(i/n)+\sigma(i/n)Y_i$, $i\in \{(j-1)\ell_n+1, \ldots, j\ell_n\}$, by the variance $\sigma(j\ell_n/n)=\sigma(j/b_n)$ of the last observation within the respective block $j$. The resulting U-statistic $U_1(n)$ is asymptotically close to a  Riemann-type double sum converging towards $\int_0^1\int_0^1 \abs{\mu(x)-\mu(y)}\mathrm{d}x\mathrm{d}y$. While this convergence holds for an arbitrary c\`adl\`ag mean function $\mu$, we require a second approximation step under the hypothesis of a constant mean to derive the limit theory presented in Theorems \ref{THM: LLN} and \ref{Thm: Central limit theorem}. Using a coupling technique for $\alpha$-mixing processes (see, Peligrad \citep{Peligrad.2002}), we replace the arguments $\frac{1}{\ell_n}\sum_{i=(j-1)\ell_n+1}^{j\ell_n} \sigma(j/b_n)Y_i+\mu_H$ of the U-statistic $U_1(n)$ by independent ones. 
The limit behaviour of the resulting statistic $U_2(n)$ can be derived by an extension of classical U-statistics theory to row-wise non-stationary triangular arrays. 

\subsection{A first approximation} By the piecewise Lipschitz-property of $\sigma$, the variances of the observations within each block $j$ are close to one another,  
$\sigma(i/n)\approx \sigma(j\ell_n/n)=\sigma(j/b_n)$ for $i\in \{(j-1)\ell_n+1, \ldots, j\ell_n\}, $ for all but finitely many blocks. Replacing the variance  of each individual observation by the variance of the last observation within the respective block, we arrive at the U-statistic
\begin{align*}
 U_1(n):=
 \frac{1}{b_n(b_n-1)}\sum_{1\leq j \neq k \leq b_n} & \bigg | \frac{1}{\ell_n}\sum_{i=(j-1)\ell_n+1}^{j\ell_n} \rbraces{\sigma\rbraces{\frac{j}{b_n}} Y_i + \mu\rbraces{\frac{i}{n}}} \\
&  - \frac{1}{\ell_n} \sum_{i=(k-1)\ell_n+1}^{k\ell_n} \rbraces{\sigma\rbraces{\frac{k}{b_n}} Y_i + \mu\rbraces{\frac{i}{n}}}\bigg |.
\end{align*}
We prove the statistic $U_1(n)$ to be sufficiently close to the original statistic $U(n)$, even if we require their difference to diminish faster than $\sqrt{n}$, which is the rate needed to replace $U$ by $U_1$ in the central limit theorem \ref{Thm: Central limit theorem}. 

\begin{proposition}{
\label{Prop: Approximation U U1}
Let the assumptions (A1) and (A2) be fulfilled. Then it holds 
$$\sqrt{n} \abs{U(n)-U_1(n)}\overset{L^2}{\longrightarrow} \;0 \quad \text{as } n \rightarrow \infty. $$ 
}\end{proposition}

Proposition \ref{Prop: Approximation U U1} holds regardless of the particular form of the mean function, such that both under the hypothesis and under the alternative it will subsequently suffice to investigate the limit behaviour of the U-statistic $U_1$. 

\subsection{Outline of the proof of Theorem \ref{Thm: Behaviour under A}} 
Given $\mu$ is an arbitrary c\`adl\`ag function, the arguments within the U-statistic $U_1(n)$ approximately behave like 
\begin{align*}
 \frac{1}{\ell_n}\sigma\rbraces{\frac{j}{b_n}} \sum_{i=(j-1)\ell_n+1}^{j\ell_n}Y_i+ \frac{1}{\ell_n}\sum_{i=(j-1)\ell_n+1}^{j\ell_n}\mu\rbraces{\frac{i}{n}} 
 \approx  \frac{1}{\ell_n}\sum_{i=(j-1)\ell_n+1}^{j\ell_n}\mu\rbraces{\frac{i}{n}} \approx \mu\rbraces{\frac{j}{b_n}} 
\end{align*}
for all but finitely many blocks $j$, $1\leq j\leq b_n$, by the law of large numbers for the centred, $\alpha$-mixing time series $(Y_i)_{i\in\N}$ and by the c\`adl\`ag property of $\mu$. Hence, 
\begin{align*}
U_1(n)\approx  \frac{1}{b_n(b_n-1)}\sum_{1\leq j \neq k \leq b_n} \abs{ \mu\rbraces{\frac{j}{b_n}} -  \mu\rbraces{\frac{k}{b_n}} }.
\end{align*}
The latter is a Riemann-type sum converging towards the desired double Riemann-integral $\int_0^1 \int_0^1 |\mu(x) -\mu(y)| \mathrm{d}x\mathrm{d}y.$ A more detailed proof is included in the appendix. 

\subsection{A second approximation}
For the remainder of this section, we assume that the hypothesis of a constant mean function holds. Without loss of generality, set $\mu=0$. 
To derive the  limit theory under the hypothesis, a second approximation is needed to replace the arguments within the U-statistic $U_1(n)$ by independent ones. To do so, we make use of a coupling technique for $\alpha$-mixing random variables introduced by Peligrad \citep{Peligrad.2002}. 
 More precisely, we replace the blocks $(Y_{(j-1)\ell_n+1}, \ldots, Y_{j\ell_n})$, $1\leq j\leq b_n$, by slightly shortened blocks $(Y_{(j-1)\ell_n+1}, \ldots, Y_{j\ell_n-m_n})$. The remainder  blocks $(Y_{j\ell_n-m_n+1}, \ldots,  Y_{j\ell_n})$, $1\leq j\leq b_n$ of length $m_n=o(\ell_n)$ function as a separation  to weaken the dependence between the bigger blocks as $m_n\rightarrow \infty$ for $n\rightarrow \infty$, but they are asymptotically negligible themselves. We denote the resulting big block sums by 
$$\tilde{S}_{nj}:=\frac{1}{\ell_n}\sum_{i=(j-1)\ell_n+1}^{j\ell_n-m_n}Y_i.$$
According to Theorem 2 in Peligrad \citep{Peligrad.2002}, there exists (on a possibly richer probability space) a sequence of independent random variables $\tilde{S}_{nj}'$, $1\leq j \leq b_n$, such that for each $j$, $\tilde{S}_{nj}$ and $\tilde{S}_{nj}'$ have the same distribution and 
\begin{equation}
\label{Eq: Peligrad bound for expectation}
\mathbb{E}\big(\big |\tilde{S}_{nj}-\tilde{S}_{nj}'\big |\big)\leq 4 \int_0^{\alpha_Y(m_n)} Q_{\abs{\tilde{S}_{nj}}}(u)\mathrm{d}u,
\end{equation}
where $Q_{\abs{\tilde{S}_{nj}}}(u):=\inf \big\{x\in \R: \Pb\big(\big |\tilde{S}_{nj}\big |>x\big)\leq u\big\}$. This yields a U-statistic with independent arguments
$$U_2(n):= \frac{1}{b_n(b_n-1)} \sum_{1\leq j\neq k\leq b_n} \abs{\sigma(j/b_n)\tilde{S}_{nj}'-\sigma(k/b_n)\tilde{S}_{nk}'}$$
that is asymptotically still sufficiently close to $U_1(n)$.

\begin{proposition}{
\label{Prop: Approx U1 U2}
Let the assumptions (A1) and (A2) be fulfilled. Let $\ell_n=n^s$ with $s>0.5$ and assume that there exists a sequence $m_n\rightarrow \infty$ as $n\rightarrow \infty$ such that $m_n=o(n^{2s-1})$ and $b_n\alpha_Y(m_n)\rightarrow 0$. Then it holds under the hypothesis
$$\sqrt{n} \abs{U_1(n)-U_2(n)}\overset{L^1}{\longrightarrow} \; 0 \quad \text{as } n\rightarrow\infty.$$
}\end{proposition}

\begin{remark} 
The exact choice of the sequence $m_n\rightarrow \infty$ as $n\rightarrow \infty$ is of no importance as long as the conditions $m_n=o(n^{2s-1})$ and $b_n\alpha_Y(m_n)\rightarrow 0$ are met. In case of polynomially decreasing mixing coefficients, i.e. $\alpha_Y(k)\leq C k^{-\rho (2+\delta)(1+\delta)/\delta^2}$, the existence of such a sequence  $(m_n)_{n\in\N}$ is guaranteed if we choose  $s>(1+\frac{\delta^2}{\rho(2+\delta)(1+\delta)})/(2+\frac{\delta^2}{\rho(2+\delta)(1+\delta)})$ which simplifies to $s>0.5$ if the mixing coefficients are decreasing at a geometric rate, that is, for $\rho\rightarrow \infty$. 
\end{remark}

$U_2(n)$ will be subsequently analysed by an adaptation of the classic U-statistics theory under independence to row-wise non-stationary triangular arrays.

\subsection{Proof of Theorem \ref{THM: LLN}} Due to the former approximation steps, it suffices to establish a law of large numbers for the U-statistic $U_2(n)$, whose arguments $\tilde{S}_{nj}'$  are independent with $\tilde{S}_{nj}'\overset{\mathcal{D}}{=}\tilde{S}_{nj}=\frac{1}{\ell_n} \sum_{i=(j-1)\ell_n+1}^{j\ell_n-m_n}Y_i.$ 
By the central limit theorem for $\alpha$-mixing random variables,
\begin{align*}
\frac{\sqrt{\ell_n}}{\kappa_Y} \tilde{S}_{nj}' \distConv Z_j
\end{align*}
for some standard normally distributed random variable $Z_j$. Since the $\tilde{S}_{nj}'$'s are independent, so are the $Z_j$'s, and roughly speaking, we intend to asymptotically replace the arguments $ \tilde{S}_{nj}'$ of $U_2(n)$ by $\frac{\kappa_Y}{\sqrt{\ell_n}} Z_j.$

\begin{proposition}
\label{Prop: LLN U2}
 Assume there is a $\rho>1$ and a $0<\delta \leq 1$ such that $\mathbb{E}\big(\abs{Y_1}^{2+\delta}\big)<\infty$ and $\alpha_Y(k)\leq Ck^{-\rho (2+\delta)(1+\delta)/\delta^2}$. 
 Then it holds under the hypothesis that 
$$\frac{\sqrt{\ell_n}}{\kappa_Y}U_2(n)\overset{L^2}{\longrightarrow}  \int_0^1 \int_0^1 \sqrt{\sigma^2(x)+\sigma^2(y)}\mathrm{d}x\mathrm{d}y \cdot \sqrt{\frac{2}{\pi}} \quad \text{as } n\rightarrow \infty. $$

\end{proposition}

\begin{proof} The proof of Proposition \ref{Prop: LLN U2} consists of two steps. First, we show that the expectation $\mathbb{E}\big(\frac{\sqrt{\ell_n}}{\kappa_Y}U_2(n)\big)$ converges towards the desired quantity $\int_0^1 \int_0^1 \sqrt{\sigma^2(x)+\sigma^2(y)}\mathrm{d}x\mathrm{d}y \cdot \sqrt{\frac{2}{\pi}}$,  and afterwards, that the variance of $\sqrt{\ell_n}U_2(n)/\kappa_Y$ converges towards zero. 
We obtain
\begin{align*}
\E{\frac{\sqrt{\ell_n}}{\kappa_Y}U_2(n)} = & \frac{\sqrt{\ell_n}}{\kappa_Y} \frac{1}{b_n(b_n-1)}\sum_{1\leq j \neq k \leq b_n} \E{\abs{\sigma(j/b_n)\tilde{S}_{nj}'-\sigma(k/b_n)\tilde{S}_{nk}'}}\\
\approx  & \frac{1}{b_n(b_n-1)}\sum_{1\leq j \neq k \leq b_n} \E{\abs{\sigma(j/b_n)Z_j-\sigma(k/b_n)Z_k}} \\
= & \frac{1}{b_n(b_n-1)}\sum_{1\leq j \neq k \leq b_n} \sqrt{\sigma^2(j/b_n)+\sigma^2(k/b_n)} \sqrt{\frac{2}{\pi}}.
\end{align*}
The double sum  is a Riemann-type  approximation converging towards the desired integral $ \int_0^1 \int_0^1 \sqrt{\sigma^2(x)+\sigma^2(y)}\mathrm{d}x\mathrm{d}y$. The convergence of the expectation $\mathbb{E}\big({\frac{\sqrt{\ell_n}}{\kappa_Y}U_2(n)}\big)$ is a weaker version of Proposition \ref{Prop: Replacing centring term} below, whose detailed proof is given in the appendix. 

Regarding the variance, we have

\begin{align*}
& \Var{\frac{\sqrt{\ell_n}}{\kappa_Y}U_2(n)} 
=  \frac{2}{b_n^2(b_n-1)^2} \sum_{1\leq j_1\neq k_1\leq b_n} \Var{\frac{\sqrt{\ell_n}}{\kappa_Y}\abs{\sigma\rbraces{\frac{j_1}{b_n}}\tilde{S}_{nj_1}'-\sigma\rbraces{\frac{k_1}{b_n}}\tilde{S}_{nk_1}'}}\\
 & + \frac{4}{b_n^2(b_n-1)^2} \sum_{1\leq j_1\neq k_1\leq b_n}\sum_{\substack{1\leq k_2\leq b_n,\\ k_2\notin \{j_1, k_1\}}} \mathrm{Cov}\left(\frac{\sqrt{\ell_n}}{\kappa_Y}\abs{\sigma\rbraces{\frac{j_1}{b_n}}\tilde{S}_{nj_1}'-\sigma\rbraces{\frac{k_1}{b_n}}\tilde{S}_{nk_1}'},\right. \\
& \hspace{6.58cm} \left. \frac{\sqrt{\ell_n}}{\kappa_Y}\abs{\sigma\rbraces{\frac{j_1}{b_n}}\tilde{S}_{nj_1}'-\sigma\rbraces{\frac{k_2}{b_n}}\tilde{S}_{nk_2}'}\right),
\end{align*}
which is (after some calculation) shown to be of order $O(1/b_n)$. For the specifics, we refer to the detailed proof in the appendix. 
\end{proof}

Theorem \ref{THM: LLN} now follows from a combination of  the former approximation results, Propositions \ref{Prop: Approximation U U1} and \ref{Prop: Approx U1 U2}, together with Proposition \ref{Prop: LLN U2}.

\subsection{Proof of Theorem \ref{Thm: Central limit theorem}}
Given an appropriate centring and scaling of $U_2(n)$, we can additionally obtain a central limit theorem for the U-statistic itself.

\begin{proposition}
\label{Prop: CLT U2}
Assume there is a $\rho>1$ and a $0<\delta \leq 1$ such that $\mathbb{E}\big(\abs{Y_1}^{2+\delta}\big)<\infty$ and $\alpha_Y(k)\leq Ck^{-\rho (2+\delta)(1+\delta)/\delta^2}$. Let $\ell_n=n^s$ with $s>0.5$. 
Then it holds under the hypothesis that 
$$\sqrt{n}\frac{1}{\kappa_Y}\rbraces{U_2(n) -\E{U_2(n)} } \distConv \NoD{0}{\psi^2} \quad \text{as } n\rightarrow \infty, $$
where  $$\psi^2=4\int_0^1 \mathbb{E}\bigg(\bigg |\int_0^1 \E{\abs{\sigma(x)Z-\sigma(y)Z'}|Z}-\E{\abs{\sigma(x)Z-\sigma(y)Z'}}\mathrm{d}y \bigg |^2\bigg) \mathrm{d}x $$ 
for two independent standard normal random variables $Z$ and $Z'$.
\end{proposition}

\begin{proof}
Note that
\begin{align*}
\frac{\sqrt{\ell_n}}{\kappa_Y} U_2(n)= \frac{1}{b_n(b_n-1)} \sum_{1\leq j\neq k\leq b_n} \abs{\frac{\sqrt{\ell_n}}{\kappa_Y}\sigma\rbraces{\frac{j}{b_n}}\tilde{S}_{nj}'-\frac{\sqrt{\ell_n}}{\kappa_Y}\sigma\rbraces{\frac{k}{b_n}}\tilde{S}_{nk}'}
\end{align*}
is a U-statistic of a row-wise non-stationary but  independent triangular array. We intend to split $U_2(n)$  up via the  Hoeffding-decomposition, which is a frequently used tool in U-statistics theory. We decompose the kernel $h(x,y)=\abs{x-y}$ by defining
\begin{align*}
\theta_{jk}^{(n)}=\theta_{kj}^{(n)}&:=\E{h\rbraces{\frac{\sqrt{\ell_n}}{\kappa_Y}\sigma\rbraces{\frac{j}{b_n}}\tilde{S}_{nj}', \frac{\sqrt{\ell_n}}{\kappa_Y}\sigma\rbraces{\frac{k}{b_n}}\tilde{S}_{nk}'}},\\
h_{1,jk}^{(n)}(x)&:=\E{h\rbraces{x, \frac{\sqrt{\ell_n}}{\kappa_Y}\sigma\rbraces{\frac{k}{b_n}}\tilde{S}_{nk}'}}-\theta_{jk}^{(n)},\\
\overline{h_{1,j}^{(n)}}\rbraces{x}&:=\frac{1}{b_n-1} \sum_{\substack{1\leq k\leq b_n\\k\neq j}} h_{1,jk}^{(n)}\rbraces{x},\\
h_{2,jk}^{(n)}(x,y)&:=h(x,y)-h_{1,jk}^{(n)}(x)-h_{1,kj}^{(n)}(y)-\theta_{jk}^{(n)}.
\end{align*}
Note that since the arguments of the U-statistic stem from a non-stationary triangular array, the Hoeffding-decomposition depends both on the sample size $n$ and on the indices $j$ and $k$.
Applying the Hoeffding-decomposition to $U_2$ yields
\begin{align*}
&\sqrt{b_n}\rbraces{\frac{\sqrt{\ell_n}}{\kappa_Y}U_2(n) - \frac{\sqrt{\ell_n}}{\kappa_Y}\E{U_2(n)} } \\
=&\frac{\sqrt{b_n}}{b_n(b_n-1)}   \sum_{1\leq j\neq k\leq b_n}\rbraces{ \abs{\frac{\sqrt{\ell_n}}{\kappa_Y}\sigma\rbraces{\frac{j}{b_n}}\tilde{S}_{nj}'-\frac{\sqrt{\ell_n}}{\kappa_Y}\sigma\rbraces{\frac{k}{b_n}}\tilde{S}_{nk}'}-  \theta_{jk}^{(n)}}\\
= & \frac{2}{\sqrt{b_n}} \sum_{1\leq j\leq b_n} \overline{h_{1,j}^{(n)}}\rbraces{\frac{\sqrt{\ell_n}}{\kappa_Y}\sigma\rbraces{\frac{j}{b_n}}\tilde{S}_{nj}'} +  \frac{\sqrt{b_n}}{b_n(b_n-1)}   \sum_{1\leq j\neq k\leq b_n}h_{2,jk}^{(n)}\rbraces{\frac{\sqrt{\ell_n}}{\kappa_Y}\sigma\rbraces{\frac{j}{b_n}}\tilde{S}_{nj}',\frac{\sqrt{\ell_n}}{\kappa_Y}\sigma\rbraces{\frac{k}{b_n}}\tilde{S}_{nk}'}.
\end{align*}
We will call the first term the linear and the second term the degenerate part of the Hoeffding-decomposition of $U_2(n)$. By independence of the $\tilde{S}_{nj}'$'s and by an application of Fubini's Theorem, 
\begin{align*}
\E{\overline{h_{1,j}^{(n)}}\rbraces{\frac{\sqrt{\ell_n}}{\kappa_Y}\sigma\rbraces{\frac{j}{b_n}}\tilde{S}_{nj}'} }=0
\end{align*}
as well as
\begin{align*}
\E{h_{2,jk}^{(n)}\rbraces{x,\frac{\sqrt{\ell_n}}{\kappa_Y}\sigma\rbraces{\frac{k}{b_n}}\tilde{S}_{nk}'}}=\E{h_{2,jk}^{(n)}\rbraces{\frac{\sqrt{\ell_n}}{\kappa_Y}\sigma\rbraces{\frac{j}{b_n}}\tilde{S}_{nj}',y}}=0
\end{align*}
for any $x,y\in\R$ and all $j,k\in \{1, \ldots, b_n\}$ with $j\neq k$. 
The linear part thus sums up the entries from a row-wise independent, centred, non-stationary triangular array and Lyapunov's central limit theorem yields the convergence in distribution towards the desired normal law $ \NoD{0}{\psi^2}$, whose limit variance is given by 
$$\psi^2= 4\limn \Var{\frac{1}{\sqrt{b_n}} \sum_{1\leq j\leq b_n} \overline{h_{1,j}^{(n)}}\rbraces{\frac{\sqrt{\ell_n}}{\kappa_Y}\sigma\rbraces{\frac{j}{b_n}}\tilde{S}_{nj}'}}.$$
The degenerate part of the Hoeffding-decomposition is a degenerate U-statistic with pairwise uncorrelated summands. Its variance is (after some calculation) consequently  given by 
\begin{align*}
\frac{1}{b_n(b_n-1)^2}   \sum_{1\leq j\neq k\leq b_n}\Var{h_{2,jk}^{(n)}\rbraces{\frac{\sqrt{\ell_n}}{\kappa_Y}\sigma\rbraces{\frac{j}{b_n}}\tilde{S}_{nj}',\frac{\sqrt{\ell_n}}{\kappa_Y}\sigma\rbraces{\frac{k}{b_n}}\tilde{S}_{nk}'}}=O\rbraces{\frac{1}{b_n-1}}.
\end{align*}
Hence, the degenerate part converges in $L^2$ towards zero and an application of Slutzky's lemma concludes the proof. Once more, a detailed proof is deferred to the appendix.
\end{proof}

So far, we used the mean $\frac{\sqrt{\ell_n}}{\kappa_Y}\E{U_2(n)} $ as the centring term within the central limit theorem. However, for a practical application of the test, we need to replace it by its limit $ \int_0^1 \int_0^1 \sqrt{\sigma^2(x)+\sigma^2(y)}\mathrm{d}x\mathrm{d}y\cdot \sqrt{2/\pi}$. The latter depends upon the data solely via the variance function. 

\begin{proposition}
\label{Prop: Replacing centring term}
Assume there is a $\rho>1$ and a $0<\delta \leq 1$ such that $\mathbb{E}\big({\abs{Y_1}^{2+\delta}}\big)<\infty$ and $\alpha_Y(k)\leq Ck^{-\rho (2+\delta)(1+\delta)/\delta^2}$. Let $\ell_n=n^s$ with $s>1/(1+\delta\frac{\rho-1}{\rho+1})$.  Then it holds under the hypothesis
$$\sqrt{b_n} \,\bigg |\mathbb{E}\bigg({\frac{\sqrt{\ell_n}}{\kappa_Y}U_2(n)}\bigg)-  \int_0^1 \int_0^1 \sqrt{\sigma^2(x)+\sigma^2(y)}\mathrm{d}x\mathrm{d}y \cdot \sqrt{\frac{2}{\pi}}\bigg |\longrightarrow 0 \quad \text{as } n \rightarrow \infty.$$
\end{proposition}

Theorem \ref{Thm: Central limit theorem} is now a consequence of Propositions \ref{Prop: Approximation U U1} and \ref{Prop: Approx U1 U2}, yielding 
$$\sqrt{n}\rbraces{U(n)-U_2(n)}\overset{L^1}{\longrightarrow} \; 0,$$
of the central limit theorem for $U_2(n) $ in Proposition \ref{Prop: CLT U2}, of Proposition \ref{Prop: Replacing centring term}, and of Slutzky's lemma.

\section{Extensions}
\label{Sec: Extensions}
\subsection{Extensions of our testing procedure}
Noticing that our test statistic constitutes, from a broader viewpoint, a U-statistic of a triangular array, two natural extensions come to mind. Firstly, one could employ other suitable symmetric kernel functions $h:\R\times \R\rightarrow \R$ and replace the U-statistic $U(n)$ in \eqref{teststatistic} by a more general one
$$U_h(n)=\frac{1}{b_n(b_n-1)}\sum_{1\leq j\neq k\leq b_n}h\rbraces{\hat{\mu}_j, \hat{\mu}_k}.$$
If the kernel function $h$ is Lipschitz-continuous, we are optimistic that the line of proof in Section \ref{Sec: Outline of Proofs} is transferable. 
Secondly, our procedure could possibly be improved by considering more robust estimators of location, such as the block median, instead of the local sample means. More generally, the triangular array within the U-statistic could even consist of local estimators of other characteristics, such as the skewness or kurtosis, thereby providing a test for their stationarity.

Furthermore, our test can be used complementary to the results in Schmidt et al. \citep{Schmidt.2021} who develop  a test for the stationarity of the variance based on Gini's mean difference of the logarithmic local sample variances. 
The authors too work under the model \eqref{Eq: Triangular array X}, i.e. with observations  $X_i=\sigma({i}/{n}) Y_i +\mu({i}/{n})$, where the local mean $\mu:[0,1]\rightarrow \R$ is assumed to be (piecewise) Lipschitz-continuous and where the hypothesis of a constant variance  is tested against the alternative of an arbitrary c\`adl\`ag function $\sigma$. Since their assumptions are quite similar to those obtained here (see, their Theorem 2.4 and Corollary 2.8), both procedures can be combined to test for stationarity of the first two moments. Moreover, note that in case the test in \citep{Schmidt.2021} finds the variance to be stationary, we can apply the simplified version of our test from Theorem \ref{Cor: Asymptotic test, sigma=c}.

\subsection{Estimation of  the mean function}
\label{Subsec: Estimation of CP and mu}
In case our test rejects the hypothesis of a constant trend function,
a natural next step is to determine the structure of the mean function $\mu$.
 Since our test is able to detect any deviation from a constant mean on a non-Lebesgue negligible set, it can likewise be used if the mean function is suspected to follow some parametric model. In a first step, one estimates the parameters of the assumed model, and afterwards, one subtracts the resulting mean function $\hat{\mu}(t)$, $t\in[0,1]$, from the data. Applying our procedure to the residuals  
provides a test for the hypothesis that the mean function follows the specified parametric model. A similar problem has been treated in Zhang and Wu \citep{Zhang.2011}, who derive a test for certain parametric forms of the mean by analysing the $L^2$-distance between the estimated parametric trend and a non-parametric local linear estimator.

If previous knowledge suggests a piecewise constant mean function, that is, in the classical context of structural break testing, one could estimate the number and location of the change-points once the hypothesis is rejected. 
We propose a simple recursive procedure that, in each step, conducts our test and if the hypothesis is rejected, estimates the position $t^*$ of the dominant change-point. Afterwards, the time series is split  at $t^*$ into two parts and the procedure is repeated on each part until the hypothesis can no longer be rejected.
A simple approach to determine the position $t^*$ of the dominant change-point is to identify the two subsequent blocks $B_{j^*}=\{(j^*-1)\ell_n+1, \ldots, j^*\ell_n\}$ and $B_{j^*+1}=\{j^*\ell_n+1, \ldots, (j^*+1)\ell_n\}$ with the largest difference in the estimated block means,
 $$j^*:= \underset{j\in\{1, \ldots, b_n-1\}}{argmax}\abs{\hat{\mu}_j-\hat{\mu}_{j+1}},$$
 and to locate the change-point via
  $$t^*:= \underset{t\in B_{j^*}\cup B_{j^*+1}}{argmax} \bigg| \frac{1}{t-(j^*-1)\ell_n-1}\sum_{i=(j^*-1)\ell_n+1}^{t}X_i-\frac{1}{(j^*+1)\ell_n-t-1}\sum_{i=t+1}^{(j^*+1)\ell_n}X_i\bigg |.$$
In order to obtain a reliable estimate of the mean on both parts, $\{(j^*-1)\ell_n+1, \ldots, t\}$ and $\{t+1, \ldots, (j^*+1)\ell_n\}$, we exclude such values of $t$ that are too close to the boundaries of $B_{j^*}\cup B_{j^*+1}$. A similar procedure has already been used  in Wornowizki, Fried and Meintanis \citep{Wornowizki.2017} and in Schmidt et al. \citep{Schmidt.2021} to obtain an estimator for structural breaks in the variance. The estimated mean function $\hat{\mu}$ is obtained by taking the sample mean on each segment between two subsequent change points. The suitability of a piecewise constant mean can be tested as outlined above, by subtracting $\hat{\mu}$ from the data and repeating our test.

\section{Simulation study}
\label{Sec: Simulations}
In this section, we evaluate the finite sample properties of our test for a stationary mean introduced in Section \ref{Sec: Main results}. We consider both the test in the presence of a  possibly non-constant variance described in Theorem \ref{Cor: Asymptotic test H} and its simpler version for a stationary variance stated in Theorem \ref{Cor: Asymptotic test, sigma=c}. Since the long run variance estimator 
introduced in Section \ref{Subsec: Estimation of LRV} is of some interest in itself, we assess its performance as well. 

We consider several data-generating processes, among which are two examples of independent observations, namely standard normal, $X_i \sim \NoD{0}{1}$, and exponential ones, $X_i\sim Exp(1)$, and four examples of short-range dependence. As dependent processes, we consider two AR(1)-processes with $\alpha_1=0.4$ and $\alpha_1=0.7$, respectively, an ARMA(2,2)-process
$$X_i= 0.8X_{i-1}-0.4X_{i-2}+\varepsilon_i+0.5 \varepsilon_{i-1}+0.34 \varepsilon_{i-2} $$
and a GARCH(1,1)-process
$$X_i=\sigma_i\varepsilon_i\quad \text{with} \quad \sigma_i^2= 0.1+0.1X_{i-1}^2+0.8 \sigma_{i-1}^2,$$
each with independent standard normal innovations $(\varepsilon_i)_{i\in\N}$. 
 To facilitate a comparison, all processes are standardized to yield a theoretical long run variance of one.

All data-generating processes possess finite sixth moments (see, Theorem 5 in Lindner \citep{Lindner2009} for the GARCH-case) and, as pointed out in Example \ref{Example: ARMA, GARCH},  
are at least mixing at a geometric rate. 
Hence, the requirements of Theorem \ref{Cor: Asymptotic test H}  boil down to $s>0.5$, $s>q$ and $q>1-s$. We recommend the choices $s=0.7$, $q=0.4$ and $c_0=10$ for the tuning parameters, as they overall performed best in extensive additional simulations not reported here.   
  
Since our test is consistent against the alternative of arbitrary non-constant c\`adl\`ag mean functions, we include smoothly varying mean functions (a linear and a sine function) and mean functions with structural breaks (with one, two or four breaks). More precisely, we investigate the following (local) alternatives, each with a mean change of magnitude $n^{-1/2}$: 
  \begin{align*}
\mathbb{A}1:  \mu(x)=&0.3\sqrt{1000/n}\cdot x \\
\mathbb{A}2:  \mu(x)=&0.15\cdot \sin(4\pi x)\cdot\sqrt{1000/n} \\
\mathbb{A}3: \mu(x)=& 0.3 \sqrt{1000/n}\cdot\1_{\{1/2\leq x\leq 1\}}\\
\mathbb{A}4: \mu(x)= &0.3 \sqrt{1000/n}\cdot\1_{\{1/3\leq x<2/3\}}\\
\mathbb{A}5: \mu(x)=& 0.3 \sqrt{1000/n}\cdot\1_{\{1/5\leq x< 2/5\}}+0.3 \sqrt{1000/n}\cdot\1_{\{3/5 \leq x< 4/5\}}
\end{align*}

All simulations are conducted in R \citep{RCoreTeam.2019}. The long run variance $\kappa_Y^2$, the centring term and the parameter $\psi$ are estimated as described in Section \ref{Subsec: Estimation of Parameters}. 
All results are obtained for a nominal significance level of $\alpha=5\%$ and are based on 4000 replications each.

\subsection{Performance of the original test from Theorem \ref{Cor: Asymptotic test H} for a stationary variance}
\label{Subsec: Simu: Original Procedure sigma=1}
First, we assess our test's empirical size and power for constant variance functions by setting $\sigma(x)=1$. Table \ref{Table: Theoretical Results original procedure} shows the simulated rejection probabilities for the sample lengths $n=500$ and $2000$ under the hypothesis and the local alternatives $\mathbb{A}1$-$\mathbb{A}5$. The test usually behaves anti-conservative, especially in the presence of strong positive dependence. This finding can partly be explained by the performance of the long run variance estimator, which underestimates the long run variance for such processes (see, Section \ref{Subsec: Simus LRV} for details). Even for a moderate sample size, $n=500$, the empirical power obtained for the mean functions with one or two jumps ($\mathbb{A}3$ and $\mathbb{A}4$) is quite high. The power for the remaining mean functions, smooth  ($\mathbb{A}1$ and $\mathbb{A}2$) or with four jumps ($\mathbb{A}5$), is lower but clearly improves as the sample size increases to $n=2000$ (except for  $\mathbb{A}_1$).

\begin{table}[H]
 \centering
    \caption{
       Simulated rejection probabilities of the original test at the nominal significance level $\alpha=0.05$ for the sample sizes $n = 500, 2000$ and $\sigma=1$ under the null hypothesis $\mathbb{H}$ and various local alternatives $\mathbb{A}1$ to $\mathbb{A}5$ with effect sizes of magnitude $n^{-1/2}$ and for different data-generating processes.}
    \label{Table: Theoretical Results original procedure}
    \begin{tabular}{|c|cccccc|}
       \hline
         & N(0,1) & Exp(1) & AR(1), 0.4 & AR(1), 0.7 & ARMA(2,2) & GARCH(1,1) \\ \hline  
        \multicolumn{7}{|c|}{$n=500$} \\\hline   
          $\mathbb{H}$ & 0.079 & 0.074 & 0.085 & 0.148 & 0.072 & 0.075 \\
       $\mathbb{A}_1$ & 0.489 & 0.501 & 0.539 & 0.635 & 0.489 & 0.500 \\
       $\mathbb{A}_2$ & 0.505 & 0.492 & 0.562 & 0.664 & 0.519 & 0.508 \\
       $\mathbb{A}_3$ & 0.931 & 0.933 & 0.951 & 0.970 & 0.925 & 0.938 \\
       $\mathbb{A}_4$ & 0.820 & 0.821 & 0.847 & 0.908 & 0.816 & 0.822 \\
       $\mathbb{A}_5$ & 0.414 & 0.400 & 0.476 & 0.542 & 0.426 & 0.429 \\  \hline     
         \multicolumn{7}{|c|}{$n=2000$} \\\hline  
         $\mathbb{H}$ & 0.073 & 0.074 & 0.082 & 0.122 & 0.070 & 0.082 \\
      $\mathbb{A}_1$ & 0.406 & 0.409 & 0.454 & 0.514 & 0.408 & 0.413 \\
      $\mathbb{A}_2$ & 0.613 & 0.600 & 0.630 & 0.708 & 0.615 & 0.627 \\
      $\mathbb{A}_3$ & 0.924 & 0.929 & 0.936 & 0.959 & 0.924 & 0.932 \\
      $\mathbb{A}_4$ & 0.824 & 0.812 & 0.840 & 0.883 & 0.818 & 0.830 \\
      $\mathbb{A}_5$ & 0.851 & 0.848 & 0.873 & 0.900 & 0.860 & 0.848 \\\hline
    \end{tabular}
\end{table}

\subsection{Performance of the original test from Theorem \ref{Cor: Asymptotic test H} under heteroscedasticity}
\label{Subsec: Simu: Original Procedure sigma neq 1}
The limit theory in Theorem \ref{Cor: Asymptotic test H} was explicitly developed to allow for  non-stationary variance functions $\sigma^2$. In this section, we investigate the influence of the following functions: 
\begin{align*}
\sigma_1(x)=&(1-\Theta_\sigma/2)+\Theta_\sigma\cdot x \\
\sigma_2(x)=&1+\Theta_\sigma/2 \cdot \sin(4\pi x) \\
\sigma_3(x)=  & (1- \Theta_\sigma/2)\cdot\1_{\{0\leq x< 1/2\}}+ (1+\Theta_\sigma/2)\cdot\1_{\{1/2\leq x\leq 1\}}
\end{align*}

Since mean changes will obviously be easier to detect when the variance is low, we employ functions with $\int_0^1\sigma(x)\mathrm{d}x=1$ to enable a comparison with Section \ref{Subsec: Simu: Original Procedure sigma=1}.
 Table \ref{Table: Theoretical Rejection Rates Variance varies with 0.3} shows the simulated rejection probabilities for $n=500$ with $\Theta_\sigma=0.3\sqrt{1000/500}\approx 0.42$, in which case changes in the mean and in the standard deviation are of the same magnitude (additional results for $\Theta_\sigma=0.6\sqrt{1000/500}\approx 0.85$ are included in Table \ref{Table: Theoretical Rejection Rates Variance varies with 0.6} of the appendix). When compared to Table \ref{Table: Theoretical Results original procedure}, one observes little difference to the results obtained for $\sigma=1$: The empirical size seems to be slightly lower for the linear and the jump function, $\sigma_1$ and $\sigma_3$, and slightly higher for the sine function $\sigma_2$.   For all three non-constant variance functions, the empirical power tends to be a bit lower, with the jump function $\sigma_3$ yielding the lowest rejection rates across all alternatives and data generating processes considered. Nevertheless, our procedure seems altogether little affected by the investigated variance changes.

\begin{table}
    \centering
    \caption{
       Simulated rejection probabilities of the original test at the nominal significance level $\alpha=0.05$ for the sample size $n = 500$ under the null hypothesis $\mathbb{H}$ and various local alternatives $\mathbb{A}1$ to $\mathbb{A}5$ with effect sizes of magnitude $n^{-1/2}$ and for different data-generating processes and different variance functions.}
    \label{Table: Theoretical Rejection Rates Variance varies with 0.3}
    \begin{tabular}{|c|cccccc|}
       \hline
         & N(0,1) & Exp(1) & AR(1), 0.4 & AR(1), 0.7 & ARMA(2,2) & GARCH(1,1) \\ \hline
      \multicolumn{7}{|c|}{$\sigma_1(x)=(1-0.15\sqrt{2})+0.3\sqrt{2}\cdot x$} \\\hline 
      $\mathbb{H}$   &0.068 & 0.067 & 0.082 & 0.131 & 0.066 & 0.073 \\
       $\mathbb{A}_1$ & 0.472 & 0.464 & 0.520 & 0.637 & 0.470 & 0.481 \\
       $\mathbb{A}_2$ &0.480 & 0.478 & 0.532 & 0.643 & 0.481 & 0.488 \\
       $\mathbb{A}_3$& 0.926 & 0.930 & 0.940 & 0.968 & 0.912 & 0.928 \\
       $\mathbb{A}_4$& 0.807 & 0.814 & 0.859 & 0.908 & 0.823 & 0.811 \\
      $\mathbb{A}_5$ & 0.411 & 0.376 & 0.437 & 0.510 & 0.396 & 0.404 \\ \hline    
         \multicolumn{7}{|c|}{$\sigma_2(x)=1+0.15\sqrt{2} \cdot \sin(4\pi x) $} \\\hline      	
    $\mathbb{H}$   &    0.080 & 0.070 & 0.094 & 0.153 & 0.076 & 0.082 \\
     $\mathbb{A}_1$ &   0.473 & 0.474 & 0.521 & 0.624 & 0.484 & 0.483 \\
     $\mathbb{A}_2$ &   0.497 & 0.487 & 0.552 & 0.639 & 0.479 & 0.502 \\
     $\mathbb{A}_3$ &   0.927 & 0.919 & 0.949 & 0.966 & 0.927 & 0.919 \\
     $\mathbb{A}_4$ &   0.815 & 0.801 & 0.842 & 0.901 & 0.813 & 0.809 \\
     $\mathbb{A}_5$ &   0.398 & 0.388 & 0.452 & 0.522 & 0.399 & 0.388 \\\hline
                \multicolumn{7}{|c|}{$\sigma_3(x)=   (1- 0.15\sqrt{2})\cdot\1_{\{0\leq x< 1/2\}}+ (1+0.15\sqrt{2})\cdot\1_{\{1/2\leq x\leq 1\}}$} \\\hline      
     $\mathbb{H}$   &   0.068 & 0.060 & 0.078 & 0.126 & 0.066 & 0.077 \\
      $\mathbb{A}_1$ &  0.438 & 0.414 & 0.495 & 0.602 & 0.443 & 0.460 \\
      $\mathbb{A}_2$ &  0.446 & 0.443 & 0.501 & 0.620 & 0.445 & 0.469 \\
    $\mathbb{A}_3$ &    0.907 & 0.919 & 0.930 & 0.960 & 0.905 & 0.913 \\
    $\mathbb{A}_4$ &    0.790 & 0.786 & 0.830 & 0.882 & 0.789 & 0.787 \\
   $\mathbb{A}_5$ &     0.367 & 0.354 & 0.418 & 0.479 & 0.366 & 0.361 \\\hline
    \end{tabular}
\end{table}

\subsection{Performance of the simplified test from Theorem \ref{Cor: Asymptotic test, sigma=c}}
\label{Subsec: Simu: Simplified Procedure sigma=1}
In case previous knowledge on the time series suggests a constant variance, we can adopt the simplified procedure from Theorem \ref{Cor: Asymptotic test, sigma=c}. If we assume $\sigma=1$, the centring term  becomes $2/\sqrt{\pi}$ and the variance of the limit distribution simplifies to $\psi^2= 4/3+ 8/\sqrt{\pi}(\sqrt{3}-2)$. Comparing with the rejection rates of the original test, we find the results of the simplified test to be very similar, such that they are reported in the appendix only (see Table \ref{Table: Theoretical Results simplified procedure}, and see Table \ref{Table: Empirical Results both procedures} for size-corrected results of both tests).

The main advantage of this easier version  is the much faster computation time since we do not require a Monte Carlo estimation of $\psi^2$.  As a rough comparison, we compare the overall computation time (on a 3.8 GHz AMD Ryzen 5800X) required to obtain the results in Tables \ref{Table: Theoretical Results original procedure} and \ref{Table: Theoretical Results simplified procedure}, i.e. for $4000\cdot 36$ executions of our procedure (including the time for the simulation of the data sets). For $n=500$ ($n=2000$), the original test took on average 7.89 (16.86) seconds per execution, while the simplified version needed 0.0007 (0.0018) seconds. Note, however, that we use a rather large number of $7000$ replications in the Monte Carlo estimation of $\psi^2$, as we are mainly interested in the quality and not the speed of our procedure. In case one uses only $1000$ Monte Carlo replications, the procedure takes 0.1338 (0.3076) seconds per execution, with the results being similar to those in Table \ref{Table: Theoretical Results original procedure}. 

 Due to its fast computation time, we additionally analyse the asymptotic behaviour of the simplified test for growing sample sizes $n=500$, $1000$, $2000$, $3000$, $4000$, $5000$, $8000$, $12000$ and $16000$ in Figure \ref{Figure: Asymptotic behaviour simplified procedure, theoretical}. Both for the empirical size and power, we observe a stabilization of the rejection rates as $n$ increases: The empirical size (top left) approaches the theoretical significance level of $0.05$,  though the test stays liberal. The empirical power stabilizes at a certain level that depends on the alternative considered. Across all alternatives, the shapes of the graphs for the different time series resemble one another. The AR(1)-process with parameter $0.7$ often achieves a distinctly higher empirical power, but this difference vanishes once a size-correction has been conducted (Figure \ref{Figure: Asymptotic behaviour simplified procedure, empirical} in the appendix reports the size-corrected graphs). Besides, the relative position of the blocks compared to the mean changes has some influence on the quality of the test. For instance, the rejection rates of the third and fourth alternative peak at $n=4000$. This is due to the break points at $2000$ and at $1333$, respectively, being  an almost exact multiple of the block length $4000^{0.7}\approx 332$.

\begin{center}
\begin{figure}
 \begin{subfigure}[c]{0.49\textwidth}
  \includegraphics[width=\textwidth,  trim={0cm 0cm 0.7cm 1.9cm},clip]{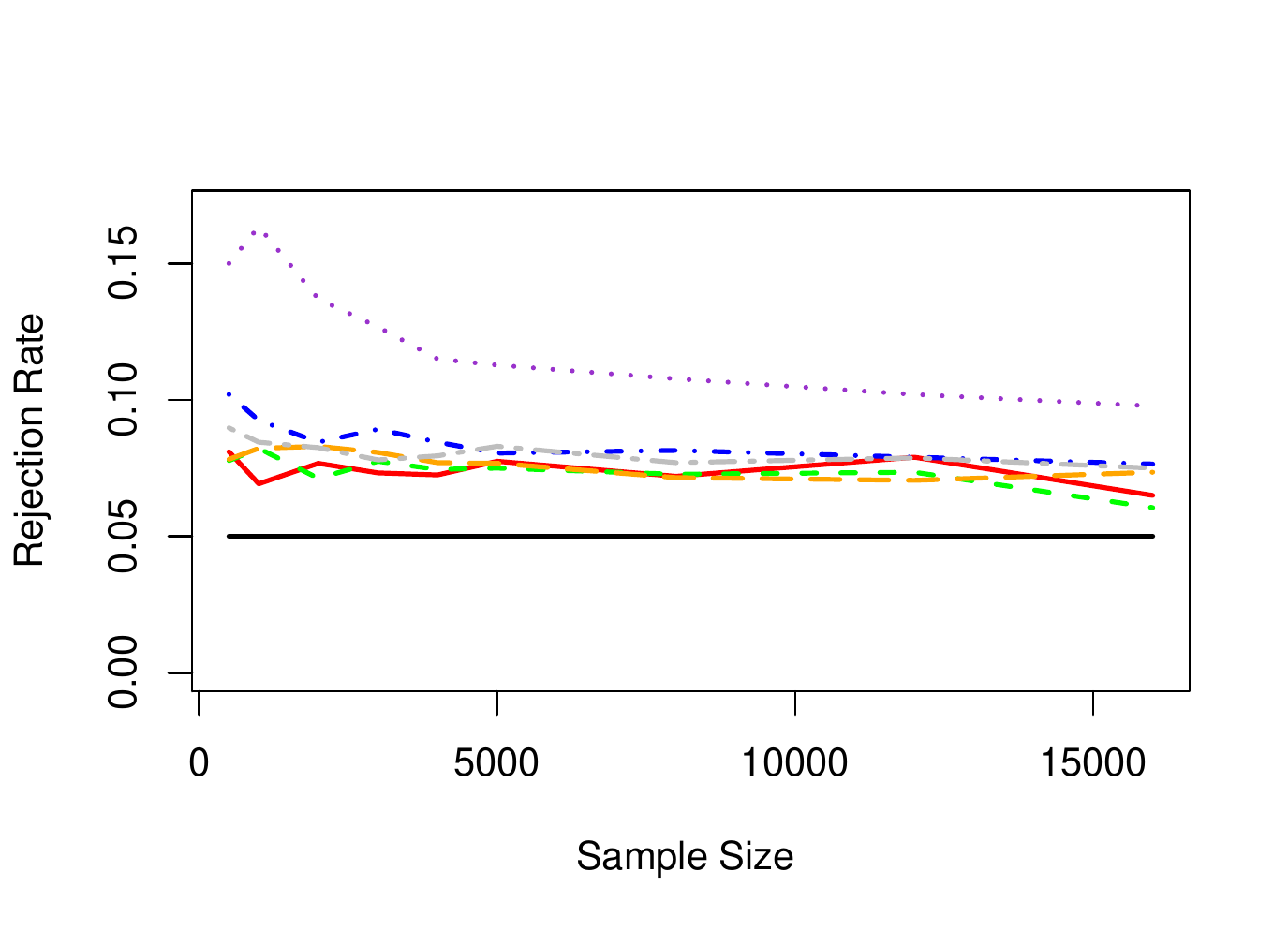}
 \end{subfigure}
   \begin{subfigure}[c]{0.49\textwidth}
  \includegraphics[width=\textwidth,  trim={0cm 0cm 0.7cm 1.9cm},clip]{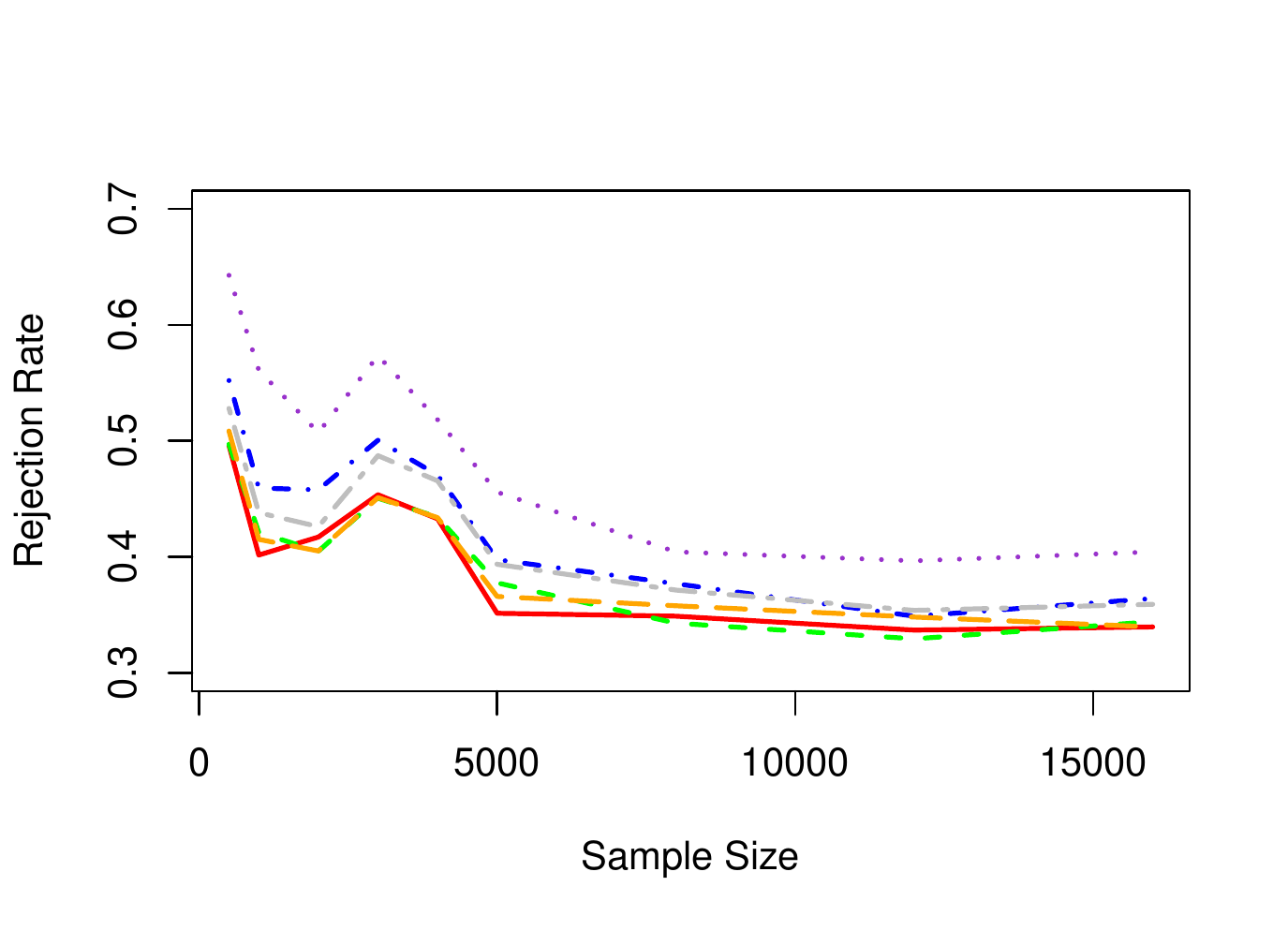}
  \end{subfigure}
   \begin{subfigure}[c]{0.49\textwidth}
  \includegraphics[width=\textwidth,  trim={0cm 0cm 0.7cm 1.9cm},clip]{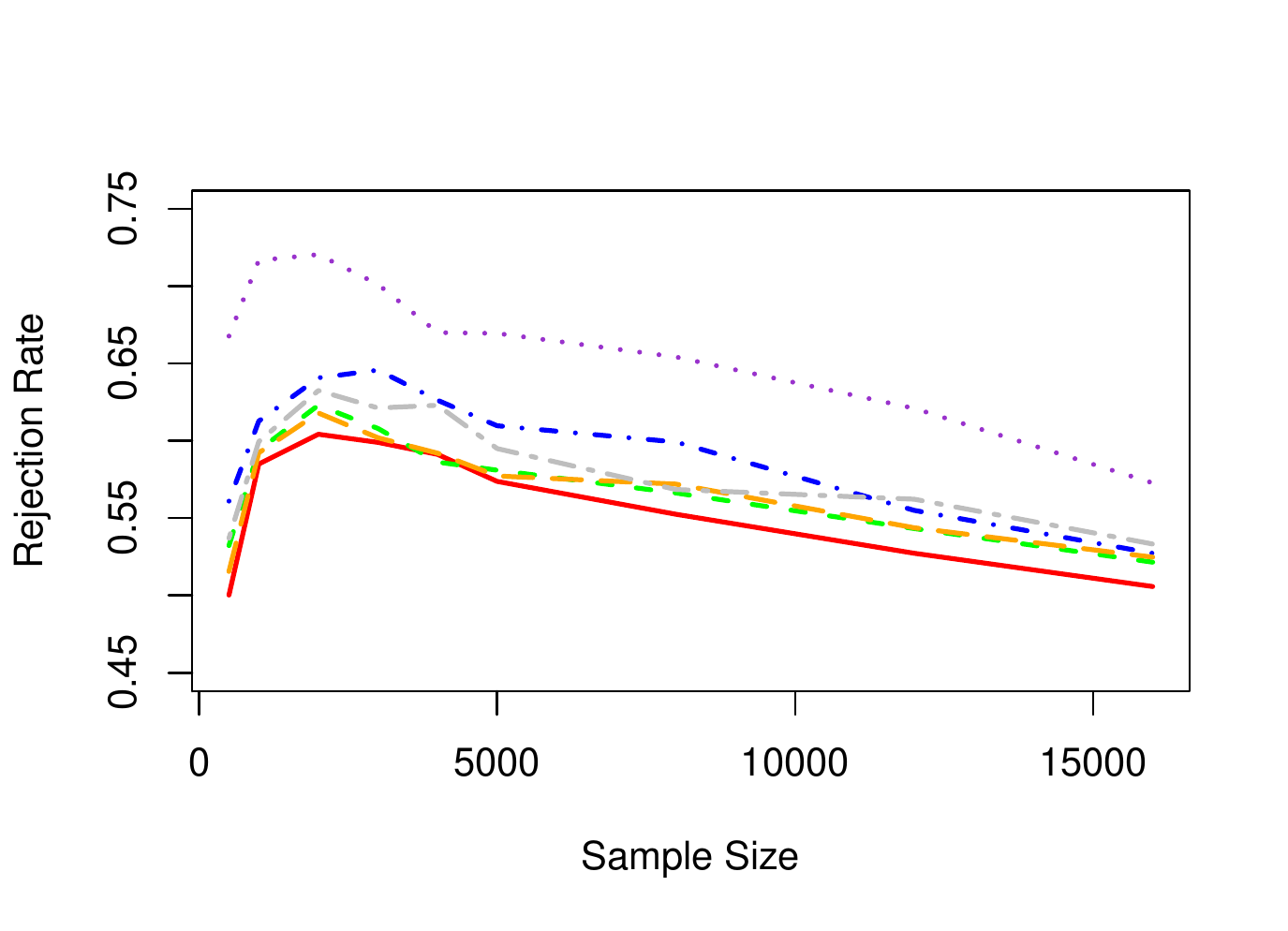}
  \end{subfigure}
   \begin{subfigure}[c]{0.49\textwidth}
  \includegraphics[width=\textwidth,  trim={0cm 0cm 0.7cm 1.9cm},clip]{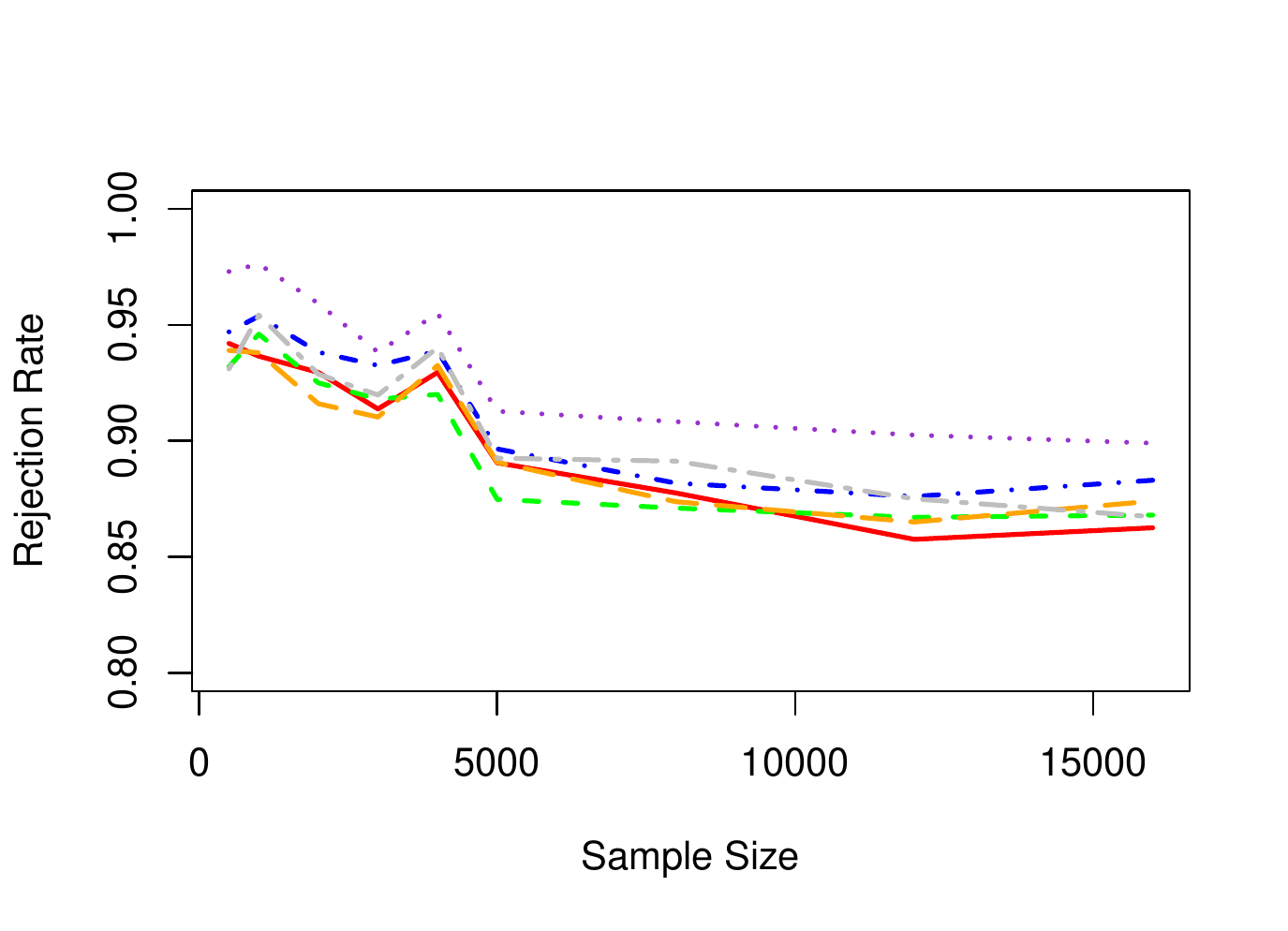}
  \end{subfigure}
     \begin{subfigure}[c]{0.49\textwidth}
  \includegraphics[width=\textwidth,  trim={0cm 0cm 0.7cm 1.9cm},clip]{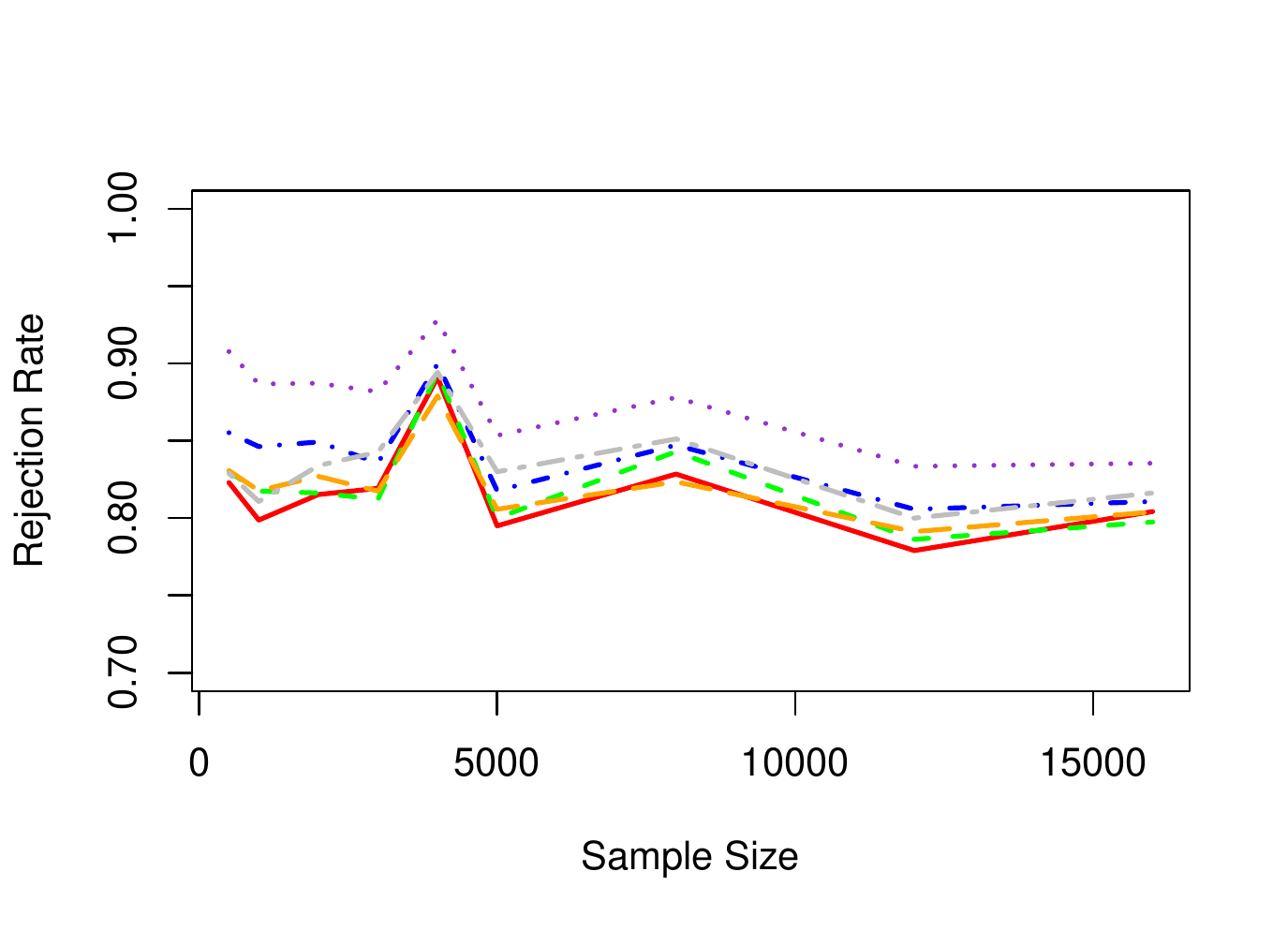}
  \end{subfigure}
     \begin{subfigure}[c]{0.49\textwidth}
  \includegraphics[width=\textwidth,  trim={0cm 0cm 0.7cm 1.9cm},clip]{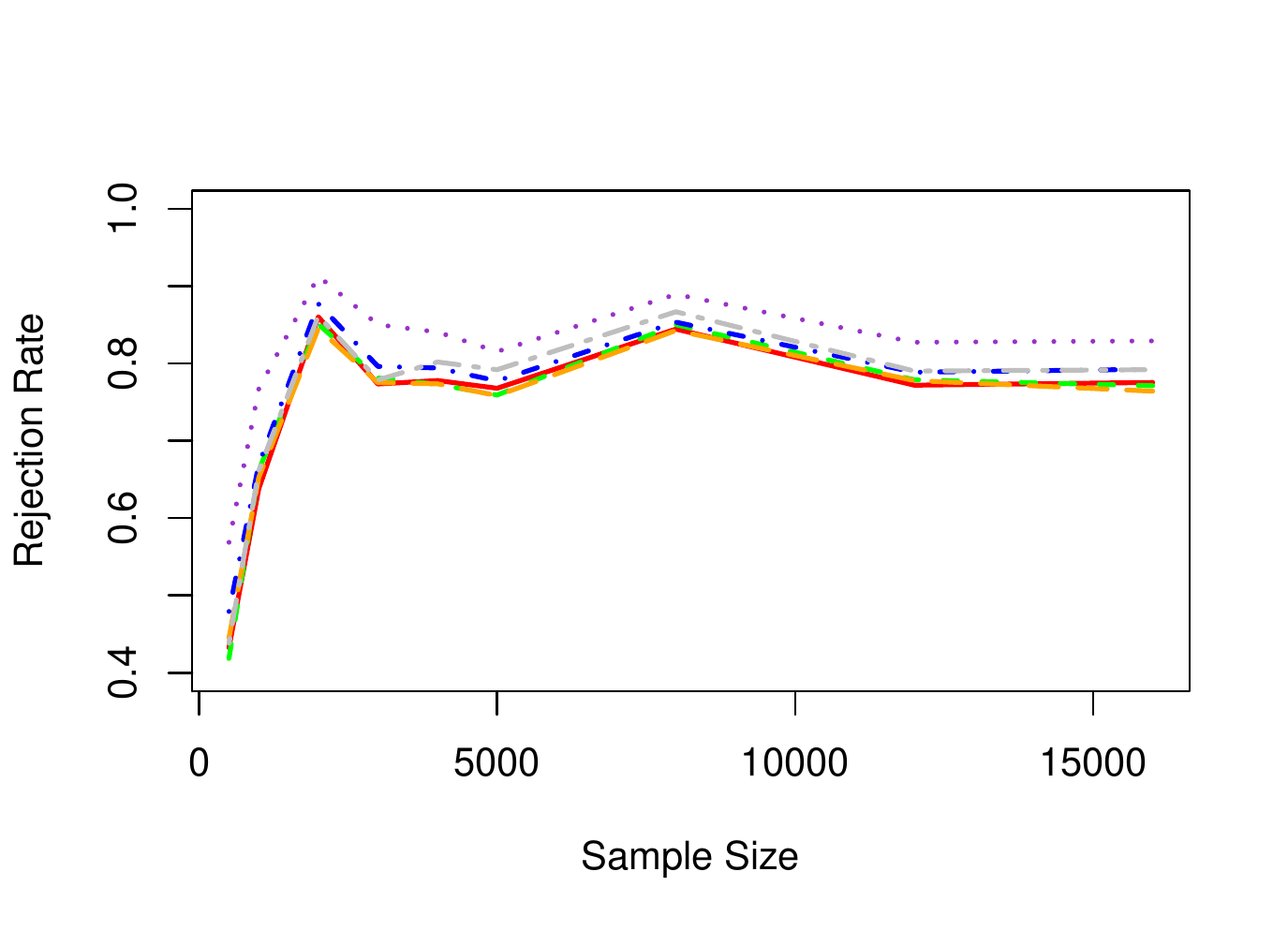}
  \end{subfigure}
   \begin{subfigure}[c]{\textwidth}
  \includegraphics[width=\textwidth, trim={0cm 7cm 0cm 6.5cm},clip]{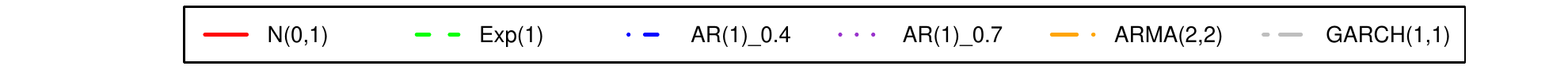}
  \end{subfigure}
  \caption{Empirical rejection rates at the nominal significance level $\alpha=0.05$ of the simplified test under the hypothesis (top left) and the (local) alternatives $\mathbb{A}_1$ (top right) -$\mathbb{A}_5$ (bottom right) as a function of the sample size for different distributions of the data-generating process.}
  \label{Figure: Asymptotic behaviour simplified procedure, theoretical}
\end{figure}
\end{center}

\subsection{Performance of the long run variance estimator}
\label{Subsec: Simus LRV}
The quality of the long run variance estimator plays a crucial role in the performance of our asymptotic test, such that we include a brief evaluation of the estimator proposed in Section \ref{Subsec: Estimation of LRV}.  Since our main concern is its behaviour in the presence of a non-constant mean, we limit our considerations to a constant variance $\sigma=1$  and consider the estimator $\hat{\kappa}_{\tilde{X}}$. To facilitate a comparison, we once more standardize the data-generating processes to yield a theoretical long run variance of 1.
 To examine the effects of a growing sample size in isolation, 
 we consider fixed alternatives $\mathbb{A}'_1$-$\mathbb{A}'_5$ here. More precisely, we replace the magnitude $0.3\sqrt{1000/n}$ of the local alternatives $\mathbb{A}_1$-$\mathbb{A}_5$ by $0.3\sqrt{2}\approx 0.42$, corresponding to the maximal mean change considered above (the change for $n=500$). Note that $\hat{\kappa}_{\tilde{X}}$ is consistent under all alternatives $\mathbb{A}'_1$-$\mathbb{A}'_5$ examined. 

Table \ref{Table: LRV estimator Bias RMSE} shows the empirical bias and root mean square error (RMSE) of  $\hat{\kappa}_{\tilde{X}}$ for $q=0.4$, $c_0=10$ and the sample lengths $n=500$ and $2000$. 
For the empirical bias, two opposite effects can be observed. On the one hand, depending on the underlying time series, $\hat{\kappa}_{\tilde{X}}$ has a negative bias that can best be seen for a constant mean ($\mathbb{H}$). This negative bias already appears with the subsampling estimators of type $\hat{\kappa}_Y$ developed for centred data and decreases as the sample size increases. 
On the other hand, the variations in the mean cause a positive bias. Its height depends on the alternative under consideration, where fewer variation in the form of a linear trend ($\mathbb{A}_1'$) or one structural break ($\mathbb{A}_2'$) is better coped with than the more pronounced variation of the sine function ($\mathbb{A}_2'$) or multiple structural breaks ($\mathbb{A}_4'$ and $\mathbb{A}_5'$). As $n$ increases, the positive bias that is due to mean distortions decreases as well. The RMSE is overall less influenced by the type of mean function and underlying time series, and likewise decreases as $n$ increases.

Recall that for the construction of the asymptotic test in Theorem \ref{Thm: Central limit theorem}, we divide the U-statistic by $\kappa_Y$. An underestimation of $\kappa_Y$, i.e. a negative bias of $\hat{\kappa}_{\tilde{X}}$, will thus lead to a large value of the test statistic, to high rejection rates and ultimately to an oversized test. In contrast, a positive bias caused by mean changes will lead to a loss of the test's power. The severity of each type of bias depends on the choice of $q$  (additional results for $q=0.3$ are reported in Table \ref{Table: LRV estimator Bias RMSE both qs} in the appendix). In general, one has to find a $q$ that balances both effects and if available, takes previous knowledge on the time series into account. Above, we choose $q=0.4$ to avoid a severely oversized test  (note that our theory for $s=0.7$ requires $q>(1-s)=0.3$).

\begin{table}[H]
\caption{Simulated Bias and RMSE of $\hat{\kappa}_{\tilde{X}}$ for the mean functions represented by $\mathbb{H}$ and the fixed alternatives $\mathbb{A}'_1$-$\mathbb{A}'_5$, the sample lengths $n=500, 2000$, and for the parameter choices $q=0.4$ with $\elln=n^q$ and $c_0=10$. The observations are standardized to yield a theoretical long run variance $\kappa_Y^2 =1$.}
\label{Table: LRV estimator Bias RMSE}
{\footnotesize
\begin{tabular}{|c|rc|rc|rc|rc|rc|rc|}
    \hline
& \multicolumn{2}{|c|}{\multirow{2}{*}{N(0,1)}} & \multicolumn{2}{c|}{\multirow{2}{*}{Exp(1)}} &  \multicolumn{2}{c|}{\multirow{2}{*}{AR(1), 0.4}} &  \multicolumn{2}{c|}{\multirow{2}{*}{AR(1), 0.7}} &  \multicolumn{2}{c|}{\multirow{2}{*}{ARMA(2,2)}} &  \multicolumn{2}{c|}{\multirow{2}{*}{GARCH(1,1)}} \\
  & \multicolumn{2}{|c|}{} & \multicolumn{2}{c|}{}  & \multicolumn{2}{c|}{} & \multicolumn{2}{c|}{} & \multicolumn{2}{c|}{} & \multicolumn{2}{c|}{} \\
 & \multicolumn{1}{|c}{Bias} & RMSE &  \multicolumn{1}{c}{Bias} & RMSE &  \multicolumn{1}{c}{Bias} & RMSE & \multicolumn{1}{c}{Bias} & RMSE & \multicolumn{1}{c}{Bias} & RMSE & \multicolumn{1}{c}{Bias} & RMSE 
  \\ \hline
 \multicolumn{13}{|c|}{$n=500$}\\\hline
   $\mathbb{H}$ &   0.003 & 0.125 & -0.003 & 0.129 & -0.038 & 0.123 & -0.121 & 0.163 & -0.001 & 0.122 & -0.017 & 0.137 \\
   $\mathbb{A}'_1$ &     0.005 & 0.121 & -0.002 & 0.128 & -0.039 & 0.125 & -0.121 & 0.165 & 0.000 & 0.124 & -0.014 & 0.135 \\
   $\mathbb{A}'_2$ &     0.119 & 0.179 & 0.110 & 0.177 & 0.082 & 0.156 & 0.006 & 0.128 & 0.108 & 0.175 & 0.101 & 0.176 \\
   $\mathbb{A}'_3$  &   0.048 & 0.137 & 0.041 & 0.138 & 0.007 & 0.123 & -0.076 & 0.140 & 0.043 & 0.135 & 0.028 & 0.141 \\
    $\mathbb{A}'_4$ &   0.101 & 0.165 & 0.094 & 0.164 & 0.062 & 0.137 & -0.010 & 0.121 & 0.092 & 0.157 & 0.085 & 0.160 \\
   $\mathbb{A}'_5$ &    0.282 & 0.324 & 0.284 & 0.328 & 0.252 & 0.298 & 0.190 & 0.247 & 0.278 & 0.320 & 0.276 & 0.319 \\
 \hline
  \multicolumn{13}{|c|}{$n=2000$}\\\hline
 $\mathbb{H}$ &       0.002 & 0.079 & -0.002 & 0.081 & -0.025 & 0.081 & -0.072 & 0.102 & -0.001 & 0.079 & -0.016 & 0.086 \\
 $\mathbb{A}'_1$ &      0.003 & 0.079 & -0.004 & 0.081 & -0.021 & 0.081 & -0.071 & 0.103 & -0.002 & 0.077 & -0.013 & 0.083 \\
  $\mathbb{A}'_2$ &   0.022 & 0.080 & 0.020 & 0.083 & 0.000 & 0.076 & -0.052 & 0.089 & 0.019 & 0.081 & 0.009 & 0.084 \\
  $\mathbb{A}'_3$ &    0.034 & 0.086 & 0.028 & 0.088 & 0.006 & 0.078 & -0.041 & 0.086 & 0.027 & 0.083 & 0.016 & 0.085 \\
  $\mathbb{A}'_4$ &    0.059 & 0.101 & 0.050 & 0.098 & 0.036 & 0.087 & -0.012 & 0.080 & 0.055 & 0.098 & 0.042 & 0.095 \\
  $\mathbb{A}'_5$ &    0.128 & 0.153 & 0.124 & 0.151 & 0.106 & 0.134 & 0.059 & 0.100 & 0.124 & 0.150 & 0.118 & 0.147 \\\hline  
    \end{tabular}}
\end{table}

\section{Data examples}
\label{Sec: Data Example}

\subsection{Central England temperature data} 
We consider the Hadley Centre Central England Temperature (HadCET) dataset publicly available at \url{https://www.metoffice.gov.uk/hadobs/hadcet/} monitoring the central England temperature from 1659 onwards. 
This data set has already been investigated by various authors, though possibly for shorter time spans (see, for instance, Khismatullina and Vogt \citep{KhismatullinaVogt.2020} or Zhang and Wu \citep{Zhang.2011} and the references therein).  We test for the stationarity of the mean of the annual average temperatures from 1659 to 2020, such that $n=362$. Our test rejects the hypothesis of a constant mean at significance level $\alpha=0.05$ and the recursive procedure described in Section \ref{Subsec: Estimation of CP and mu} detects four change points, 1691, 1702, 1896 and 1997, marked by the red vertical lines in Figure \ref{Fig: Temperature plot and CPS}. One can thus calculate a piecewise constant estimate $\hat{\mu}$ by taking the sample mean on each resulting segment (see, Figure \ref{Fig: Temperature plot and CPS with pw constant mean function} in the appendix). Testing the annual average temperatures centred by $\hat{\mu}$ for a stationary mean, the hypothesis is accepted at $5\%$ significance level.   
Alternatively, we fit two parametric models to the time series, a quadratic and a cubic trend, which are depicted in Figure \ref{Fig: Temperature plot and CPS} by the green and blue line, respectively. The time series is then centred by the respective fitted trend and once more, we test for the stationarity of its mean. 
While the quadratic trend fit is rejected for $\alpha=0.05$, the cubic model is accepted.
 We thus conclude that the mean function could be modelled by $\hat{\mu}(t)= 8.67(0.12)+3.7(1.1)t-9.2(2.5)t^2+7.2(1.6)t^3$ for $t\in[0,1]$ (standard errors are given in parentheses), which is in line with the findings in Zhang and Wu \citep{Zhang.2011}. 
We additionally apply the asymptotic test proposed in Schmidt et al. \citep{Schmidt.2021} and find that the hypothesis of a constant variance can not be rejected for $\alpha=0.05$. Hence, we can alternatively employ our simplified test, which arrives at the same 
conclusions as the original one. 

 \begin{figure}[H]
 \includegraphics[width=\textwidth, trim={0 0 0 1.3cm}, clip]{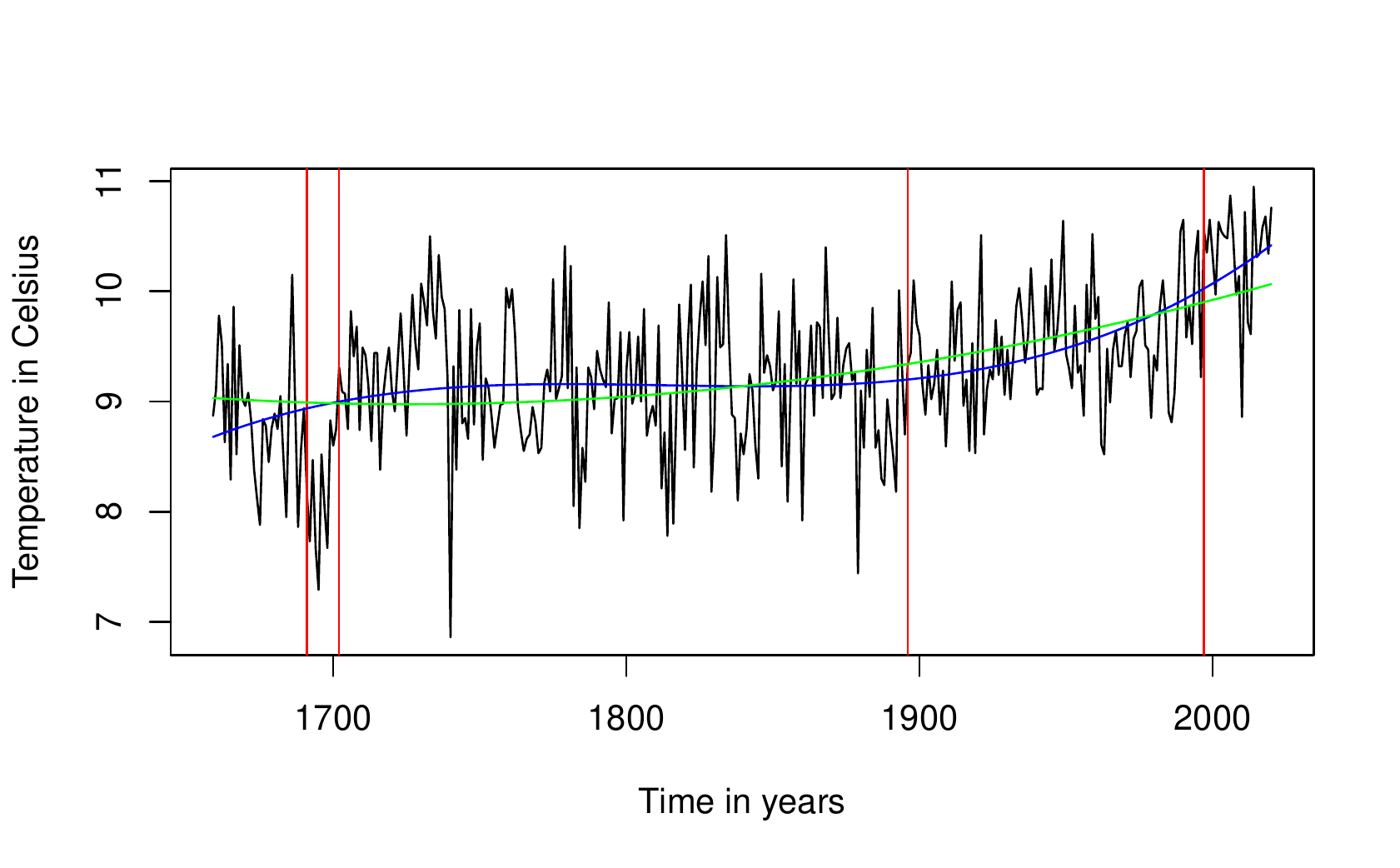}
\vspace*{-1cm}\caption{Annual central England temperatures from 1659 to 2020. Detected mean change points for $\alpha=0.05$ are marked by the red vertical lines. The fitted quadratic and cubic trend are depicted by the green and blue line, respectively.}
\label{Fig: Temperature plot and CPS}
\end{figure}

\subsection{Number of live births of girls in Germany} 
As a second data example, we consider the number of live births of girls in Germany from January 1950 to December 2020. 
The data is publicly available at the website of the Federal Statistical Office of Germany \url{https://www-genesis.destatis.de/genesis/online?operation=table&code=12612-0002}. The development of the number of births over time is depicted in Figure \ref{Fig: Number of births} and can roughly be divided into three periods:
A period of high or even rising birth numbers during the German ``Wirtschaftswunder'' after the Second World War, followed by a sharp decline marking the end of the baby boomer generation in 
1964, and a third period of comparatively low birth numbers from the early 1970's onwards. 

 \begin{figure}[H]
 \includegraphics[width=\textwidth, trim={0 0 0 1.9cm}, clip]{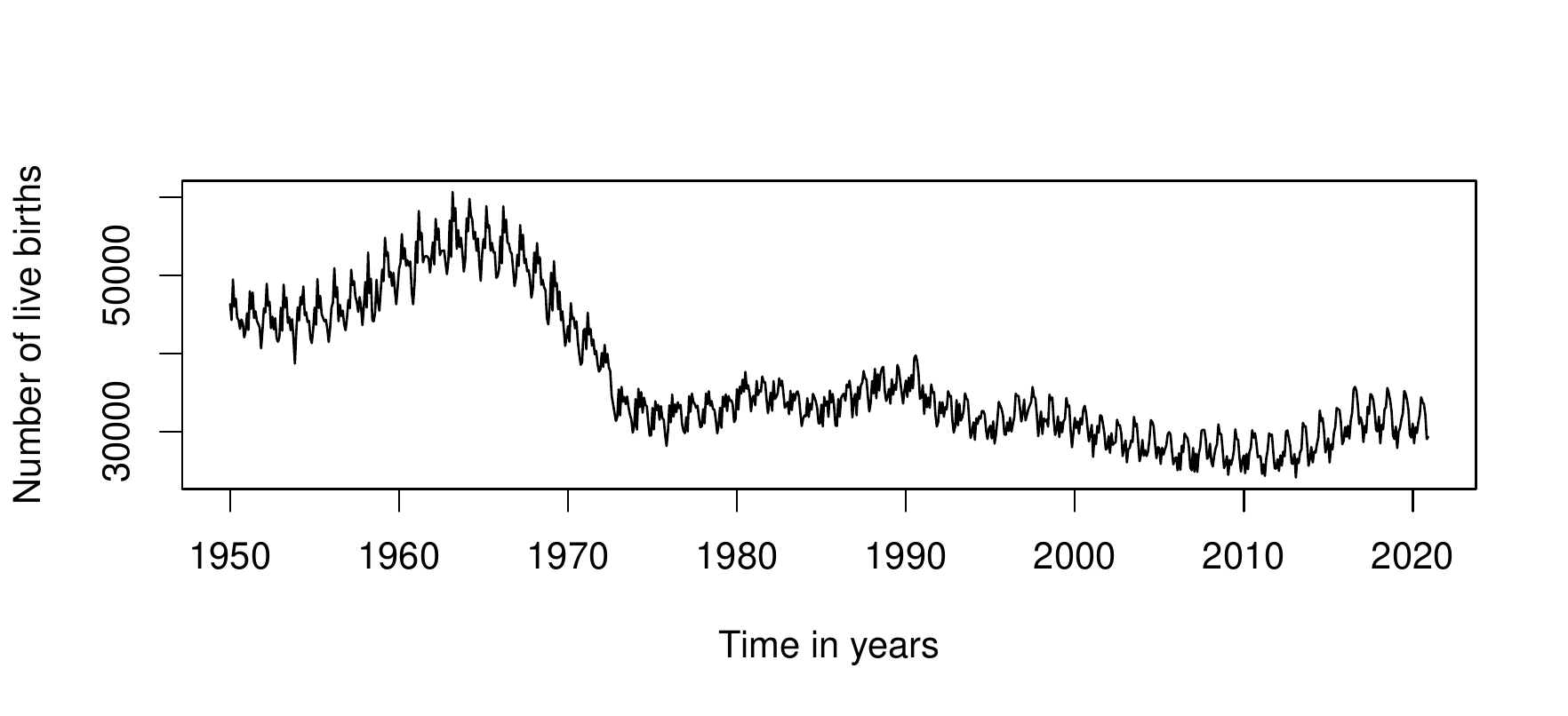}
\vspace*{-1cm}\caption{Monthly number of live births of girls in Germany from January 1950 to December 2020.}
\label{Fig: Number of births}
\end{figure}

To eliminate seasonal effects, we analyse the time series of annual differences $(Z_i)_{i\in\N}$ with $Z_i=X_i-X_{i-12}$, measuring the change 
 in the number of births. Starting in January 1951, we arrive at a total sample length of $n=840$. 
We subsequently use our test to assess the stationarity of the mean in the time series $(Z_i)_{i\in\N}$. 
At significance level $\alpha=0.05$,  the recursive procedure described in Section \ref{Subsec: Estimation of CP and mu} detects four mean change points located in February 1956, January 1965, November 1966 and August 1973. 
The upper graph of Figure \ref{Fig: Change in number of births} shows the observations $Z_1, \ldots, Z_{840}$ together with the locations of the estimated mean change points marked by the red vertical lines and the resulting piecewise constant mean function $\hat{\mu}$ depicted in blue. 
The hypothesis of a stationary variance is rejected at significance level $\alpha =0.05$ by the test of Schmidt et al. \citep{Schmidt.2021} and their recursive procedure detects altogether six change points in the variance, marked by the green vertical lines in the lower graph of Figure \ref{Fig: Change in number of births}. 
Despite the non-constant variance, our procedure seems to capture the mean changes quite well and is successful
in distinguishing them from the changes in the variance. 
Interestingly, our test detects no changes in the mean during the last nearly five decades, for which the sample mean of the annual differences is -54.39. 
Coming back to the original time series, this signifies a 
persisting decline in the number of births by approximately 54 per year.

 \begin{figure}
 \includegraphics[width=\textwidth, trim={0 0 0 1.9cm}, clip]{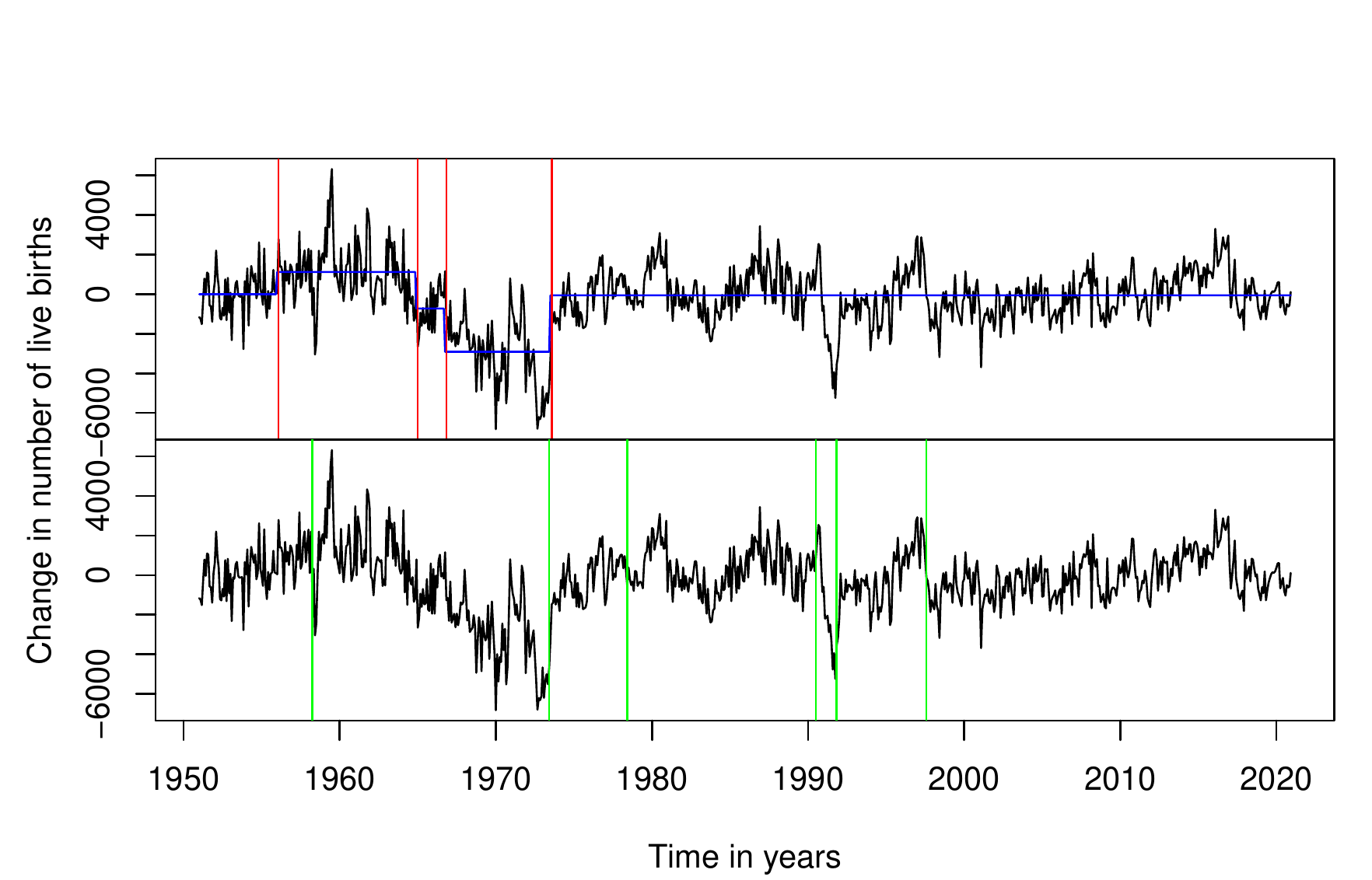}
\vspace*{-0.6cm}\caption{Change in the monthly number of live births of girls in Germany from January 1951 to December 2020. The upper graph shows the estimated change points in the mean, marked by the red vertical lines, together with the estimated piecewise constant mean function depicted in blue. The lower graph shows the estimated change points in the variance marked by the green vertical lines.}
\label{Fig: Change in number of births}
\end{figure}

\section*{Acknowledgements} The research was supported by the DFG Collaborative Research Center 823 \emph{Statistical Modelling of Nonlinear Dynamic Processes} and by the Friedrich-Ebert-Stiftung. 
Moreover, the author would like to thank Herold Dehling for valuable discussions and for proofreading the manuscript, both of which helped to improve the presentation significantly.

%

\pagebreak

\pagebreak
\appendix

\section{Proofs}
This appendix contains the proof details for the results in Sections \ref{Sec: Main results} and \ref{Sec: Outline of Proofs}. Throughout, we will use the shortening notations $\mu_i:=\mu(i/n)$ as well as $\sigma_i:=\sigma(i/n)$. Due to $\mu$ and $\sigma$ both being  c\`adl\`ag functions, they are bounded and we will denote $\mu_{\sup}:=\sup_{x\in [0,1]} \abs{\mu(x)}$ and $\sigma_{\sup}:=\sup_{x\in [0,1]} \sigma (x)$. Since $\sigma$ is additionally piecewise Lipschitz-continuous,  we split the variance function into $\sigma=:\sigma_L+\sigma_J$, where $\sigma_L$ and $\sigma_J$ denote the Lipschitz-continuous and the jump component, respectively. For $\sigma_L$, we will make frequent use of
$$\abs{\sigma_L\rbraces{\frac{i}{n}}-\sigma_L\rbraces{\frac{r}{n}}}\leq K \frac{\abs{i-r}}{n}\leq K \frac{1}{b_n},$$
whenever $i$ and $r$ are from the same block, that is, $i,r\in \{(j-1)\ell_n+1, \ldots, j\ell_n\}$ for some $1\leq j\leq b_n$. Without loss of generality, we  set $K=1$. Note that the above inequality likewise holds for $\sigma$ itself on all but (at most) $k_\sigma$ blocks, on which the absolute difference is dominated by the jump part $\sigma_J$ and can only be bounded by a constant.

  We will subsequently use $C$ to denote a positive constant, whose exact value is of no importance and  might even change from line to line. 
\subsection{A first approximation}
\begin{proof}[Proof of Proposition \ref{Prop: Approximation U U1}]
It holds 
\begin{align*}
& \sqrt{n} \abs{U(n)-U_1(n)}\\
\leq & \sqrt{n} \frac{1}{b_n(b_n-1)}\sum_{1\leq j \neq k \leq b_n} \abs{\rbraces{ \frac{1}{\ell_n}\sum_{i=(j-1)\ell_n+1}^{j\ell_n} (\sigma_i-\sigma_{j\ell_n}) Y_i} - \rbraces{\frac{1}{\ell_n}\sum_{i=(k-1)\ell_n+1}^{k\ell_n} (\sigma_i-\sigma_{k\ell_n}) Y_i }}\\
\leq & \sqrt{n} \frac{2}{b_n}\sum_{1\leq j \leq b_n} \abs{\frac{1}{\ell_n}\rbraces{\sum_{i=(j-1)\ell_n+1}^{j\ell_n} (\sigma_i-\sigma_{j\ell_n}) Y_i} }.
\end{align*}
Thus,
 \begin{align*}
& \norm{\sqrt{n} \rbraces{U(n)-U_1(n)}}_2\\
 \leq & 2\frac{\sqrt{n} }{b_n}\sum_{1\leq j \leq b_n} \norm{\frac{1}{\ell_n}\sum_{i=(j-1)\ell_n+1}^{j\ell_n} (\sigma_{L,i}-\sigma_{L,j\ell_n}) Y_i}_2 + 2\frac{\sqrt{n} }{b_n}\sum_{1\leq j \leq b_n} \norm{\frac{1}{\ell_n}\sum_{i=(j-1)\ell_n+1}^{j\ell_n} (\sigma_{J,i}-\sigma_{J,j\ell_n}) Y_i}_2,
\end{align*}
where we split $\sigma=\sigma_L+\sigma_J$ up into the Lipschitz-continuous part $\sigma_L$ and the jump part $\sigma_J$. Using the Lipschitz-continuity, we obtain 
\begin{align*}
& \frac{\sqrt{n} }{b_n}\sum_{1\leq j \leq b_n} \norm{\frac{1}{\ell_n}\sum_{i=(j-1)\ell_n+1}^{j\ell_n} (\sigma_{L,i}-\sigma_{L,j\ell_n}) Y_i}_2\\
= &  \frac{\sqrt{n} }{b_n}\sum_{1\leq j \leq b_n} \rbraces{\frac{1}{\ell_n^2}\sum_{i=(j-1)\ell_n+1}^{j\ell_n} \sum_{r=(j-1)\ell_n+1}^{j\ell_n} (\sigma_{L,i}-\sigma_{L,j\ell_n})(\sigma_{L,r}-\sigma_{L,j\ell_n})  \E{Y_iY_r} }^{1/2}\\
\leq  &  \frac{\sqrt{n} }{b_n}\sum_{1\leq j \leq b_n} \rbraces{\frac{1}{\ell_n^2}\sum_{i=(j-1)\ell_n+1}^{j\ell_n} \sum_{r=(j-1)\ell_n+1}^{j\ell_n} \abs{\sigma_{L,i}-\sigma_{L,j\ell_n}}\abs{\sigma_{L,r}-\sigma_{L,j\ell_n}} \abs{\E{Y_iY_r}} }^{1/2}\\
\leq &  \frac{\sqrt{n} }{b_n^2}b_n\rbraces{\frac{1}{\ell_n^2}\sum_{i=1}^{\ell_n} \sum_{r=1}^{\ell_n} \abs{\cov{Y_i}{Y_r}} }^{1/2}
=  \frac{\sqrt{n} }{b_n^2} \frac{b_n}{\sqrt{\ell_n}} \rbraces{\Var{Y_1}+\frac{2}{\ell_n}\sum_{k=1}^{\ell_n-1} (\ell_n-k) \abs{\cov{Y_1}{Y_{k+1}}} }^{1/2}\\
\leq & \frac{1}{\sqrt{b_n}} \rbraces{1+2\sum_{k=1}^{\infty}  \abs{\cov{Y_1}{Y_{k+1}}} }^{1/2}
\leq  \frac{1}{\sqrt{b_n}} \rbraces{1+ 16 \norm{Y_1}_{2+\delta}^2\sum_{k=1}^{\infty} \alpha_Y(k)^{\frac{\delta}{2+\delta}}}^{1/2},
\end{align*}
where the last inequality is due to Davydov's covariance inequality. By assumption (A2), the last sum is finite and by (A1), $\norm{Y_1}_{2+\delta}<\infty$, such that the last expression converges towards zero.

Recall that at most $k_\sigma$ blocks are influenced by a jump in the variance (of maximal absolute height $\Delta_\sigma$), whereas $\sigma_J$ is constant on the remaining $b_n-k_\sigma$ ones. Moreover, the error made due to the different jumps is additive (in case more than one jump occurs within one block). This leads to 
\begin{align*}
& \frac{\sqrt{n} }{b_n}\sum_{1\leq j \leq b_n} \norm{\frac{1}{\ell_n}\sum_{i=(j-1)\ell_n+1}^{j\ell_n} (\sigma_{J,i}-\sigma_{J,j\ell_n}) Y_i}_2 \\
 = & \frac{\sqrt{n} }{b_n}\sum_{1\leq j \leq b_n} \rbraces{\frac{1}{\ell_n^2}\sum_{i=(j-1)\ell_n+1}^{j\ell_n} \sum_{r=(j-1)\ell_n+1}^{j\ell_n} \abs{\sigma_{J,i}-\sigma_{J,j\ell_n}}\abs{\sigma_{J,r}-\sigma_{J,j\ell_n}} \abs{\E{Y_iY_r}} }^{1/2}\\
 \leq &  \frac{\sqrt{n} }{b_n}k_\sigma \Delta_\sigma \rbraces{\frac{1}{\ell_n^2}\sum_{i=1}^{\ell_n} \sum_{r=1}^{\ell_n} \abs{\cov{Y_i}{Y_r}} }^{1/2}
  \leq \frac{\sqrt{n} }{b_n \sqrt{\ell_n}}k_\sigma \Delta_\sigma \rbraces{1+ 16 \norm{Y_1}_{2+\delta}^2\sum_{k=1}^{\infty} \alpha_Y(k)^{\frac{\delta}{2+\delta}}}^{1/2}, 
\end{align*}
which is of order $O(1/\sqrt{b_n})$ as argued above. 

\end{proof}

\subsection{Behaviour of the test statistic under the alternative}
\begin{proof}[Proof of Theorem \ref{Thm: Behaviour under A}]
Due to Proposition \ref{Prop: Approximation U U1}, we only have to verify that 
$$ \abs{U_1(n)-\int_0^1 \int_0^1 \abs{\mu(x)-\mu(y)} \mathrm{d}x \mathrm{d}y}$$
converges in $L^2$ towards zero.
First, we will prove that $U_1(n)$ can be approximated by a deterministic expression, 
\begin{align*}
&\abs{U_1(n)-\frac{1}{b_n(b_n-1)} \sum_{1\leq j \neq k \leq b_n} \abs{\mu_{j\ell_n}-\mu_{k\ell_n}}}\\
\leq & \frac{2}{b_n} \sum_{1\leq j \leq b_n} \sigma_{j\ell_n} \frac{1}{\ell_n} \abs{\sum_{i=(j-1)\ell_n+1}^{j\ell_n} Y_i}
+ \frac{2}{b_n} \sum_{1\leq j \leq b_n} \abs{\frac{1}{\ell_n} \sum_{i=(j-1)\ell_n+1}^{j\ell_n} \mu_i - \mu_{j\ell_n}}.
\end{align*}
The first sum converges in $L^2$ towards zero due to 
\begin{align*}
\norm{\frac{1}{b_n} \sum_{1\leq j \leq b_n} \sigma_{j\ell_n} \frac{1}{\ell_n} \abs{\sum_{i=(j-1)\ell_n+1}^{j\ell_n} Y_i}}_2 \leq \sigma_{\sup} \frac{1}{\sqrt{\ell_n}} \norm{\frac{1}{\sqrt{\ell_n}}\sum_{i=1}^{\ell_n}Y_i}_2=O\rbraces{\frac{1}{\sqrt{\ell_n}}}
\end{align*}
since $\norm{\frac{1}{\sqrt{\ell_n}}\sum_{i=1}^{\ell_n}Y_i}_2\rightarrow \kappa_Y$. Due to $\mu$ be being a c\`adl\`ag function, there exists for every $\varepsilon>0$ an $r\in \N$  and $0=t_0<t_1< \ldots <t_r=1$ such that 
$$\max_{1\leq i\leq r} \sup_{s,t\in [t_{i-1}, t_i)} \abs{\mu(s)-\mu(t)} <\varepsilon$$
(see, Lemma 1, Section 14 in \citep{Billingsley.1968}). Hence, only $r$ of the $b_n$ intervals $(\frac{(j-1)\ell_n}{n}, \frac{j \ell_n}{n}]$, $1\leq j \leq b_n$, intersect with more than one of the intervals $[t_{i-1}, t_i)$. Considering the second sum above, there are at most $r$ cases with 
$$ \abs{\frac{1}{\ell_n} \sum_{i=(j-1)\ell_n+1}^{j\ell_n} \mu_i - \mu_{j\ell_n}} \leq \frac{1}{\ell_n} \sum_{i=(j-1)\ell_n+1}^{j\ell_n}  \abs{\mu_i - \mu_{j\ell_n}} \leq 2 \mu_{\sup}.$$ 
In the other $b_n-r$ cases, we have 
$$
 \frac{1}{\ell_n} \sum_{i=(j-1)\ell_n+1}^{j\ell_n}  \abs{\mu_i - \mu_{j\ell_n}} \leq \max_{1\leq i\leq r} \sup_{s,t\in [t_{i-1}, t_i)} \abs{\mu(s)-\mu(t)} <\varepsilon.
 $$
Altogether, the second and deterministic sum converges towards zero due to 
$$ \frac{2}{b_n} \sum_{1\leq j \leq b_n} \abs{\frac{1}{\ell_n} \sum_{i=(j-1)\ell_n+1}^{j\ell_n} \mu_i - \mu_{j\ell_n}}
\leq 4 \frac{r}{b_n} \mu_{\sup}+ 2 \frac{b_n-r}{b_n} \varepsilon < \varepsilon
$$
for $n$ chosen large enough. Thus, it remains to prove convergence towards the Riemann-integral,
$$\frac{1}{b_n(b_n-1)} \sum_{1\leq j \neq k \leq b_n} \abs{\mu_{j\ell_n}-\mu_{k\ell_n}} \rightarrow \int_0^1 \int_0^1 \abs{\mu(x)-\mu(y)} \mathrm{d}x \mathrm{d}y.$$ 
As a c\`adl\`ag function, $\mu$ is bounded and has only countably many discontinuities on $[0,1]$ (see, Section 14 in \citep{Billingsley.1968}). Hence, the points of discontinuity of the function $g(x,y):=\abs{\mu(x)-\mu(y)}$ on $[0,1]\times [0,1]$ form a Lebesgue null set and $g$ is bounded as well. By Lebesgue's integrability criterion for multiple Riemann-integrals (see, e.g., Theorem 14.5 in Apostol \citep{Apostol.1974}), $g$ is Riemann-integrable and we obtain convergence of the Riemann sums 
$$\frac{1}{b_n(b_n-1)} \sum_{1\leq j \neq k \leq b_n}\abs{\mu_{j\ell_n}-\mu_{k\ell_n}} = \frac{b_n}{b_n-1} \sum_{1\leq j \leq b_n}   \sum_{1\leq k \leq b_n}\abs{\mu\rbraces{\frac{j}{b_n}}-\mu\rbraces{\frac{k}{b_n}}} \frac{1}{b_n^2}$$
towards the desired integral. 
\end{proof}

\subsection{Behaviour of the test statistic under the hypothesis}
\subsubsection{A second approximation}
\begin{proof}[Proof of Proposition \ref{Prop: Approx U1 U2}]
It holds 
\begin{align*}
\sqrt{n}\abs{U_1(n)-U_2(n)}
\leq &\sqrt{n} \abs{U_1(n)-\frac{1}{b_n(b_n-1)} \sum_{1\leq j\neq k\leq b_n} \abs{\sigma_{j\ell_n}\tilde{S}_{nj}-\sigma_{k\ell_n}\tilde{S}_{nk}}}\\
+& \sqrt{n} \abs{U_2(n)-\frac{1}{b_n(b_n-1)} \sum_{1\leq j\neq k\leq b_n} \abs{\sigma_{j\ell_n}\tilde{S}_{nj}-\sigma_{k\ell_n}\tilde{S}_{nk}}}.
\end{align*}
For the first expression, it holds
\begin{align*}
&\sqrt{n} \E{\abs{U_1(n)-\frac{1}{b_n(b_n-1)} \sum_{1\leq j\neq k\leq b_n} \abs{\sigma_{j\ell_n}\tilde{S}_{nj}-\sigma_{k\ell_n}\tilde{S}_{nk}}}}\\
\leq & \frac{\sqrt{n}}{b_n(b_n-1)} \sum_{1\leq j\neq k\leq b_n} \E{\abs{\sigma_{j\ell_n}\frac{1}{\ell_n} \sum_{i=j\ell_n-m_n+1}^{j\ell_n} Y_i-\sigma_{k\ell_n}\frac{1}{\ell_n} \sum_{i=k\ell_n-m_n+1}^{k\ell_n} Y_i}}\\
\leq & 2 \sigma_{\sup} \sqrt{n} \E{\abs{\frac{1}{\ell_n} \sum_{i=\ell_n-m_n+1}^{\ell_n} Y_i}} 
\leq 2 \sigma_{\sup} \frac{\sqrt{n}\sqrt{m_n}}{\ell_n} \norm{\frac{1}{\sqrt{m_n}} \sum_{i=\ell_n-m_n+1}^{\ell_n} Y_i}_2
 \leq C \frac{\sqrt{b_n m_n}}{\sqrt{\ell_n}}.
\end{align*}
Turning towards the second expression, we obtain
\begin{align*}
&\sqrt{n} \E{\abs{U_2(n)-\frac{1}{b_n(b_n-1)} \sum_{1\leq j\neq k\leq b_n} \abs{\sigma_{j\ell_n}\tilde{S}_{nj}-\sigma_{k\ell_n}\tilde{S}_{nk}}}}\\
\leq &  \frac{\sqrt{n}}{b_n(b_n-1)} \sum_{1\leq j\neq k\leq b_n} \E{\abs{\sigma_{j\ell_n}(\tilde{S}_{nj}-\tilde{S}_{nj}')-\sigma_{k\ell_n}(\tilde{S}_{nk}-\tilde{S}_{nk}')}}\\
\leq & 2\sigma_{\sup} \sqrt{n}\E{\abs{\tilde{S}_{n1}-\tilde{S}_{n1}'}}
\leq  8\sigma_{\sup} \sqrt{n}  \int_0^{\alpha_Y(m_n)} Q_{\abs{\tilde{S}_{n1}}}(u)\mathrm{d}u\\
\leq & 8\sigma_{\sup} \sqrt{n}  \rbraces{\int_0^{1} \1_{\alpha_Y(m_n)}^2(u)\mathrm{d}u}^{1/2} \cdot \rbraces{\int_0^{1} Q_{\abs{\tilde{S}_{n1}}}^2(u)\mathrm{d}u}^{1/2}\\
= & 8 \sigma_{\sup} \sqrt{n} \sqrt{\alpha_Y(m_n)}\norm{\tilde{S}_{n1}}_2\\
=  & 8\sigma_{\sup} \sqrt{b_n} \sqrt{\alpha_Y(m_n)}\norm{\frac{1}{\sqrt{\ell_n}}\sum_{i=1}^{\ell_n-m_n}Y_i}_2\leq C \sqrt{b_n\alpha_Y(m_n)},
\end{align*}
where the third inequality follows by equation \eqref{Eq: Peligrad bound for expectation} 
and the fourth one is  an application of the Cauchy-Schwarz inequality.
\end{proof}
\subsubsection{Law of large numbers}

\begin{proof}[Proof of Proposition \ref{Prop: LLN U2}]
We will prove in Proposition \ref{Prop: Replacing centring term} below that $$\sqrt{b_n} \abs{\E{\frac{\sqrt{\ell_n}}{\kappa_Y} U_2(n)} - \int_0^1 \int_0^1 \sqrt{\sigma^2(x)+\sigma^2(y)}\mathrm{d}x\mathrm{d}y \cdot \sqrt{\frac{2}{\pi}}}\rightarrow 0.$$ 
A close look at the proof there reveals that 
$$\E{\frac{\sqrt{\ell_n}}{\kappa_Y} U_2(n)} \rightarrow \int_0^1 \int_0^1 \sqrt{\sigma^2(x)+\sigma^2(y)}\mathrm{d}x\mathrm{d}y \cdot \sqrt{\frac{2}{\pi}}$$
already holds under the weaker assumptions made here. 
It remains to prove that $\Var{\frac{\sqrt{\ell_n}}{\kappa_Y}U_2(n)}\rightarrow 0$. It holds
{\allowdisplaybreaks\begin{align*}
& \Var{\frac{\sqrt{\ell_n}}{\kappa_Y}U_2(n)} \\
= & \frac{1}{b_n^2(b_n-1)^2} \sum_{1\leq j_1\neq k_1\leq b_n}\sum_{1\leq j_2 \neq k_2\leq b_n} \cov{\frac{\sqrt{\ell_n}}{\kappa_Y}\abs{\sigma_{j_1\ell_n}\tilde{S}_{nj_1}'-\sigma_{k_1\ell_n}\tilde{S}_{nk_1}'}}{\frac{\sqrt{\ell_n}}{\kappa_Y}\abs{\sigma_{j_2\ell_n}\tilde{S}_{nj_2}'-\sigma_{k_2\ell_n}\tilde{S}_{nk_2}'}}\\
= & \frac{2}{b_n^2(b_n-1)^2} \sum_{1\leq j_1\neq k_1\leq b_n} \Var{\frac{\sqrt{\ell_n}}{\kappa_Y}\abs{\sigma_{j_1\ell_n}\tilde{S}_{nj_1}'-\sigma_{k_1\ell_n}\tilde{S}_{nk_1}'}}\\
 & + \frac{4}{b_n^2(b_n-1)^2} \sum_{1\leq j_1\neq k_1\leq b_n}\sum_{\substack{1\leq k_2\leq b_n,\\ k_2\notin \{j_1, k_1\}}} \cov{\frac{\sqrt{\ell_n}}{\kappa_Y}\abs{\sigma_{j_1\ell_n}\tilde{S}_{nj_1}'-\sigma_{k_1\ell_n}\tilde{S}_{nk_1}'}}{\frac{\sqrt{\ell_n}}{\kappa_Y}\abs{\sigma_{j_1\ell_n}\tilde{S}_{nj_1}'-\sigma_{k_2\ell_n}\tilde{S}_{nk_2}'}}.
\end{align*}}
Turning towards the sum of variances, we obtain 
{\allowdisplaybreaks\begin{align*}
& \frac{1}{b_n^2(b_n-1)^2} \sum_{1\leq j_1\neq k_1\leq b_n} \Var{\frac{\sqrt{\ell_n}}{\kappa_Y}\abs{\sigma_{j_1\ell_n}\tilde{S}_{nj_1}'-\sigma_{k_1\ell_n}\tilde{S}_{nk_1}'}}\\
 \leq & \frac{1}{b_n^2(b_n-1)^2} \sum_{1\leq j_1\neq k_1\leq b_n} \E{\rbraces{\frac{\sqrt{\ell_n}}{\kappa_Y}\abs{\sigma_{j_1\ell_n}\tilde{S}_{nj_1}'-\sigma_{k_1\ell_n}\tilde{S}_{nk_1}'}}^2}
 \leq  \frac{4\sigma_{\sup}^2}{b_n(b_n-1)} \E{\abs{\frac{\sqrt{\ell_n}}{\kappa_Y}\tilde{S}_{n1}'}^2}\\
 =& \frac{4\sigma_{\sup}^2}{b_n(b_n-1)} \E{\abs{\frac{1}{\sqrt{\ell_n}\kappa_Y}\sum_{i=1}^{\ell_n-m_n}Y_i}^2}=O\rbraces{\frac{1}{b_n(b_n-1)}}.
\end{align*}}
Quite similarly, an application of the Cauchy-Schwarz inequality yields
{\allowdisplaybreaks \begin{align*}
 & \frac{1}{b_n^2(b_n-1)^2} \sum_{1\leq j_1\neq k_1\leq b_n}\sum_{\substack{1\leq k_2\leq b_n,\\ k_2\notin \{j_1, k_1\}}} \cov{\frac{\sqrt{\ell_n}}{\kappa_Y}\abs{\sigma_{j_1\ell_n}\tilde{S}_{nj_1}'-\sigma_{k_1\ell_n}\tilde{S}_{nk_1}'}}{\frac{\sqrt{\ell_n}}{\kappa_Y}\abs{\sigma_{j_1\ell_n}\tilde{S}_{nj_1}'-\sigma_{k_2\ell_n}\tilde{S}_{nk_2}'}}\\
 \leq & \frac{1}{b_n^2(b_n-1)^2} \sum_{1\leq j_1\neq k_1\leq b_n}\sum_{\substack{1\leq k_2\leq b_n,\\ k_2\notin \{j_1, k_1\}}} \E{\rbraces{\frac{\sqrt{\ell_n}}{\kappa_Y}\abs{\sigma_{j_1\ell_n}\tilde{S}_{nj_1}'-\sigma_{k_1\ell_n}\tilde{S}_{nk_1}'}} \cdot\rbraces{\frac{\sqrt{\ell_n}}{\kappa_Y}\abs{\sigma_{j_1\ell_n}\tilde{S}_{nj_1}'-\sigma_{k_2\ell_n}\tilde{S}_{nk_2}'}}}\\ 
 \leq & \frac{1}{b_n^2(b_n-1)^2} \sum_{1\leq j_1\neq k_1\leq b_n}\sum_{\substack{1\leq k_2\leq b_n,\\ k_2\notin \{j_1, k_1\}}} \norm{\frac{\sqrt{\ell_n}}{\kappa_Y}\rbraces{\sigma_{j_1\ell_n}\tilde{S}_{nj_1}'-\sigma_{k_1\ell_n}\tilde{S}_{nk_1}'}}_2 \cdot \norm{\frac{\sqrt{\ell_n}}{\kappa_Y}\rbraces{\sigma_{j_1\ell_n}\tilde{S}_{nj_1}'-\sigma_{k_2\ell_n}\tilde{S}_{nk_2}'}}_2\\
 \leq & \frac{4 \sigma_{\sup}^2 (b_n-2)}{b_n (b_n-1)} \E{\abs{\frac{1}{\sqrt{\ell_n}\kappa_Y}\sum_{i=1}^{\ell_n-m_n}Y_i}^2}=O\rbraces{\frac{b_n-2}{b_n(b_n-1)}}.
\end{align*}}
\end{proof}
\subsubsection{Central limit theorem}
\begin{proof}[Proof of Proposition \ref{Prop: CLT U2}]
Note that
\begin{align*}
\frac{\sqrt{\ell_n}}{\kappa_Y} U_2(n)= \frac{1}{b_n(b_n-1)} \sum_{1\leq j\neq k\leq b_n} \abs{\frac{\sqrt{\ell_n}}{\kappa_Y}\sigma_{j\ell_n}\tilde{S}_{nj}'-\frac{\sqrt{\ell_n}}{\kappa_Y}\sigma_{k\ell_n}\tilde{S}_{nk}'}
\end{align*}
is a U-statistic of a row-wise non-stationary but  independent triangular array. We intend to split $U_2(n)$ with kernel $h(x,y)=\abs{x-y}$ up via the Hoeffding-decomposition. We define
\begin{align*}
\theta_{jk}^{(n)}=\theta_{kj}^{(n)}&:=\E{h\rbraces{\frac{\sqrt{\ell_n}}{\kappa_Y}\sigma_{j\ell_n}\tilde{S}_{nj}', \frac{\sqrt{\ell_n}}{\kappa_Y}\sigma_{k\ell_n}\tilde{S}_{nk}'}},\\
h_{1,jk}^{(n)}(x)&:=\E{h\rbraces{x, \frac{\sqrt{\ell_n}}{\kappa_Y}\sigma_{k\ell_n}\tilde{S}_{nk}'}}-\theta_{jk}^{(n)},\\
\overline{h_{1,j}^{(n)}}\rbraces{x}&:=\frac{1}{b_n-1} \sum_{\substack{1\leq k\leq b_n\\k\neq j}} h_{1,jk}^{(n)}\rbraces{x},\\
h_{2,jk}^{(n)}(x,y)&:=h(x,y)-h_{1,jk}^{(n)}(x)-h_{1,kj}^{(n)}(y)-\theta_{jk}^{(n)}.
\end{align*}
We thus obtain 
\begin{align*}
&\sqrt{b_n}\rbraces{\frac{\sqrt{\ell_n}}{\kappa_Y}U_2(n) - \frac{\sqrt{\ell_n}}{\kappa_Y}\E{U_2(n)} } \\
=&\frac{\sqrt{b_n}}{b_n(b_n-1)}   \sum_{1\leq j\neq k\leq b_n}\rbraces{ \abs{\frac{\sqrt{\ell_n}}{\kappa_Y}\sigma_{j\ell_n}\tilde{S}_{nj}'-\frac{\sqrt{\ell_n}}{\kappa_Y}\sigma_{k\ell_n}\tilde{S}_{nk}'}-  \theta_{jk}^{(n)}}\\
= & \frac{\sqrt{b_n}}{b_n(b_n-1)}   \sum_{1\leq j\neq k\leq b_n} \rbraces{h_{1,jk}^{(n)}\rbraces{\frac{\sqrt{\ell_n}}{\kappa_Y}\sigma_{j\ell_n}\tilde{S}_{nj}'}+h_{1,kj}^{(n)}\rbraces{\frac{\sqrt{\ell_n}}{\kappa_Y}\sigma_{k\ell_n}\tilde{S}_{nk}'}+h_{2,jk}^{(n)}\rbraces{\frac{\sqrt{\ell_n}}{\kappa_Y}\sigma_{j\ell_n}\tilde{S}_{nj}',\frac{\sqrt{\ell_n}}{\kappa_Y}\sigma_{k\ell_n}\tilde{S}_{nk}'}}\\
= & \frac{2}{\sqrt{b_n}} \sum_{1\leq j\leq b_n} \overline{h_{1,j}^{(n)}}\rbraces{\frac{\sqrt{\ell_n}}{\kappa_Y}\sigma_{j\ell_n}\tilde{S}_{nj}'} +  \frac{\sqrt{b_n}}{b_n(b_n-1)}   \sum_{1\leq j\neq k\leq b_n}h_{2,jk}^{(n)}\rbraces{\frac{\sqrt{\ell_n}}{\kappa_Y}\sigma_{j\ell_n}\tilde{S}_{nj}',\frac{\sqrt{\ell_n}}{\kappa_Y}\sigma_{k\ell_n}\tilde{S}_{nk}'}.
\end{align*}
We will call the first sum the linear part and the second sum the degenerate part of the Hoeffding-decomposition. In particular, we will prove that the linear part converges towards the desired normal limit and that the degenerate part, consisting of a degenerate U-statistic, converges towards zero in $L^2$. An application of Slutzky's lemma then concludes the proof. 
Starting with the linear part, note that it is a sum of independent and centred, though not identically distributed, random variables. We intend to apply Lyapunov's central limit theorem and define
$$Y_{nj}:= \frac{2}{\sqrt{b_n}} \overline{h_{1,j}^{(n)}}\rbraces{\frac{\sqrt{\ell_n}}{\kappa_Y}\sigma_{j\ell_n}\tilde{S}_{nj}'} \cdot \rbraces{\frac{4}{b_n}\sum_{k=1}^{b_n} \Var{\overline{h_{1,k}^{(n)}}\rbraces{\frac{\sqrt{\ell_n}}{\kappa_Y}\sigma_{k\ell_n}\tilde{S}_{nk}'} } } ^{-1/2}.$$
Then the $Y_{nj}$'s for $1\leq j \leq b_n$ are likewise independent 
 with mean zero and $\sum_{j=1}^{b_n} \Var{Y_{nj}}=1$. We have to verify that for some $\eta>0$, it holds
$$\limn \sum_{j=1}^{b_n} \E{\abs{Y_{nj}}^{2+\eta}}=0.$$
We have 
\begin{align}
\label{Eq: Proof CLT Ljapunov expectation}
& \sum_{j=1}^{b_n} \E{\abs{Y_{nj}}^{2+\eta}}\\
\nonumber =&b_n^{-\eta/2}  \rbraces{\frac{4}{b_n} \sum_{k=1}^{b_n}\Var{\overline{h_{1,k}^{(n)}}\rbraces{\frac{\sqrt{\ell_n}}{\kappa_Y}\sigma_{k\ell_n}\tilde{S}_{nk}'} }}^{-(1+\eta/2)} \rbraces{ \frac{2^{2+\eta}}{b_n} \sum_{j=1}^{b_n} \E{\abs{\overline{h_{1,j}^{(n)}}\rbraces{\frac{\sqrt{\ell_n}}{\kappa_Y}\sigma_{j\ell_n}\tilde{S}_{nj}'}}^{2+\eta}}},
\end{align}
which converges towards zero if we can prove that the last two terms in \eqref{Eq: Proof CLT Ljapunov expectation} are bounded. To show the boundedness for the latter of the two, we will bound the expectation $\E{\abs{\overline{h_{1,j}^{(n)}}\rbraces{\frac{\sqrt{\ell_n}}{\kappa_Y}\sigma_{j\ell_n}\tilde{S}_{nj}'}}^{2+\eta}}$ by a constant independent of $j$.
It holds by the $c_r$-inequality
\begin{align*}
&\E{\abs{\overline{h_{1,j}^{(n)}}\rbraces{\frac{\sqrt{\ell_n}}{\kappa_Y}\sigma_{j\ell_n}\tilde{S}_{nj}'}}^{2+\eta}}\\
= & \E{\abs{\frac{1}{b_n-1} \sum_{\substack{1\leq k\leq b_n\\k\neq j}} \E{h\rbraces{\frac{\sqrt{\ell_n}}{\kappa_Y}\sigma_{j\ell_n}\tilde{S}_{nj}', \frac{\sqrt{\ell_n}}{\kappa_Y}\sigma_{k\ell_n}\tilde{S}_{nk}'}\middle|\tilde{S}_{nj}'}-\theta_{jk}^{(n)}}^{2+\eta}}\\
\leq & 2^{1+\eta} \E{\abs{\frac{1}{b_n-1} \sum_{\substack{1\leq k\leq b_n\\k\neq j}} \E{h\rbraces{\frac{\sqrt{\ell_n}}{\kappa_Y}\sigma_{j\ell_n}\tilde{S}_{nj}', \frac{\sqrt{\ell_n}}{\kappa_Y}\sigma_{k\ell_n}\tilde{S}_{nk}'}\middle|\tilde{S}_{nj}'}}^{2+\eta}} 
+ 2^{1+\eta} \rbraces{\frac{1}{b_n-1} \sum_{\substack{1\leq k\leq b_n\\k\neq j}}\theta_{jk}^{(n)}}^{2+\eta}.
\end{align*}
We have $\theta_{jk}^{(n)}\leq 2 \sigma_{\sup}\E{\frac{\sqrt{\ell_n}}{\kappa_Y}\abs{\tilde{S}_{n1}'}}$, which is finite. Thus, the second of these terms  is bounded. 
Considering the finiteness of the first term, it holds by independence and by stationarity of the $\tilde{S}_{nj}'$'s that
\begin{align*}
&   \E{\abs{ \frac{1}{b_n-1} \sum_{\substack{1\leq k\leq b_n\\k\neq j}}\E{\abs{\frac{\sqrt{\ell_n}}{\kappa_Y}\sigma_{j\ell_n}\tilde{S}_{nj}'- \frac{\sqrt{\ell_n}}{\kappa_Y}\sigma_{k\ell_n}\tilde{S}_{nk}'}\middle|\tilde{S}_{nj}'}}^{2+\eta}} \\
\leq & \sigma_{\sup}^{2+\eta}\E{\abs{\frac{\sqrt{\ell_n}}{\kappa_Y}\abs{\tilde{S}_{nj}'}+  \frac{1}{b_n-1} \sum_{\substack{1\leq k\leq b_n\\k\neq j}} \E{\frac{\sqrt{\ell_n}}{\kappa_Y}\abs{\tilde{S}_{nk}'}}}^{2+\eta}}
\leq  C \E{\abs{\frac{\sqrt{\ell_n}}{\kappa_Y}\tilde{S}_{n1}'}^{2+\eta}}.
\end{align*}
Choose $\eta=\delta/4$, then Theorem  \ref{Literature: Thm Yokoyama} by Yokoyama yields
\begin{align*}
& \E{\abs{\frac{\sqrt{\ell_n}}{\kappa_Y}\tilde{S}_{n1}'}^{2+\eta}}
= \rbraces{\sqrt{\ell_n}\kappa_Y}^{-(2+\eta)}\E{\abs{\sum_{i=1}^{\ell_n-m_n} Y_i}^{2+\eta}}\\
\leq &C \ \rbraces{\sqrt{\ell_n}\kappa_Y}^{-(2+\eta)} (\ell_n-m_n)^{1+\eta/2}  \leq C,
\end{align*}
which is applicable due to $\E{\abs{Y_1}^{2+\delta}}<\infty$ and since, by the polynomial decay of the mixing coefficients, 
\begin{align*}
\sum_{k=1}^{\infty} k^{\eta/2} \alpha(k)^{(\delta-\eta)/(2+\delta)}
= \sum_{k=1}^{\infty} k^{\delta/8} \alpha(k)^{(3/4)\delta/(2+\delta)}
\leq C \sum_{k=1}^{\infty} k^{\delta/8} k^{-(3/4)\rho(1+\delta)/\delta}
\end{align*}
 is finite for $0<\delta\leq 1$ and $\rho>1$.
Since this upper bound holds for every $j$, we have shown 
$$ \frac{2^{2+\eta}}{b_n} \sum_{j=1}^{b_n} \E{\abs{\overline{h_{1,j}^{(n)}}\rbraces{\frac{\sqrt{\ell_n}}{\kappa_Y}\sigma_{j\ell_n}\tilde{S}_{nj}'}}^{2+\eta}}\leq C.$$
Next, we will check that 
$$\tilde{\psi}_n^2:=4\frac{1}{b_n} \sum_{j=1}^{b_n}\Var{\overline{h_{1,j}^{(n)}}\rbraces{\frac{\sqrt{\ell_n}}{\kappa_Y}\sigma_{j\ell_n}\tilde{S}_{nj}'} } $$
in \eqref{Eq: Proof CLT Ljapunov expectation} is bounded away from zero. Since $\E{\overline{h_{1,j}^{(n)}}\rbraces{\frac{\sqrt{\ell_n}}{\kappa_Y}\sigma_{j\ell_n}\tilde{S}_{nj}'} }=0$ by independence and an application of Fubini's Theorem, it holds 
\begin{align*}
\tilde{\psi}_n^2
= & 4\frac{1}{b_n} \sum_{j=1}^{b_n}\E{\abs{\overline{h_{1,j}^{(n)}}\rbraces{\frac{\sqrt{\ell_n}}{\kappa_Y}\sigma_{j\ell_n}\tilde{S}_{nj}'} }^2}\\
= & 4\frac{1}{b_n} \sum_{j=1}^{b_n}\E{\abs{\frac{1}{b_n-1} \sum_{\substack{1\leq k\leq b_n\\k\neq j}} \rbraces{\E{h\rbraces{\frac{\sqrt{\ell_n}}{\kappa_Y}\sigma_{j\ell_n}\tilde{S}_{nj}', \frac{\sqrt{\ell_n}}{\kappa_Y}\sigma_{k\ell_n}\tilde{S}_{nk}'}\middle|\tilde{S}_{nj}'}-\theta_{jk}^{(n)}}}^2}\\
= &4 \frac{1}{b_n} \sum_{j=1}^{b_n}\E{\abs{\frac{1}{b_n-1} \sum_{\substack{1\leq k\leq b_n\\k\neq j}} \rbraces{\E{h\rbraces{\frac{\sqrt{\ell_n}}{\kappa_Y}\sigma_{j\ell_n}\tilde{S}_{n1}', \frac{\sqrt{\ell_n}}{\kappa_Y}\sigma_{k\ell_n}\tilde{S}_{n2}'}\middle|\tilde{S}_{n1}'}-\theta_{jk}^{(n)}}}^2}
\end{align*}
by stationarity and independence of the $\tilde{S}_{nj}'$. Our intention is to show that $\tilde{\psi}_n^2$ is asymptotically equivalent to 
$$\psi_n^2:= 4\frac{1}{b_n} \sum_{j=1}^{b_n}\E{\abs{\frac{1}{b_n-1} \sum_{\substack{1\leq k\leq b_n\\k\neq j}} \rbraces{\E{h\rbraces{\sigma_{j\ell_n}Z,\sigma_{k\ell_n}Z'}\middle|Z}-\E{h\rbraces{\sigma_{j\ell_n}Z, \sigma_{k\ell_n}Z'}}}}^2}$$
for two independent standard normal random variables $Z$ and $Z'$. 
 By independence, it holds $(\frac{\sqrt{\ell_n}}{\kappa_Y}\tilde{S}_{n1}', \frac{\sqrt{\ell_n}}{\kappa_Y}\tilde{S}_{n2}')\distConv (Z, Z')$ as $n\rightarrow \infty$. An application of Skorohods representation theorem yields the existence of a probability space $(\tilde{\Omega}, \tilde{\mathcal{F}}, \tilde{P})$, on which there are random variables $(T_{n1})_{n\in\N}$, $(T_{n2})_{n\in\N}$, $N$ and $N'$ such that for all $n\in \N$, $(T_{n1},T_{n2})\overset{\mathcal{D}}{=}(\frac{\sqrt{\ell_n}}{\kappa_Y}\tilde{S}_{n,1}', \frac{\sqrt{\ell_n}}{\kappa_Y}\tilde{S}_{n,2}')$ and $(N,N')\overset{\mathcal{D}}{=} (Z,Z')$ and such that  
$(T_{n1},T_{n2})$ converges almost surely to $(N,N')$ as $n\rightarrow \infty$. We will now work with the newly defined random variables $T_{n1}, T_{n2}, N$ and $N'$ to prove that for each summand $j$, the $L^2$-distance of the respective expression within the expectations in $\tilde{\psi}_n$ and $\psi_n$ converges towards zero, i.e. we examine
\begin{align*}
& \left\lVert\frac{1}{b_n-1} \sum_{\substack{1\leq k\leq b_n\\k\neq j}} \rbraces{\E{h\rbraces{\sigma_{j\ell_n}T_{n1}, \sigma_{k\ell_n}T_{n2}}\middle|T_{n1}}-\theta_{jk}^{(n)}}\right.\\
& \left. -  \frac{1}{b_n-1} \sum_{\substack{1\leq k\leq b_n\\k\neq j}}\rbraces{\E{h\rbraces{\sigma_{j\ell_n}N,\sigma_{k\ell_n}N'}\middle|N}-\E{h\rbraces{\sigma_{j\ell_n}N, \sigma_{k\ell_n}N'}}}\right\rVert_2\\
\leq  &  \frac{1}{b_n-1} \sum_{\substack{1\leq k\leq b_n\\k\neq j}} \left\lVert  \E{h\rbraces{\sigma_{j\ell_n}T_{n1}, \sigma_{k\ell_n}T_{n2}}\middle|T_{n1}}-\E{h\rbraces{\sigma_{j\ell_n}N,\sigma_{k\ell_n}N'}\middle|N} \right\rVert_2\\
& + \frac{1}{b_n-1} \sum_{\substack{1\leq k\leq b_n\\k\neq j}}\abs{ \theta_{jk}^{(n)}-\E{h\rbraces{\sigma_{j\ell_n}N, \sigma_{k\ell_n}N'}}}.
\end{align*}
The convergence of the second, deterministic expression can be shown as in Proposition \ref{Prop: Replacing centring term}, and we obtain a bound independent of $j$ and $k$ of order $O((\ell_n-m_n)^{-(\delta/2)(\rho-1)/(\rho+1)} \log(C (\ell_n-m_n)))$. 
To prove convergence of the $L^2$-distance towards zero, define the functions 
$$f_{jk}^{(n)}(x):= \E{\abs{\sigma_{j\ell_n}x-\sigma_{k\ell_n}T_{n2}}} \quad \text{and} \quad f_{jk}(x):=\E{\abs{\sigma_{j\ell_n}x-\sigma_{k\ell_n}N'}}.$$
Then it holds for every $x,y\in\R$,
\begin{equation*}
\abs{f_{jk}^{(n)}(x)-f_{jk}(x)}\leq \sigma_{\sup} \E{\abs{T_{n2}-N'}}
\end{equation*}
as well as 
\begin{equation*}
\abs{f_{jk}(x)-f_{jk}(y)}\leq \sigma_{\sup} \abs{x-y}
\end{equation*}
and hence, 
\begin{align*}
& \left\lVert  \E{h\rbraces{\sigma_{j\ell_n}T_{n1}, \sigma_{k\ell_n}T_{n2}}\middle|T_{n1}}-\E{h\rbraces{\sigma_{j\ell_n}N,\sigma_{k\ell_n}N'}\middle|N} \right\rVert_2 = \left\lVert f_{jk}^{(n)}(T_{n1})-f_{jk}(N)\right\rVert_2\\
\leq & \left\lVert f_{jk}^{(n)}(T_{n1})- f_{jk}(T_{n1})\right\rVert_2+\left\lVert f_{jk}(T_{n1})-f_{jk}(N)\right\rVert_2
\leq   \sigma_{\sup} \rbraces{ \E{\abs{T_{n2}-N'}}+  \left\lVert T_{n1}-N\right\rVert_2}\\
\leq & 2 \sigma_{\sup}  \left\lVert T_{n1}-N\right\rVert_2. 
\end{align*}
In particular, note that the above bound is likewise independent of both $j$ and $k$. By construction, it now holds $T_{n1}\rightarrow N$ almost surely. Above, we have already proven that $\sup_{n\in\N} \E{\abs{T_{n1}}^{2+\eta}}=\sup_{n\in\N} \E{\abs{\frac{\sqrt{\ell_n}}{\kappa_Y}\tilde{S}_{n1}'}^{2+\eta}}<\infty $ for some $\eta>0$ such that $(T_{n1}^2)_{n\in\N}$ is uniformly integrable and hence
$$T_{n1}\overset{L^2}{\longrightarrow} N \quad \text{as } n \rightarrow \infty.$$
Putting the pieces together, one obtains, using $a^2-b^2=(a+b)(a-b)$ for the expressions within the expectations and the Cauchy-Schwarz-inequality,
\begin{align*}
& \abs{\tilde{\psi}_n^2-\psi_n^2}\\
\leq &4 \frac{1}{b_n} \sum_{j=1}^{b_n} \left\lVert\frac{1}{b_n-1} \sum_{\substack{1\leq k\leq b_n\\k\neq j}} \rbraces{\E{h\rbraces{\sigma_{j\ell_n}T_{n1}, \sigma_{k\ell_n}T_{n2}}\middle|T_{n1}}-\theta_{jk}^{(n)}}\right.\\
& \left. + \frac{1}{b_n-1} \sum_{\substack{1\leq k\leq b_n\\k\neq j}}\rbraces{\E{h\rbraces{\sigma_{j\ell_n}N,\sigma_{k\ell_n}N'}\middle|N}-\E{h\rbraces{\sigma_{j\ell_n}N, \sigma_{k\ell_n}N'}}}\right\rVert_2\\
& \cdot \left\lVert\frac{1}{b_n-1} \sum_{\substack{1\leq k\leq b_n\\k\neq j}} \rbraces{\E{h\rbraces{\sigma_{j\ell_n}T_{n1}, \sigma_{k\ell_n}T_{n2}}\middle|T_{n1}}-\theta_{jk}^{(n)}}\right.\\
& \left. -  \frac{1}{b_n-1} \sum_{\substack{1\leq k\leq b_n\\k\neq j}}\rbraces{\E{h\rbraces{\sigma_{j\ell_n}N,\sigma_{k\ell_n}N'}\middle|N}-\E{h\rbraces{\sigma_{j\ell_n}N, \sigma_{k\ell_n}N'}}}\right\rVert_2.
\end{align*}
We verified above that the second $L^2$-distance can be bounded independently of $j$ and converges towards zero, such that $\abs{\tilde{\psi}_n^2-\psi_n^2}\rightarrow 0$ if the first $L^2$-distance is bounded independent of $j$, which holds due to 
\begin{align*}
& \left\lVert\frac{1}{b_n-1} \sum_{\substack{1\leq k\leq b_n\\k\neq j}} \rbraces{\E{h\rbraces{\sigma_{j\ell_n}T_{n1}, \sigma_{k\ell_n}T_{n2}}\middle|T_{n1}}-\theta_{jk}^{(n)}}\right.\\
& \left. + \frac{1}{b_n-1} \sum_{\substack{1\leq k\leq b_n\\k\neq j}}\rbraces{\E{h\rbraces{\sigma_{j\ell_n}N,\sigma_{k\ell_n}N'}\middle|N}-\E{h\rbraces{\sigma_{j\ell_n}N, \sigma_{k\ell_n}N'}}}\right\rVert_2\\
\leq & \frac{1}{b_n-1} \sum_{\substack{1\leq k\leq b_n\\k\neq j}} \left(\norm{\E{\abs{\sigma_{j\ell_n}T_{n1}-\sigma_{k\ell_n}T_{n2}}\middle|T_{n1}}}_2 + \E{\abs{\sigma_{j\ell_n}T_{n1}-\sigma_{k\ell_n}T_{n2}}} \right.\\
 & \left. + \norm{\E{\abs{\sigma_{j\ell_n}N-\sigma_{k\ell_n}N'}\middle|N}}_2+  \E{\abs{\sigma_{j\ell_n}N-\sigma_{k\ell_n}N'}} \right)\\
 \leq & 4 \sigma_{\sup} \rbraces{\norm{T_{n1}}_2+\norm{N}_2}\leq 4 \sigma_{\sup} \rbraces{\sup_{n\in\N}\norm{T_{n1}}_2+\norm{N}_2}  \leq C.
\end{align*}
In the next step, we will prove that $\abs{\psi_n^2-\psi^2}\rightarrow 0$, where 
$$\psi^2:=4 \int_0^1 \E{\abs{\int_0^1 \E{\abs{\sigma(x)N-\sigma(y)N'}|N}-\E{\abs{\sigma(x)N-\sigma(y)N'}}\mathrm{d}y }^2} \mathrm{d}x$$
It holds
\begin{align*}
& \abs{\psi_n^2-\psi^2}\\
\leq & 4\left|\frac{1}{b_n} \sum_{j=1}^{b_n}\E{\abs{\frac{1}{b_n-1} \sum_{\substack{1\leq k\leq b_n\\k\neq j}} \rbraces{\E{\abs{\sigma_{j\ell_n}N-\sigma_{k\ell_n}N'}\middle|N}-\E{\abs{\sigma_{j\ell_n}N-\sigma_{k\ell_n}N'}}}}^2}\right.\\
& - \left. \frac{1}{b_n} \sum_{j=1}^{b_n}\E{\abs{\frac{1}{b_n} \sum_{\substack{1\leq k\leq b_n}} \rbraces{\E{\abs{\sigma_{j\ell_n}N-\sigma_{k\ell_n}N'}\middle|N}-\E{\abs{\sigma_{j\ell_n}N-\sigma_{k\ell_n}N'}}}}^2}  \right|\\
+& 4 \left|\frac{1}{b_n} \sum_{j=1}^{b_n}\E{\abs{\frac{1}{b_n} \sum_{\substack{1\leq k\leq b_n}} \rbraces{\E{\abs{\sigma_{j\ell_n}N-\sigma_{k\ell_n}N'}\middle|N}-\E{\abs{\sigma_{j\ell_n}N-\sigma_{k\ell_n}N'}}}}^2}\right.\\
& - \left. \int_0^1 \E{\abs{\int_0^1 \E{\abs{\sigma(x)N-\sigma(y)N'}|N}-\E{\abs{\sigma(x)N-\sigma(y)N'}}\mathrm{d}y }^2} \mathrm{d}x \right|.
\end{align*}
The first difference is of order $O(1/b_n)$. Defining $I_j:=(\frac{(j-1)\ell_n}{n}, \frac{j\ell_n}{n}]$ for $1\leq j\leq b_n$, the  second difference can be rewritten as
\begin{align*}
 & \left|\sum_{j=1}^{b_n}\int_{I_j}\E{\abs{ \sum_{\substack{1\leq k\leq b_n}} \int_{I_k} \E{\abs{\sigma_{j\ell_n}N-\sigma_{k\ell_n}N'}\middle|N}-\E{\abs{\sigma_{j\ell_n}N-\sigma_{k\ell_n}N'}}\mathrm{d}y}^2} \mathrm{d}x\right.\\
& - \left. \sum_{j=1}^{b_n}\int_{I_j} \E{\abs{ \sum_{\substack{1\leq k\leq b_n}} \int_{I_k}  \E{\abs{\sigma(x)N-\sigma(y)N'}|N}-\E{\abs{\sigma(x)N-\sigma(y)N'}}\mathrm{d}y }^2} \mathrm{d}x \right|\\
\leq & \sum_{j=1}^{b_n}\int_{I_j}\mathbb{E}\left. \bigg (\left.\bigg | \sum_{\substack{1\leq k\leq b_n}} \int_{I_k} \E{\abs{\sigma_{j\ell_n}N-\sigma_{k\ell_n}N'}\middle|N}- \E{\abs{\sigma(x)N-\sigma(y)N'}|N} \right. \right.\\
& \left. \left.+\E{\abs{\sigma(x)N-\sigma(y)N'}}-\E{\abs{\sigma_{j\ell_n}N-\sigma_{k\ell_n}N'}}\mathrm{d}y \right.\bigg | \right.  \\
&\left. \cdot \left. \bigg | \sum_{\substack{1\leq k\leq b_n}} \int_{I_k} \E{\abs{\sigma_{j\ell_n}N-\sigma_{k\ell_n}N'}\middle|N}+ \E{\abs{\sigma(x)N-\sigma(y)N'}|N} \right.\right.\\
& \left.\left.  -\E{\abs{\sigma(x)N-\sigma(y)N'}}-\E{\abs{\sigma_{j\ell_n}N-\sigma_{k\ell_n}N'}}\mathrm{d}y \right. \bigg |  \right. \bigg)\mathrm{d}x\\
\leq  & \sum_{j=1}^{b_n}\int_{I_j}\left. \bigg (\mathbb{E}\left.  \bigg (\left.  \bigg| \sum_{\substack{1\leq k\leq b_n}} \int_{I_k} \E{\abs{\sigma_{j\ell_n}N-\sigma_{k\ell_n}N'}\middle|N}- \E{\abs{\sigma(x)N-\sigma(y)N'}|N} \right. \right. \right.\\
& \left. \left.\left.+\E{\abs{\sigma(x)N-\sigma(y)N'}}-\E{\abs{\sigma_{j\ell_n}N- \sigma_{k\ell_n}N'}}\mathrm{d}y \right.  \bigg|^2 \right.  \bigg)\right.  \bigg)^{1/2}  \\
& \cdot \left.  \bigg(\mathbb{E}\left.  \bigg(\left.  \bigg| \sum_{\substack{1\leq k\leq b_n}} \int_{I_k} \E{\abs{\sigma_{j\ell_n}N-\sigma_{k\ell_n}N'}\middle|N}+ \E{\abs{\sigma(x)N-\sigma(y)N'}|N} \right.\right.\right.\\
& \left.\left.\left.  -\E{\abs{\sigma(x)N-\sigma(y)N'}}-\E{\abs{\sigma_{j\ell_n}N- \sigma_{k\ell_n}N'}}\mathrm{d}y \right.  \bigg|^2 \right.  \bigg)\right.  \bigg)^{1/2}\mathrm{d}x.
\end{align*}
We will consider the two expected values above separately. For the first one, we obtain by splitting the variance function into the Lipschitz-continuous and the jump part that
\begin{align*}
& \left.  \bigg \lVert \sum_{\substack{1\leq k\leq b_n}} \int_{I_k} \E{\abs{\sigma_{j\ell_n}N-\sigma_{k\ell_n}N'}\middle|N}- \E{\abs{\sigma(x)N-\sigma(y)N'}|N} \right.\\
 & \left. +\E{\abs{\sigma(x)N-\sigma(y)N'}}-\E{\abs{\sigma_{j\ell_n}N- \sigma_{k\ell_n}N'}}\mathrm{d}y\right.\bigg \rVert _2\\
 \leq &  \sum_{\substack{1\leq k\leq b_n}} \left\lVert\int_{I_k} \E{\abs{\sigma_{j\ell_n}N-\sigma_{k\ell_n}N'}\middle|N}- \E{\abs{\sigma(x)N-\sigma(y)N'}|N}\mathrm{d}y \right\rVert _2\\
 & +  \sum_{\substack{1\leq k\leq b_n}} \left\lVert\int_{I_k}\E{\abs{\sigma(x)N-\sigma(y)N'}}-\E{\abs{\sigma_{j\ell_n}N- \sigma_{k\ell_n}N'}}\mathrm{d}y\right\rVert _2\\
 \leq &  \sum_{\substack{1\leq k\leq b_n}} \left\lVert\int_{I_k} \abs{\sigma_{j\ell_n}-\sigma(x)} \cdot\abs{N}+ \abs{\sigma_{k\ell_n}-\sigma(y)}\E{\abs{N}}\mathrm{d}y \right\rVert _2\\
 & +  \sum_{\substack{1\leq k\leq b_n}} \left|\int_{I_k} \rbraces{\abs{\sigma_{j\ell_n}-\sigma(x)}+\abs{\sigma_{k\ell_n}-\sigma(y)}} \E{\abs{N}}\mathrm{d}y\right|\\
  \leq & \abs{\sigma_{j\ell_n}-\sigma(x)}\rbraces{\norm{N}_2+\E{\abs{N}}}+2\E{\abs{N}}\sum_{\substack{1\leq k\leq b_n}} \int_{I_k}\abs{\sigma_{k\ell_n}-\sigma(y)} \mathrm{d}y \\
 \leq & \frac{1}{b_n}\rbraces{\norm{N}_2+3\E{\abs{N}}}+  \abs{\sigma_{J,j\ell_n}-\sigma_J(x)}\rbraces{\norm{N}_2+\E{\abs{N}}}+2\E{\abs{N}} \frac{k_\sigma \Delta_\sigma}{b_n} 
.\end{align*}
For the second expectation, we similarly obtain 
\begin{align*}
& \left.  \bigg \lVert \sum_{\substack{1\leq k\leq b_n}} \int_{I_k} \E{\abs{\sigma_{j\ell_n}N-\sigma_{k\ell_n}N'}\middle|N}+ \E{\abs{\sigma(x)N-\sigma(y)N'}|N} \right.\\
& \left.  -\E{\abs{\sigma(x)N-\sigma(y)N'}}-\E{\abs{\sigma_{j\ell_n}N- \sigma_{k\ell_n}N'}}\mathrm{d}y \right.  \bigg \rVert_2\\
 \leq & \sum_{\substack{1\leq k\leq b_n}} \left\lVert\frac{2\sigma_{\sup}}{b_n}\rbraces{\abs{N}+\E{\abs{N}}} \right\rVert _2 +  \sum_{\substack{1\leq k\leq b_n}}\frac{4\sigma_{\sup}}{b_n} \E{\abs{N}}\\
\leq & \rbraces{2\sigma_{\sup}\E{N^2}^{1/2}+ 6\sigma_{\sup}\E{\abs{N}}} \leq C.
\end{align*}
Thus, 
\begin{align*}
& \abs{\psi_n^2-\psi^2} \\
\leq & O\rbraces{\frac{1}{b_n}}+ 4\sum_{j=1}^{b_n} \int_{I_j}  \rbraces{\frac{\norm{N}_2+3\E{\abs{N}}+2\E{\abs{N}}k_\sigma \Delta_\sigma}{b_n}+  \abs{\sigma_{J,j\ell_n}-\sigma_J(x)}\rbraces{\norm{N}_2+\E{\abs{N}}}} \cdot C\mathrm{d}x \\
\leq & O\rbraces{\frac{1}{b_n}}+ C\frac{(C+k_\sigma\Delta_\sigma) \cdot \rbraces{\norm{N}_2+\E{\abs{N}}}}{b_n}=O\rbraces{\frac{1}{b_n}}.
\end{align*}
Since $\psi^2>0$ as long as $\sigma$ is not almost surely equal to zero and since $\tilde{\psi}_n\rightarrow \psi$, we have $\tilde{\psi}_n>\psi/2$ for $n$ large enough, which proves the boundedness of the variance term $ \rbraces{\frac{4}{b_n} \sum_{k=1}^{b_n}\Var{\overline{h_{1,k}^{(n)}}\rbraces{\frac{\sqrt{\ell_n}}{\kappa_Y}\sigma_{k\ell_n}\tilde{S}_{nk}'} }}^{-(1+\eta/2)} $.
Hence, we can apply Lyapunov's central limit theorem and obtain
$$\frac{2}{\sqrt{b_n}} \sum_{1\leq j\leq b_n} \overline{h_{1,j}^{(n)}}\rbraces{\frac{\sqrt{\ell_n}}{\kappa_Y}\sigma_{j\ell_n}\tilde{S}_{nj}'} =  \sum_{1\leq j\leq b_n}  Y_{nj} \cdot \tilde{\psi}_n \distConv \NoD{0}{\psi^2}.$$

Turning towards the degenerate term, since $\E{h_{2,jk}^{(n)}\rbraces{\frac{\sqrt{\ell_n}}{\kappa_Y}\sigma_{j\ell_n}\tilde{S}_{nj}',\frac{\sqrt{\ell_n}}{\kappa_Y}\sigma_{k\ell_n}\tilde{S}_{nk}'}}=0$, it remains to prove that its variance converges towards zero. Due to the summands being pairwise uncorrelated, the variance of the degenerate part is given by
\begin{align*}
\frac{1}{b_n(b_n-1)^2}   \sum_{1\leq j\neq k\leq b_n}\Var{h_{2,jk}^{(n)}\rbraces{\frac{\sqrt{\ell_n}}{\kappa_Y}\sigma_{j\ell_n}\tilde{S}_{nj}',\frac{\sqrt{\ell_n}}{\kappa_Y}\sigma_{k\ell_n}\tilde{S}_{nk}'}}=O\rbraces{\frac{1}{b_n-1}}
\end{align*}
since each summand can be bounded by
{\allowdisplaybreaks
\begin{align*}
&\Var{ h_{2,jk}^{(n)}\rbraces{\frac{\sqrt{\ell_n}}{\kappa_Y}\sigma_{j\ell_n}\tilde{S}_{nj}',\frac{\sqrt{\ell_n}}{\kappa_Y}\sigma_{k\ell_n}\tilde{S}_{nk}'}}
=  \E{h_{2,jk}^{(n)}\rbraces{\frac{\sqrt{\ell_n}}{\kappa_Y}\sigma_{j\ell_n}\tilde{S}_{nj}',\frac{\sqrt{\ell_n}}{\kappa_Y}\sigma_{k\ell_n}\tilde{S}_{nk}'}^2}\\
= & \E{ \rbraces{h\rbraces{\frac{\sqrt{\ell_n}}{\kappa_Y}\sigma_{j\ell_n}\tilde{S}_{nj}',\frac{\sqrt{\ell_n}}{\kappa_Y}\sigma_{k\ell_n}\tilde{S}_{nk}'}-\theta_{jk}^{(n)}-h_{1,jk}^{(n)}\rbraces{\frac{\sqrt{\ell_n}}{\kappa_Y}\sigma_{j\ell_n}\tilde{S}_{nj}'}-h_{1,kj}^{(n)}\rbraces{\frac{\sqrt{\ell_n}}{\kappa_Y}\sigma_{k\ell_n}\tilde{S}_{nk}'} }^2}\\
= & \E{h\rbraces{\frac{\sqrt{\ell_n}}{\kappa_Y}\sigma_{j\ell_n}\tilde{S}_{nj}',\frac{\sqrt{\ell_n}}{\kappa_Y}\sigma_{k\ell_n}\tilde{S}_{nk}'}^2}- \rbraces{\theta_{jk}^{(n)}}^2\\
 & - 2\E{h\rbraces{\frac{\sqrt{\ell_n}}{\kappa_Y}\sigma_{j\ell_n}\tilde{S}_{nj}',\frac{\sqrt{\ell_n}}{\kappa_Y}\sigma_{k\ell_n}\tilde{S}_{nk}'} h_{1,jk}^{(n)}\rbraces{\frac{\sqrt{\ell_n}}{\kappa_Y}\sigma_{j\ell_n}\tilde{S}_{nj}'}} +  \E{h_{1,jk}^{(n)}\rbraces{\frac{\sqrt{\ell_n}}{\kappa_Y}\sigma_{j\ell_n}\tilde{S}_{nj}'}^2}\\
& - 2\E{h\rbraces{\frac{\sqrt{\ell_n}}{\kappa_Y}\sigma_{j\ell_n}\tilde{S}_{nj}',\frac{\sqrt{\ell_n}}{\kappa_Y}\sigma_{k\ell_n}\tilde{S}_{nk}'} h_{1,kj}^{(n)}\rbraces{\frac{\sqrt{\ell_n}}{\kappa_Y}\sigma_{k\ell_n}\tilde{S}_{nk}'}}+  \E{h_{1,kj}^{(n)}\rbraces{\frac{\sqrt{\ell_n}}{\kappa_Y}\sigma_{k\ell_n}\tilde{S}_{nk}'}^2} \\
= &  \E{h\rbraces{\frac{\sqrt{\ell_n}}{\kappa_Y}\sigma_{j\ell_n}\tilde{S}_{nj}',\frac{\sqrt{\ell_n}}{\kappa_Y}\sigma_{k\ell_n}\tilde{S}_{nk}'}^2}- \rbraces{\theta_{jk}^{(n)}}^2\\
& -\E{h_{1,jk}^{(n)}\rbraces{\frac{\sqrt{\ell_n}}{\kappa_Y}\sigma_{j\ell_n}\tilde{S}_{nj}'}^2}-  \E{h_{1,kj}^{(n)}\rbraces{\frac{\sqrt{\ell_n}}{\kappa_Y}\sigma_{k\ell_n}\tilde{S}_{nk}'}^2}  \\
\leq & \E{h\rbraces{\frac{\sqrt{\ell_n}}{\kappa_Y}\sigma_{j\ell_n}\tilde{S}_{nj}',\frac{\sqrt{\ell_n}}{\kappa_Y}\sigma_{k\ell_n}\tilde{S}_{nk}'}^2}
= \E{\abs{\frac{\sqrt{\ell_n}}{\kappa_Y}\sigma_{j\ell_n}\tilde{S}_{nj}'-\frac{\sqrt{\ell_n}}{\kappa_Y}\sigma_{k\ell_n}\tilde{S}_{nk}'}^2}\\
\leq & 4 \sigma_{\sup}^2 \E{\rbraces{\frac{\sqrt{\ell_n}}{\kappa_Y}\tilde{S}_{n1}'}^2}\leq C.
\end{align*}}

\end{proof}

\subsubsection{Replacing the centring term}

\begin{proof}[Proof of Proposition \ref{Prop: Replacing centring term}]
We split the difference up via
\begin{align}
\label{eq: Prop centring term, main ineq}
& \sqrt{b_n}\abs{\E{\frac{\sqrt{\ell_n}}{\kappa_Y}U_2(n)}-  \int_0^1 \int_0^1 \sqrt{\sigma^2(x)+\sigma^2(y)}\mathrm{d}x\mathrm{d}y \cdot \E{\abs{Z}}}\\
\nonumber \leq & \sqrt{b_n}\abs{\E{\frac{\sqrt{\ell_n}}{\kappa_Y}U_2(n)}- \frac{1}{b_n(b_n-1)}\sum_{1\leq j \neq k\leq b_n}  \sqrt{\sigma_{j\ell_n}^2+\sigma_{k\ell_n}^2}\E{\abs{Z}} } \\
\nonumber & + \sqrt{b_n}\abs{\frac{1}{b_n(b_n-1)}\sum_{1\leq j \neq k\leq b_n}  \sqrt{\sigma_{j\ell_n}^2+\sigma_{k\ell_n}^2}\E{\abs{Z}} - \int_0^1 \int_0^1 \sqrt{\sigma^2(x)+\sigma^2(y)}\mathrm{d}x\mathrm{d}y \cdot \E{\abs{Z}}}
\end{align}
for an independent standard normal random variable $Z$. First, consider the second expression in \eqref{eq: Prop centring term, main ineq}. We need to prove that 
\begin{align*}
& \sqrt{b_n}\abs{\frac{1}{b_n(b_n-1)}\sum_{1\leq j \neq k\leq b_n}  \sqrt{\sigma_{k\ell_n}^2+\sigma_{j\ell_n}^2} - \int_0^1 \int_0^1 \sqrt{\sigma^2(x)+\sigma^2(y)}\mathrm{d}x\mathrm{d}y} \\
\leq & \sqrt{b_n}\abs{\frac{b_n}{(b_n-1)}\sum_{1\leq j \leq b_n}\sum_{1\leq k \leq b_n} \frac{1}{b_n^2} \sqrt{\sigma^2\rbraces{\frac{k}{b_n}}+\sigma^2\rbraces{\frac{j}{b_n}}} - \int_0^1 \int_0^1 \sqrt{\sigma^2(x)+\sigma^2(y)}\mathrm{d}x\mathrm{d}y} \\
 & + \abs{\frac{\sqrt{b_n}}{b_n(b_n-1)}\sum_{1\leq j \leq b_n} \sqrt{2\sigma^2\rbraces{\frac{j}{b_n}}}}\longrightarrow 0.
\end{align*}
The second expression above converges towards zero since $\sigma$ is bounded. Turning towards the first expression, define the function $g:[0,1]\times[0,1]\rightarrow \R_{\geq 0}$ by $g(x,y):=\sqrt{\sigma^2(x)+\sigma^2(y)}$ and note that it is bounded from below by $\sqrt{2}\sigma_0$ and from above by $\sqrt{2}\sigma_{\sup}$. Since $\sigma$ is bounded (and, most importantly, bounded away from zero), it holds 
\begin{align*}
& \abs{g(x_1,y_1)-g(x_2,y_2)} = \abs{\sqrt{\sigma^2(x_1)+\sigma^2(y_1)}- \sqrt{\sigma^2(x_2)+\sigma^2(y_2)}}\\
\leq & C_1 \abs{\rbraces{\sigma^2(x_1)+\sigma^2(y_1)}-\rbraces{\sigma^2(x_2)+\sigma^2(y_2)}}\\
\leq & C_1 \rbraces{\abs{\rbraces{\sigma_L^2(x_1)+\sigma_L^2(y_1)}-\rbraces{\sigma_L^2(x_2)+\sigma_L^2(y_2)}}+\abs{\rbraces{\sigma_J^2(x_1)+\sigma_J^2(y_1)}-\rbraces{\sigma_J^2(x_2)+\sigma_J^2(y_2)}}}
\end{align*}
and $g$ is Riemann-integrable. Moreover, the difference between the double Riemann-sum and -integral is of order  
\begin{align*}
& \sqrt{b_n}\abs{\sum_{1\leq j \leq b_n}\sum_{1\leq k \leq b_n} \frac{1}{b_n^2} \sqrt{\sigma^2\rbraces{\frac{k}{b_n}}+\sigma^2\rbraces{\frac{j}{b_n}}} - \int_0^1 \int_0^1 \sqrt{\sigma^2(x)+\sigma^2(y)}\mathrm{d}x\mathrm{d}y} \\
= & \sqrt{b_n}\abs{\sum_{1\leq j \leq b_n}\sum_{1\leq k \leq b_n}\rbraces{ \int_{I_j}\int_{I_k} \sqrt{\sigma^2\rbraces{\frac{k}{b_n}}+\sigma^2\rbraces{\frac{j}{b_n}}}\mathrm{d}x\mathrm{d}y -  \int_{I_j}\int_{I_k} \sqrt{\sigma^2(x)+\sigma^2(y)}\mathrm{d}x\mathrm{d}y}} \\
\leq &  \sqrt{b_n}\sum_{1\leq j \leq b_n}\sum_{1\leq k \leq b_n}\int_{I_j}\int_{I_k} \abs{ \sqrt{\sigma^2\rbraces{\frac{k}{b_n}}+\sigma^2\rbraces{\frac{j}{b_n}}}-\sqrt{\sigma^2(x)+\sigma^2(y)}}\mathrm{d}x\mathrm{d}y \\
\leq &  2C_1\sigma_{\sup} \sqrt{b_n}\sum_{1\leq j \leq b_n}\sum_{1\leq k \leq b_n}\int_{I_j}\int_{I_k} \rbraces{\abs{\frac{k}{b_n}-x}+\abs{\frac{j}{b_n}-y}}\mathrm{d}x\mathrm{d}y\\
& +C_1\sqrt{b_n}\sum_{1\leq j \leq b_n}\sum_{1\leq k \leq b_n}\int_{I_j}\int_{I_k} \rbraces{\abs{\sigma^2_J\rbraces{\frac{k}{b_n}}-\sigma^2_J(x)}+\abs{\sigma^2_J\rbraces{\frac{j}{b_n}}-\sigma^2_J(y)}}\mathrm{d}x\mathrm{d}y\\
\leq & 4 C_1\sigma_{\sup}\frac{\sqrt{b_n}}{b_n}\sum_{1\leq j \leq b_n}\sum_{1\leq k \leq b_n}\int_{I_j}\int_{I_k} 1\mathrm{d}x\mathrm{d}y +4C_1\sqrt{b_n}\frac{k_\sigma\Delta_\sigma \sigma_{\sup}}{b_n}
= O\rbraces{\frac{1}{\sqrt{b_n}}}
\end{align*}
for each $n\in \N$, where $I_j:=((j-1)\ell_n/n, j\ell_n/n]$ for $1\leq j\leq b_n$, and where we omitted the asymptotically negligible factor $b_n/(b_n-1)$. We have thus shown convergence of the second expression in \eqref{eq: Prop centring term, main ineq}, and we now turn towards the first.

 We will first prove that we can replace $\kappa_Y$ by the sample-size dependent 
$$\kappa_{Y,n}:=\Var{\frac{1}{\sqrt{\ell_n}}\sum_{i=1}^{\ell_n-m_n}Y_i},$$
for which we need to check 
\begin{align*}
\sqrt{b_n}\abs{\E{\frac{\sqrt{\ell_n}}{\kappa_Y}U_2(n)}-\E{\frac{\sqrt{\ell_n}}{\kappa_{Y,n}}U_2(n)}}
\leq C \sqrt{b_n}\frac{\abs{\kappa_{Y,n}-\kappa_Y}}{\kappa_Y \kappa_{Y,n}}\longrightarrow 0.
\end{align*}
Since we assumed $\kappa_Y>0$ and we have $\kappa_{Y,n}\rightarrow \kappa_Y$ and thus $\kappa_{Y,n}>\kappa_Y/2$ for $n$ large enough, it suffices to show $\sqrt{b_n}\abs{\kappa_{Y,n}-\kappa_Y}\rightarrow 0$. This can be done analogously to the proof of Proposition A.1 in the supplement to \citep{Schmidt.2021}, given the existence of a $\rho>1$ and a $1\geq \delta>0$ such that $\E{\abs{Y_1}^{2+\delta}}<\infty$ and $\alpha_Y(k)\leq Ck^{-\rho (2+\delta)(1+\delta)/\delta^2}$ and given $\ell_n=n^s$ with $s>0.5$. 

Moreover, 
\begin{align*}
& \sqrt{b_n} \abs{\E{\frac{\sqrt{\ell_n}}{\kappa_{Y,n}}U_2(n)}- \frac{1}{b_n(b_n-1)}\sum_{1\leq j \neq k\leq b_n}  \sqrt{\sigma_{j\ell_n}^2+\sigma_{k\ell_n}^2}\E{\abs{Z}}}\\
 \leq &\frac{ \sqrt{b_n} }{b_n(b_n-1)} \sum_{1\leq j\neq k\leq b_n} \sqrt{\sigma_{j\ell_n}^2+\sigma_{k\ell_n}^2}\abs{\E{ \frac{1}{\sqrt{\sigma_{j\ell_n}^2+\sigma_{k\ell_n}^2}} \frac{\sqrt{\ell_n}}{\kappa_{Y,n}}\abs{\sigma_{j\ell_n}\tilde{S}_{nj}'-\sigma_{k\ell_n}\tilde{S}_{nk}'}}-\E{\abs{Z}}}\\
\leq &  \frac{\sqrt{2}\sigma_{\sup}  \sqrt{b_n} }{b_n(b_n-1)} \sum_{1\leq j\neq k\leq b_n} \abs{\E{\frac{1}{\sqrt{\sigma_{j\ell_n}^2+\sigma_{k\ell_n}^2}} \frac{\sqrt{\ell_n}}{\kappa_{Y,n}}\abs{\sigma_{j\ell_n}\tilde{S}_{n1}'-\sigma_{k\ell_n}\tilde{S}_{n2}'}}-\E{\abs{Z}}}.
\end{align*}
We will bound the above absolute difference by a term of order $o\rbraces{\frac{1}{\sqrt{b_n}}}$, independent of $j$ and $k$. 
Recall that the coupled random variables  $\tilde{S}_{n1}'$ and $\tilde{S}_{n2}'$ are independent with the same distribution as $\tilde{S}_{n1}$. Taking two independent blocks $(Y_1', \ldots, Y_{\ell_n-m_n}')$ and $(Y_{\ell_n+1}', \ldots, Y_{2\ell_n-m_n}')$ with the same distribution as a block of length $\ell_n-m_n$ from the original stationary time series $(Y_i)_{i\in\N}$, we can set 
\begin{align*}
\tilde{S}_{n1}'=\frac{1}{\ell_n}\sum_{i=1}^{\ell_n-m_n} Y_i' \quad \text{and} \quad \tilde{S}_{n2}'=\frac{1}{\ell_n}\sum_{i=\ell_n+1}^{2\ell_n-m_n} Y_i' 
\end{align*}
in the expectation above.

Define the cumulative distribution functions
$$ F_n(x):= F_{n,jk}(x):= \Pb\rbraces{\frac{1}{\sqrt{\sigma_{j\ell_n}^2+\sigma_{k\ell_n}^2}} \frac{\sqrt{\ell_n}}{\kappa_{Y,n}}\rbraces{\sigma_{j\ell_n}\tilde{S}_{n1}'-\sigma_{k\ell_n}\tilde{S}_{n2}'}\leq x} \quad \text{and} \quad \Phi(x):=\Pb\rbraces{Z\leq x}$$
and denote their maximal difference by
$$\Delta_n:=\sup_x\abs{F_n(x)-\Phi(x)}.$$
To obtain a bound for the difference $\abs{\E{\frac{1}{\sqrt{\sigma_{j\ell_n}^2+\sigma_{k\ell_n}^2}} \frac{\sqrt{\ell_n}}{\kappa_{Y,n}}\abs{\sigma_{j\ell_n}\tilde{S}_{n1}'-\sigma_{k\ell_n}\tilde{S}_{n2}'}}-\E{\abs{Z}}}$, we first find an upper bound for $\Delta_n$. Note that
$$\frac{1}{\sqrt{\sigma_{j\ell_n}^2+\sigma_{k\ell_n}^2}} \frac{\sqrt{\ell_n}}{\kappa_{Y,n}}\rbraces{\sigma_{j\ell_n}\tilde{S}_{n1}'-\sigma_{k\ell_n}\tilde{S}_{n2}'}= \frac{1}{\sqrt{\sigma_{j\ell_n}^2+\sigma_{k\ell_n}^2}} \frac{1}{\kappa_{Y,n}\sqrt{\ell_n}} \sum_{i=1}^{\ell_n-m_n} \rbraces{Y_i'\sigma_{j\ell_n}- Y_{i+\ell_n}' \sigma_{k\ell_n}}$$
is centred with variance one and that $(Y_i'\sigma_{j\ell_n}-Y_{i+\ell_n}'\sigma_{k\ell_n})$, $1\leq i\leq \ell_n$ is again $\alpha$-mixing with coefficients smaller than or equal to $2\alpha_Y(k)\leq Ck^{-\rho(2+\delta)(1+\delta)/\delta^2}$ (see, Theorem 1, Chapter 1.1, in Doukhan \citep{Doukhan.1994}), and has finite $(2+\delta)$-moments. According to Theorem 1 in Tikhomirov \citep{Tikhomirov.1980}, there now exists a constant $C_1$ depending solely on $\rho$ and $\delta$ (and thus being independent of $j$ and $k$)
such that 
$$\Delta_n\leq C_1 (\ell_n-m_n)^{-(\delta/2)(\rho-1)/(\rho+1)}.$$
In particular, for $n$ large enough, it holds $\Delta_n\leq 1/\sqrt{e}$ and Theorem 9, Chapter V, in Petrov \citep{Petrov.1975} yields for all $x\in \R$
$$\abs{F_n(x)-\Phi(x)}\leq \frac{C_2\Delta_n \log(1/\Delta_n)}{1+x^2}$$
for a constant $C_2$ independent of $F_n$. Together this yields 
 $$\abs{F_n(x)-\Phi(x)}\leq \frac{C(\ell_n-m_n)^{-(\delta/2)(\rho-1)/(\rho+1)} \log(C (\ell_n-m_n))}{1+x^2}.$$
By symmetry, one furthermore obtains 
 $$\abs{\tilde{F}_n(x)-\tilde{\Phi}(x)}\leq 2 \frac{C(\ell_n-m_n)^{-(\delta/2)(\rho-1)/(\rho+1)} \log(C (\ell_n-m_n))}{1+x^2},$$
 where
 $$ \tilde{F}_n(x):= \Pb\rbraces{\frac{1}{\sqrt{\sigma_{j\ell_n}^2+\sigma_{k\ell_n}^2}} \frac{\sqrt{\ell_n}}{\kappa_{Y,n}}\abs{\sigma_{j\ell_n}\tilde{S}_{n1}'-\sigma_{k\ell_n}\tilde{S}_{n2}'}\leq x} \quad \text{and} \quad \tilde{\Phi}(x):=\Pb\rbraces{\abs{Z}\leq x}.$$
 Hence, 
 \begin{align*}
& \sqrt{b_n} \abs{\E{\frac{1}{\sqrt{\sigma_{j\ell_n}^2+\sigma_{k\ell_n}^2}} \frac{\sqrt{\ell_n}}{\kappa_{Y,n}}\abs{\sigma_{j\ell_n}\tilde{S}_{n1}'-\sigma_{k\ell_n}\tilde{S}_{n2}'}}-\E{\abs{Z}}}
 \leq  \sqrt{b_n} \abs{\int_0^\infty (\tilde{F}_n(x)-\tilde{\Phi}(x) ) \mathrm{d}x}\\
 \leq &  \sqrt{b_n}\int_0^\infty \abs{\tilde{F}_n(x)-\tilde{\Phi}(x) } \mathrm{d}x
 \leq  \sqrt{b_n}   \int_0^\infty  2 \frac{C(\ell_n-m_n)^{-(\delta/2)(\rho-1)/(\rho+1)} \log(C (\ell_n-m_n))}{1+x^2} \mathrm{d}x\\
 \leq & C \sqrt{b_n}(\ell_n-m_n)^{-(\delta/2)(\rho-1)/(\rho+1)} \log(C (\ell_n-m_n)).
 \end{align*}
 Since this upper bound is independent of $j$ and $k$, we overall obtain 
\begin{align*}
& \frac{\sqrt{b_n} }{b_n(b_n-1)} \sum_{1\leq j\neq k\leq b_n} \abs{\E{\frac{1}{\sqrt{\sigma_{j\ell_n}^2+\sigma_{k\ell_n}^2}} \frac{\sqrt{\ell_n}}{\kappa_{Y,n}}\abs{\sigma_{j\ell_n}\tilde{S}_{n1}'-\sigma_{k\ell_n}\tilde{S}_{n2}'}}-\E{\abs{Z}}}\\
\leq & C \sqrt{b_n}(\ell_n-m_n)^{-(\delta/2)(\rho-1)/(\rho+1)} \log(C (\ell_n-m_n)),
\end{align*}
which converges towards zero for $s>\rbraces{1+\delta\frac{\rho-1}{\rho+1}}^{-1} $ and $m_n=o(\ell_n)$.
 \end{proof}

\subsection{Estimation of the nuisance parameters}
\subsubsection{Estimation of the long run variance}
We will start by examining the limit behaviour of the estimator 
\begin{equation*}
\hat{\kappa}_{\tilde{X}}(n)
=\sqrt{\frac{2c_0}{1+2c_0}}\frac{1}{\bn}\sqrt{\frac{\pi}{2}} \sum_{j=1}^{\bn} \frac{1}{\sqrt{\elln}}\abs{\sum_{i=(j-1)\elln+1}^{j\elln} X_i-\frac{1}{2c_0}\rbraces{\sum_{i=(j-1-c_0)\elln+1}^{(j-1)\elln}X_i+\sum_{i=j\elln+1}^{(j+c_0)\elln}X_i}}.
\end{equation*}
Afterwards, we will use Proposition \ref{Prop: Consistency alpha-mix} and Corollary \ref{Cor: Convergence speed LRV esti} below to prove Proposition \ref{Prop: convergence estimator kappahat kappaY under H}. 
\begin{proposition}
\label{Prop: Consistency alpha-mix}
 Assume there is  a $0< \delta \leq 1$ and a $\rho>(\delta^2+4\delta)/(2+2\delta)\vee 1$ such that $\mathbb{E}\big({\abs{Y_1}^{2+\delta}}\big)<\infty$ and $\alpha_Y(k)\leq Ck^{-\rho (2+\delta)(1+\delta)/\delta^2}$.  
 Let $\mu$ be either a constant or a piecewise Lipschitz-continuous function.
 If $\mu$ is piecewise Lipschitz-continuous, additionally assume that $q<2/3$ for $\elln=n^q$. 
 Then, 
$$\hat{\kappa}_{\tilde{X}}(n) \overset{L^2}{\longrightarrow} \int_0^1 \sigma(x)\mathrm{d}x\,\kappa_Y.$$
\end{proposition}

\begin{proof}[Proof of Proposition \ref{Prop: Consistency alpha-mix}.]
We have to show that 
 $$\E{\abs{\hat{\kappa}_{\tilde{X}}(n)-\int_0^1\sigma(x)\mathrm{d}x\,\kappa_Y}^2}\rightarrow 0.$$
We start with some simplifying assumptions. 
For ease of notation, assume that $c_0\in \N$ and ignore the fact that the first and last $c_0$ subsampling blocks are centred by less than $2c_0\elln$ observations, as their influence is asymptotically negligible. 
Define 
$$\hat{\kappa}_{\tilde{Y}}(n):= \sqrt{\frac{2c_0}{2c_0+1}} \frac{1}{\bn}\sqrt{\frac{\pi}{2}}\sum_{j=1}^{\bn} \frac{\sigma_{j\elln}}{\sqrt{\elln}}\abs{\sum_{i=(j-1)\elln+1}^{j\elln} Y_i- \frac{1}{2c_0}\rbraces{\sum_{i=(j-c_0-1)\elln+1}^{(j-1)\elln}Y_i+\sum_{i=j\elln+1}^{(j+c_0)\elln}Y_i}}.$$
In the following, we will show that instead of $\hat{\kappa}_{\tilde{X}}(n)$, we can consider $\hat{\kappa}_{\tilde{Y}}(n)$. It holds
\begin{align*}
& \E{\abs{\hat{\kappa}_{\tilde{X}}(n)-\hat{\kappa}_{\tilde{Y}}(n)}^2}\\
\leq & C\frac{\elln}{n^2}\bigg(\sum_{j=1}^{\bn} \bigg |\sum_{i=(j-1)\elln+1}^{j\elln} \mu_i - \frac{1}{2c_0} \sum_{i=(j-c_0-1)\elln+1}^{(j-1)\elln} \mu_i - \frac{1}{2c_0} \sum_{i=j\elln+1}^{(j+c_0)\elln} \mu_i\bigg |\bigg)^2\\
+&C\frac{\elln}{n^2} \mathbb{E}\bigg(\bigg(\sum_{j=1}^{\bn} \bigg|\sum_{i=(j-1)\elln+1}^{j\elln}(\sigma_i-\sigma_{j\elln}) Y_i -\frac{1}{2c_0}\sum_{i=(j-c_0-1)\elln+1}^{(j-1)\elln} (\sigma_i-\sigma_{j\elln})Y_i\\ 
&  -\frac{1}{2c_0}\sum_{i=j \elln+1}^{(j+c_0)\elln} (\sigma_i-\sigma_{j\elln}) Y_i\bigg |\bigg)^2\bigg).
\end{align*}
We show the first of these terms to be of order $O(\elln^3c_0^2/n^2+ \elln^3c_0^2k_\mu^2\Delta_\mu^2/n^2)$, where $k_\mu$ and $\Delta_\mu$ correspond to the number and maximal absolute jump height of the piecewise Lipschitz-continuous mean function $\mu$, respectively. First, consider the Lipschitz-continuous part $\mu_L$ of $\mu=:\mu_L+\mu_J$. For $i\in \{(j-1)\elln+1, \dots, j\elln\}$ and $r\in (\{(j-c_0-1)\elln+1, \dots, (j-1)\elln\}\cup\{ j\elln+1,\dots ,(j+c_0)\elln\})$, we obtain $\abs{\mu_{L,i}-\mu_{L,r}}\leq C\frac{(c_0+1)\elln}{n}= O(c_0\elln/n)$. 
Consequently,
\begin{align*}
 &\frac{\elln}{n^2} \rbraces{ \sum_{j=1}^{\bn} \abs{\sum_{i=(j-1)\elln+1}^{j\elln} \mu_{L,i}-\frac{1}{2c_0}\rbraces{\sum_{r=(j-c_0-1)\elln+1}^{(j-1)\elln}\mu_{L,r}+\sum_{r=j\elln+1}^{(j+c_0)\elln}\mu_{L,r}}}}^2\\
\leq  & \frac{\elln}{n^2}   \rbraces{ \sum_{j=1}^{\bn} \sum_{i=(j-1)\elln+1}^{j\elln} \frac{1}{2c_0\elln}\rbraces{\sum_{r=(j-c_0-1)\elln+1}^{(j-1)\elln}\abs{\mu_{L,i}-\mu_{L,r}}+\sum_{r=j\elln+1}^{(j+c_0)\elln}\abs{\mu_{L,i}-\mu_{L,r}}}}^2 \\
= & O\rbraces{\frac{c_0^2\elln^3}{n^2}}.
\end{align*}

We now have to account for the additional errors made due to the jumps in the mean function. Recall that $k_\mu$ and $\Delta_\mu\geq 0$ denote the number of jumps and the maximal absolute jump height, respectively. 
Imagine there occurs a jump in the $j$th block, whose height is at most $\Delta_\mu$. This $j$th block, at least in part, and the subsequent blocks will then have a mean function at a by $\Delta_\mu$ different level. Thus, the calculation of the centring terms of all blocks $(j-c_0)$ up to $(j+c_0)$  will be at least in part corrupted by the change and we obtain the upper bound $\abs{\mu_{J,i}-\mu_{J,r}} \leq \Delta_\mu$. Moreover, the error made due to the different jumps is additive such that we altogether obtain that the jump part $\mu_J$ of $\mu$ can bounded by a term of order
$$O\rbraces{\rbraces{ k_\mu\frac{1}{\bn} c_0 \elln^{1/2}\Delta_\mu}^2}=O\rbraces{\frac{\elln^3 c_0^2}{n^2}k_\mu^2\Delta_\mu^2}.$$

The second term of the above terms can be bounded (up to a constant) by 
\begin{align*}
&\frac{\elln}{n^2} \E{\abs{\sum_{j=1}^{\bn} \abs{\sum_{i=(j-1)\elln+1}^{j\elln}(\sigma_i-\sigma_{j\elln}) Y_i}}^2} + \frac{\elln}{n^2}\E{\abs{\sum_{j=1}^{\bn}\abs{ \frac{1}{2c_0}\sum_{i=(j-c_0-1)\elln+1}^{(j-1)\elln} (\sigma_i-\sigma_{j\elln})Y_i}}^2}\\
&+\frac{\elln}{n^2}\E{\abs{\sum_{j=1}^{\bn} \abs{\frac{1}{2c_0}\sum_{i=j \elln+1}^{(j+c_0)\elln} (\sigma_i-\sigma_{j\elln}) Y_i}}^2}.
\end{align*}
The three terms above are each of order $O(1/\bn^2)$, e.g., for the second term, one obtains
\begin{align*}
& \frac{\elln}{n^2}\E{\abs{\sum_{j=1}^{\bn}\abs{ \frac{1}{2c_0}\sum_{i=(j-c_0-1)\elln+1}^{(j-1)\elln} (\sigma_i-\sigma_{j\elln})Y_i}}^2}\\
\leq &\frac{2\elln}{n^2}  \E{\abs{\sum_{j=1}^{\bn}\abs{ \frac{1}{2c_0}\sum_{i=(j-c_0-1)\elln+1}^{(j-1)\elln} (\sigma_{L,i}-\sigma_{L,j\elln})Y_i}}^2}\\
& +\frac{2\elln}{n^2}  \E{\abs{\sum_{j=1}^{\bn}\abs{ \frac{1}{2c_0}\sum_{i=(j-c_0-1)\elln+1}^{(j-1)\elln} (\sigma_{J,i}-\sigma_{J,j\elln})Y_i}}^2},
\end{align*}
where the Lipschitz-continuous component of the variance function can be bounded by
\begin{align*}
& \frac{\elln}{n^2}\E{\abs{\sum_{j=1}^{\bn}\abs{ \frac{1}{2c_0}\sum_{i=(j-c_0-1)\elln+1}^{(j-1)\elln} (\sigma_{L,i}-\sigma_{L,j\elln})Y_i}}^2}\\
\leq & C \frac{\elln}{n^2} \rbraces{\sum_{j=1}^{\bn}\rbraces{ \E{\abs{\sum_{i=(j-c_0-1)\elln+1}^{(j-1)\elln} (\sigma_{L,i}-\sigma_{L,j\elln})Y_i}^2}}^{1/2}}^2\\
\leq & C \frac{\elln}{n^2} \rbraces{\sum_{j=1}^{\bn} \rbraces{\sum_{i=(j-c_0-1)\elln+1}^{(j-1)\elln}\sum_{r=(j-c_0-1)\elln+1}^{(j-1)\elln} \abs{\sigma_{L,i}-\sigma_{L,j\elln}}\abs{\sigma_{L,r}-\sigma_{L,j\elln}}\abs{\E{Y_iY_r}}}^{1/2}}^2\\
\leq &  C \frac{\elln}{n^2} \frac{(c_0+1)^2}{\bn^2} \rbraces{\bn \rbraces{\sum_{i=1}^{c_0\elln}\sum_{r=1}^{c_0\elln} \abs{\E{Y_iY_r}}}^{1/2}}^2
\leq C \frac{\elln^2}{n^2} \rbraces{\norm{Y_1}^2_2+2\sum_{k=2}^{\infty}\abs{\E{Y_1Y_k}}}\\
\leq & C \frac{\elln^2}{n^2} \rbraces{\norm{Y_1}^2_2+16\sum_{k=1}^{\infty}  \alpha_Y(k)^{\delta/(2+\delta)} \norm{Y}^2_{2+\delta}}
\end{align*}
by the Minkowski inequality and by Davydovs covariance inequality. Analogously, the jump component can be bounded by 
\begin{align*}
& \frac{\elln}{n^2}\E{\abs{\sum_{j=1}^{\bn}\abs{ \frac{1}{2c_0}\sum_{i=(j-c_0-1)\elln+1}^{(j-1)\elln} (\sigma_{J,i}-\sigma_{J,j\elln})Y_i}}^2}\\
\leq & C \frac{\elln}{n^2} \rbraces{\sum_{j=1}^{\bn} \rbraces{\sum_{i=(j-c_0-1)\elln+1}^{(j-1)\elln}\sum_{r=(j-c_0-1)\elln+1}^{(j-1)\elln} \abs{\sigma_{J,i}-\sigma_{J,j\elln}}\abs{\sigma_{J,r}-\sigma_{J,j\elln}}\abs{\E{Y_iY_r}}}^{1/2}}^2\\
\leq &  C \frac{\elln}{n^2} \rbraces{(2c_0+1)k_\sigma\Delta_\sigma}^2\sum_{i=1}^{c_0\elln}\sum_{r=1}^{c_0\elln} \abs{\E{Y_iY_r}}
\leq C \frac{\elln^2}{n^2} \rbraces{\norm{Y_1}^2_2+16\sum_{k=1}^{\infty}  \alpha_Y(k)^{\delta/(2+\delta)} \norm{Y}^2_{2+\delta}}.
\end{align*}

In the following, it hence suffices to consider
$\hat{\kappa}_{\tilde{Y}}(n)$ 
and to show $\E{\abs{\hat{\kappa}_{\tilde{Y}}(n)-\E{\hat{\kappa}_{\tilde{Y}}(n)}}^2}\rightarrow 0$ as well as
$\abs{\E{\hat{\kappa}_{\tilde{Y}}(n)}-\int_0^1\sigma(x)\mathrm{d}x\,\kappa_Y}^2\rightarrow 0$. Some parts of this proof resemble the proof of Proposition 3.1 in Dehling et al. \citep{Dehling.2013}. To show the first assertion, define $$S_j(n):= \sum_{i=(j-1)\elln+1}^{j\elln} Y_i -\frac{1}{2c_0}\sum_{i=(j-c_0-1)\elln+1}^{(j-1)\elln} Y_i -\frac{1}{2c_0}\sum_{i=j \elln+1}^{(j+c_0)\elln} Y_i$$ and note that the process $(S_j(n))_{j\in \N}$ is again $\alpha$-mixing with coefficients
\begin{align*}
 \alpha_{S(n)}(k)= & \sup_j \alpha(\sigma(S_1(n), \dots, S_j(n)), \sigma(S_{j+k}(n), \dots))\\ = & \sup_j \alpha(\sigma(Y_1, \dots, Y_{(j+c_0)\elln}), \sigma(Y_{(j+k-c_0-1)\elln+1}, \dots))\\
 = & \alpha_Y((k-2c_0-1)\elln+1) \longrightarrow 0 
\end{align*}
as $k\rightarrow \infty$. Thus, we can apply the Rosenthal-type inequality from Theorem \ref{Literature: Thm Rosenthal}, the $c_r$-inequality from Lemma \ref{Literature: Lemma cr-ineq} and the inequality by Yokoyama from Theorem \ref{Literature: Thm Yokoyama}  to obtain 
\begin{align*}
& \E{\abs{\hat{\kappa}_{\tilde{Y}}(n)-\E{\hat{\kappa}_{\tilde{Y}}(n)}}^2} = C \E{\abs{\frac{1}{\bn}\sum_{j=1}^{\bn} \frac{1}{\sqrt{\elln}} \sigma_{j\elln} \rbraces{\abs{S_j(n)}-\E{\abs{S_j(n)}}} }^2}\\
\leq & \frac{C}{\bn^2}\frac{1}{\elln}\sum_{j=1}^{\bn} \E{\abs{\sigma_{j\elln} \rbraces{\abs{S_j(n)}-\E{\abs{S_j(n)}}}}^{2+\delta/2}}^{\frac{2}{2+\delta/2}}\\
\leq & \sigma_{\sup}^2 \frac{C}{n} \rbraces{\rbraces{\E{\abs{S_1(n)}^{2+\delta/2}}}^{\frac{2}{2+\delta/2}}} 
\leq C \frac{1}{n} \rbraces{\elln^\frac{2+\delta/2}{2}}^{\frac{2}{2+\delta/2}}=O\rbraces{\frac{1}{\bn}},
\end{align*}
where the first inequality holds under the condition $\sum_{k=1}^\infty \alpha_{S(n)}(k)^{\frac{\delta/2}{2+\delta/2}}\leq \sum_{k=1}^\infty \alpha_{Y}(k)^{\frac{\delta/2}{2+\delta/2}} <\infty$  and the third inequality under the condition $\sum_{k=1}^\infty k^{\delta/4} \alpha_Y(k)^{\frac{\delta/2}{2+\delta}}<\infty$. Due to $\alpha_Y(k)\leq Ck^{-\rho(1+\delta)(2+\delta)/\delta^2}$, the first condition is met for $\rho>1$ since $0<\delta\leq 1$ and the second one for $\rho>(\delta^2+4\delta)/(2+2\delta)$, which is met for all $0<\delta\leq 1$ if $\rho>1.25$. 

Turning towards the last assertion $\abs{\E{\hat{\kappa}_{\tilde{Y}}(n)}-\int_0^1\sigma(x)\mathrm{d}x\,\kappa_Y}^2\rightarrow 0$, we note that 
$$\E{\hat{\kappa}_{\tilde{Y}}(n)}= \sqrt{\frac{2c_0}{2c_0+1}}\sqrt{\frac{\pi}{2}}\frac{1}{\sqrt{\elln}}\E{\abs{S_1(n)}} \frac{1}{\bn}\sum_{j=1}^{\bn} \sigma_{j\elln}$$
and 
\begin{align*}
& \abs{\E{\hat{\kappa}_{\tilde{Y}}(n)}-\int_0^1\sigma(x)\mathrm{d}x\,\kappa_Y}^2\\
\leq & C\rbraces{ \abs{\E{\hat{\kappa}_{\tilde{Y}}(n)}-\frac{1}{\bn}\sum_{j=1}^{\bn}\sigma_{j\elln} \kappa_Y}^2+\abs{\frac{1}{\bn}\sum_{j=1}^{\bn} \sigma_{j\elln} \kappa_Y-\int_0^1\sigma(x)\mathrm{d}x\,\kappa_Y}^2 }.
\end{align*}
Due to the piecewise Lipschitz-continuity of $\sigma$, the squared difference between the Riemann-sum and -integral, $\kappa_Y^2\abs{\frac{1}{\bn}\sum_{j=1}^{\bn} \sigma_{j\elln}-\int_0^1\sigma(x)\mathrm{d}x}^2$, is of order $O(1/\bn^2)$.
Turning towards the first term, we obtain
\begin{align*}
& \abs{\E{\hat{\kappa}_{\tilde{Y}}(n)}-\frac{1}{\bn}\sum_{j=1}^{\bn} \sigma_{j\elln} \kappa_Y}^2\\
\leq & \sigma_{\sup}^2\abs{\sqrt{\frac{2c_0}{2c_0+1}}\sqrt{\frac{\pi}{2}} \E{ \frac{1}{\sqrt{\elln}}\abs{\sum_{i=(j-1)\elln+1}^{j\elln} Y_i -\frac{1}{2c_0}\sum_{i=(j-c_0-1)\elln+1}^{(j-1)\elln} Y_i -\frac{1}{2c_0}\sum_{i=j \elln+1}^{(j+c_0)\elln} Y_i} }-\kappa_Y}^2\\
= &  \sigma_{\sup}^2\abs{\sqrt{\frac{2c_0}{2c_0+1}}\sqrt{\frac{\pi}{2}} \E{ \frac{1}{\sqrt{\elln}}\abs{\sum_{i=(j-1)\elln+1}^{j\elln} Y_i -\frac{1}{2c_0}\sum_{k=1}^{c_0}\rbraces{\sum_{i=(j-k-1)\elln+1}^{(j-k)\elln} Y_i +\sum_{i=(j+k-1) \elln+1}^{(j+k)\elln} Y_i}} }-\kappa_Y}^2.
\end{align*}
Each of the $(2c_0+1)$ block sums within the expectation above  converges towards a normal distribution by the central limit theorem for $\alpha$-mixing random variables. However, as these sums are dependent, to obtain an overall limit, we need to introduce a small separation block of  $\tilde{m}_n$ observations between them and use a coupling technique for the resulting sums. More precisely, 
\begin{align*}
&\left| \E{ \frac{1}{\sqrt{\elln}}\abs{\sum_{i=(j-1)\elln+1}^{j\elln} Y_i -\frac{1}{2c_0}\sum_{k=1}^{c_0}\rbraces{\sum_{i=(j-k-1)\elln+1}^{(j-k)\elln} Y_i +\sum_{i=(j+k-1) \elln+1}^{(j+k)\elln} Y_i}} } \right. \\
&  \left. - \E{ \frac{1}{\sqrt{\elln}}\abs{\sum_{i=(j-1)\elln+1}^{j\elln-\tilde{m}_n} Y_i -\frac{1}{2c_0}\sum_{k=1}^{c_0}\rbraces{\sum_{i=(j-k-1)\elln+1}^{(j-k)\elln-\tilde{m}_n} Y_i +\sum_{i=(j+k-1) \elln+1}^{(j+k)\elln-\tilde{m}_n} Y_i}} }\right|^2\\
\leq & \rbraces{\E{ \frac{1}{\sqrt{\elln}}\abs{\sum_{i=j\elln-\tilde{m}_n+1}^{j\elln} Y_i -\frac{1}{2c_0}\sum_{k=1}^{c_0}\rbraces{\sum_{i=(j-k)\elln-\tilde{m}_n+1}^{(j-k)\elln} Y_i +\sum_{i=(j+k)\elln-\tilde{m}_n+1}^{(j+k)\elln} Y_i}} }}^2\\
\leq & 4 \frac{\tilde{m}_n}{\elln} \E{\abs{\frac{1}{\sqrt{\tilde{m}_n}}\sum_{i=1}^{\tilde{m}_n}Y_i}^2} =O\rbraces{\frac{\tilde{m}_n}{\elln}}.
\end{align*}
Moreover, by Theorem \ref{Literature: Thm Peligrad}, there exists (on a possibly richer probability space) a sequence of independent random variables $W_j(n)$, $1\leq j \leq \bn$, such that $W_j(n) \overset{\mathcal{D}}{=}\sum_{i=(j-1)\elln+1}^{j\elln-\tilde{m}_n}Y_i $ for each $j$
and
 $$\E{\abs{\sum_{i=(j-1)\elln+1}^{j\elln-\tilde{m}_n}Y_i - W_j(n) }}
 \leq 4\sqrt{\alpha_Y(\tilde{m}_n)} \E{\abs{\sum_{i=1}^{\elln-\tilde{m}_n}Y_i}^2}^{1/2}.$$
Thus, we obtain
\begin{align*}
&\left|\E{ \frac{1}{\sqrt{\elln}}\abs{\sum_{i=(j-1)\elln+1}^{j\elln-\tilde{m}_n} Y_i -\frac{1}{2c_0}\sum_{k=1}^{c_0}\rbraces{\sum_{i=(j-k-1)\elln+1}^{(j-k)\elln-\tilde{m}_n} Y_i +\sum_{i=(j+k-1) \elln+1}^{(j+k)\elln-\tilde{m}_n} Y_i}} }\right. \\
&  \left. - \E{ \frac{1}{\sqrt{\elln}}\abs{W_j(n) -\frac{1}{2c_0}\sum_{k=1}^{c_0}\rbraces{W_{j-k}(n) +W_{j+k}(n)}} } \right|^2\\
\leq & \frac{4}{\elln} \rbraces{\E{\abs{\sum_{i=(j-1)\elln+1}^{j\elln-\tilde{m}_n}Y_i - W_j(n) }} }^2
 \leq C \alpha_Y(\tilde{m}_n) \E{\abs{\frac{1}{\sqrt{\elln}}\sum_{i=1}^{\elln-\tilde{m}_n}Y_i}^2}= O(\alpha_Y(\tilde{m}_n)) .
\end{align*}
Hence, we reduced the problem to considering
$$\abs{\sqrt{\frac{2c_0}{1+2c_0}}\sqrt{\frac{\pi}{2}}\E{ \frac{1}{\sqrt{\elln}}\abs{W_j(n) -\frac{1}{2c_0}\sum_{k=1}^{c_0}\rbraces{W_{j-k}(n) +W_{j+k}(n)}} }-\kappa_Y}^2$$
Following Proposition \ref{Prop: Replacing centring term}, we can once more assume a sum structure for the $W_j(n)$'s. Take $(2c_0+1)$ independent blocks $(Y'_{(j-c_0+u-1)\elln+1}, \ldots, Y'_{(j-c_0+u)\elln-\tilde{m}_n}))$, $u=0, \ldots, 2c_0$, with the same distribution as $(Y_1, \ldots, Y_{\elln-\tilde{m}_n})$. Then, we can
set  $W_{j-c_0+u}(n)=\sum_{i=(j-c_0+u-1)\elln+1}^{(j-c_0+u)\elln-\tilde{m}_n}Y_i'$
and rewrite the above expectation as

\begin{align*}
&\abs{\sqrt{\frac{2c_0}{1+2c_0}}\sqrt{\frac{\pi}{2}}\E{ \frac{1}{\sqrt{\elln}}\abs{\sum_{i=(j-1)\elln+1}^{j\elln-\tilde{m}_n} Y_i' -\frac{1}{2c_0}\sum_{k=1}^{c_0}\rbraces{\sum_{i=(j-k-1)\elln+1}^{(j-k)\elln-\tilde{m}_n} Y_i' +\sum_{i=(j+k-1) \elln+1}^{(j+k)\elln-\tilde{m}_n} Y_i'}} }-\kappa_Y}^2\\
\leq & 2\left( \tilde{\kappa}_{Y,n}^2 \frac{\pi}{2} \abs{\E{\abs{\sqrt{\frac{2c_0}{1+2c_0}}\frac{1}{\tilde{\kappa}_{Y,n}} \frac{1}{\sqrt{\elln}}\sum_{i=(j-1)\elln+1}^{j\elln-\tilde{m}_n} \rbraces{Y_i' -\frac{1}{2c_0}\sum_{k=1}^{c_0}\rbraces{ Y_{i-k\elln}' + Y_{i+k\elln}'}}} }-
\E{\abs{Z}}}^2 \right.\\
& \left. +\abs{\tilde{\kappa}_{Y,n}-\kappa_Y}^2\right),
\end{align*}
where $Z$ is a standard normal random variable and
$$\tilde{\kappa}_{Y,n}^2:= \Var{\frac{1}{\sqrt{\elln}}\sum_{i=1}^{\elln-\tilde{m}_n}Y_i'}.$$ 
The difference $\abs{\tilde{\kappa}_{Y,n}-\kappa_Y}^2 $ is of order $O\rbraces{\frac{\tilde{m}_n}{\elln}+\frac{1}{\elln-\tilde{m}_n}}$, which can be seen as in Proposition A.1 in the appendix to \citep{Schmidt.2021}. Turning towards the first difference, since the blocks $\sum_{i=(j-1)\elln+1}^{j\elln-\tilde{m}_n}Y_i'$, $1\leq j \leq \bn$ are independent, the $\rbraces{Y_i' -\frac{1}{2c_0}\sum_{k=1}^{c_0}\rbraces{ Y_{i-k\elln}' + Y_{i+k\elln}'}}$'s are stationary, centred, $\alpha$-mixing with coefficients smaller than or equal to $(2c_0+1)\alpha_Y$ by Theorem \ref{Literature: Doukhan Thm 1 Sec 1.1}, and have finite $(2+\delta)$-moments by assumption and by an application of the $c_r$-inequality. Moreover, 
$$\Var{\sqrt{\frac{2c_0}{1+2c_0}}\frac{1}{\tilde{\kappa}_{Y,n}} \frac{1}{\sqrt{\elln}}\sum_{i=(j-1)\elln+1}^{j\elln-\tilde{m}_n} \rbraces{Y_i' -\frac{1}{2c_0}\sum_{k=1}^{c_0}\rbraces{ Y_{i-k\elln}' + Y_{i+k\elln}'}}}=1.$$
By the central limit theorem for $\alpha$-mixing processes, it holds 
$$\sqrt{\frac{2c_0}{1+2c_0}}\frac{1}{\tilde{\kappa}_{Y,n}} \frac{1}{\sqrt{\elln}}\sum_{i=(j-1)\elln+1}^{j\elln-\tilde{m}_n} \rbraces{Y_i' -\frac{1}{2c_0}\sum_{k=1}^{c_0}\rbraces{ Y_{i-k\elln}' + Y_{i+k\elln}'}}\distConv Z$$
and the above difference of expectations converges towards zero, but we still need to determine the rate of convergence. This can be done analogously to Proposition \ref{Prop: Replacing centring term}, where Theorems \ref{Literature: Tikhomirov} and \ref{Literature: Petrov} are used to ultimately obtain the bound
\begin{align*}
&\abs{\E{\abs{\sqrt{\frac{2c_0}{1+2c_0}}\frac{1}{\tilde{\kappa}_{Y,n}} \frac{1}{\sqrt{\elln}}\sum_{i=(j-1)\elln+1}^{j\elln-\tilde{m}_n} \rbraces{Y_i' -\frac{1}{2c_0}\sum_{k=1}^{c_0}\rbraces{ Y_{i-k\elln}' + Y_{i+k\elln}'}}} }-
\E{\abs{Z}}}^2 \\
=& O\rbraces{\rbraces{\elln-\tilde{m}_n}^{-\delta(\rho-1)/(\rho+1)}\log\rbraces{C(\elln-\tilde{m}_n)}^2}.
\end{align*}
\end{proof}

\begin{corollary}
\label{Cor: Convergence speed LRV esti}
 Assume there is  a $0< \delta \leq 1$ and a $\rho>(\delta^2+4\delta)/(2+2\delta)\vee 1$ such that $\mathbb{E}\big({\abs{Y_1}^{2+\delta}}\big)<\infty$ and $\alpha_Y(k)\leq Ck^{-\rho (2+\delta)(1+\delta)/\delta^2}$.  
 Denote $\ell_n=n^s$ and $\elln=n^q$. 
Assume 
 $q<s$ as well as $q>(1-s) \cdot \max\big(\delta^{-1}\frac{\rho+1}{\rho-1},\rbraces{1+\delta^2(\rho(1+\delta)(2+\delta))^{-1}}\big)$. Then, it holds for  a constant mean function $\mu$ that
 $$ \sqrt{b_n}\abs{\hat{\kappa}_{\tilde{X}}(n)-\kappa_Y \int_0^1 \sigma(x)\mathrm{d}x}\overset{L^2}{\longrightarrow} \;0.$$
\end{corollary}

\begin{proof}[Proof of Corollary \ref{Cor: Convergence speed LRV esti}]
 A look at the proof of Proposition \ref{Prop: Consistency alpha-mix} reveals that the difference $\E{\abs{\hat{\kappa}_{\tilde{X}}(n)-\kappa_Y \int_0^1 \sigma(x)\mathrm{d}x}^2}$ is of order 
$$O\rbraces{\frac{1}{\bn}}+O\rbraces{\frac{\tilde{m}_n}{\elln}}+ O\rbraces{\alpha_Y(\tilde{m}_n)}+O\rbraces{\elln^{\, -\delta \frac{\rho-1}{\rho+1}}\log(\elln)^2}.$$
Multiplying with $b_n$, we arrive at the conditions $b_n/\bn\rightarrow 0$, $b_n/\elln \rightarrow 0$ as well as $b_n \elln^{\, -\delta \frac{\rho-1}{\rho+1}}\log(\elln)^2 \rightarrow 0$. Using the notation $\ell_n=n^s$ and $ \elln=n^q$, these translate to $s>q$, $q>(1-s)$ and $q>(1-s)\delta^{-1}\frac{\rho+1}{\rho-1}$. 
We additionally need to ensure the existence of a sequence $\tilde{m}_n\rightarrow \infty$ with $\frac{b_n \tilde{m}_n}{\elln}\rightarrow 0$ and $b_n \alpha_Y(\tilde{m}_n)\rightarrow 0$. The first condition translates to $\tilde{m}_n=o(n^{q+s-1})$ and given the assumption $(1-s)<q$, we can simply set $\tilde{m_n}=n^{q+s-1-\varepsilon}$ for some $\varepsilon>0$ small enough. Turning towards the second condition, we insert $\alpha_Y(\tilde{m}_n)\leq C \tilde{m}_n^{-\rho(1+\delta)(2+\delta)/\delta^2}$ together with the above form of $\tilde{m}_n$ and arrive at the condition $q-\varepsilon>(1-s)(1+\delta^2(\rho(1+\delta)(2+\delta))^{-1})$. We hence obtain the condition  $q>(1-s)(1+\delta^2(\rho(1+\delta)(2+\delta))^{-1})$ and define $\varepsilon:=(q-(1-s)(1+\delta^2(\rho(1+\delta)(2+\delta))^{-1}))/2\wedge (q+s-1)/2$.  
\end{proof}

\begin{proof}[Proof of Proposition \ref{Prop: convergence estimator kappahat kappaY under H}]
We split the expression up via
\begin{align}
\label{eq: Proof convergence kappahat kappaY H}
\sqrt{b_n}\abs{\hat{\kappa}(n)-\kappa_Y}
\leq \sqrt{b_n}\rbraces{\abs{\frac{\hat{\kappa}_{\tilde{X}}(n)-\kappa_Y \int_0^1 \sigma(x)\mathrm{d}x}{\frac{1}{b_n}\sum_{j=1}^{b_n}\hat{\sigma}_{j,n}}}+ \kappa_Y\abs{\frac{\frac{1}{b_n}\sum_{j=1}^{b_n}\hat{\sigma}_{j,n}-\int_0^1 \sigma(x)\mathrm{d}x}{\frac{1}{b_n}\sum_{j=1}^{b_n}\hat{\sigma}_{j,n}}}}.
\end{align}

i) Let us first consider the behaviour under the hypothesis of a constant mean function.
To prove the convergence in probability of \eqref{eq: Proof convergence kappahat kappaY H}, due to Corollary \ref{Cor: Convergence speed LRV esti}, it only remains to show
$$\sqrt{b_n}\abs{\frac{1}{b_n}\sum_{j=1}^{b_n}\hat{\sigma}_{j,n}-\int_0^1 \sigma(x)\mathrm{d}x} \pConv 0$$
since $\sigma\geq\sigma_0>0$. We can bound this difference from above by
\begin{align*}
  \sqrt{b_n}\abs{\frac{1}{b_n}\sum_{j=1}^{b_n}\sqrt{\hat{\sigma}_{j,n}^2}-\frac{1}{b_n}\sum_{j=1}^{b_n} \sqrt{\sigma^2\rbraces{\frac{j}{b_n}}} }+  \sqrt{b_n} \abs{\frac{1}{b_n}\sum_{j=1}^{b_n} \sigma\rbraces{\frac{j}{b_n}}- \int_0^1 \sigma(x)\mathrm{d}x}.
\end{align*}
The first term can be treated by an easier version of the proof of Proposition \ref{Prop: Convergence estimated centring term} to obtain convergence in probability (and in fact, even in $L^1$) towards zero given the conditions $s>0.5$ and $\E{\abs{Y_1}^{4+2\delta}}<\infty$. For the second term, the piecewise Lipschitz-continuity of $\sigma$ yields
\begin{align*}
\sqrt{b_n} \abs{\frac{1}{b_n}\sum_{j=1}^{b_n} \sigma\rbraces{\frac{j}{b_n}}- \int_0^1 \sigma(x)\mathrm{d}x}
\leq  \sqrt{b_n}\sum_{j=1}^{b_n} \int_{I_j} \abs{\sigma\rbraces{\frac{j}{b_n}}-\sigma(x)} \mathrm{d}x=O\rbraces{\frac{1}{b_n^{1/2}}}
\end{align*}
with $I_j:=\left(\frac{(j-1)\ell_n}{n}, \frac{j\ell_n}{n}\right]$ for $1\leq j\leq b_n$. \\

ii) Turning towards a piecewise Lipschitz-continuous mean function, we can drop the rate $\sqrt{b_n}$ in \eqref{eq: Proof convergence kappahat kappaY H}.  Proposition \ref{Prop: Consistency alpha-mix} yields $\E{\abs{\hat{\kappa}_{\tilde{X}}(n)-\kappa_Y \int_0^1 \sigma(x)\mathrm{d}x}^2}\rightarrow 0$ and it remains to prove $$\abs{\frac{1}{b_n}\sum_{j=1}^{b_n}\hat{\sigma}_{j,n}-\int_0^1 \sigma(x)\mathrm{d}x} \pConv 0.$$ 
In the following, we will even show $\E{\abs{\frac{1}{b_n}\sum_{j=1}^{b_n}\hat{\sigma}_{j,n}-\int_0^1 \sigma(x)\mathrm{d}x}} \rightarrow 0$. We start by bounding the difference from above by
\begin{align} 
\label{eq: Appendix Proof of consistency kappahat for Lipschitz}
&\abs{\frac{1}{b_n}\sum_{j=1}^{b_n}\hat{\sigma}_{j,n}-\frac{1}{b_n}\sum_{j=1}^{b_n}\rbraces{ \frac{1}{\ell_n} \sum_{i=(j-1)\ell_n+1}^{j\ell_n} \sigma_i^2Y_i^2-\rbraces{\frac{1}{\ell_n}\sum_{i=(j-1)\ell_n+1}^{j\ell_n} \sigma_iY_i}^2}^{1/2}}\\
\nonumber + & \abs{\frac{1}{b_n}\sum_{j=1}^{b_n}\rbraces{ \frac{1}{\ell_n} \sum_{i=(j-1)\ell_n+1}^{j\ell_n} \sigma_i^2Y_i^2-\rbraces{\frac{1}{\ell_n}\sum_{i=(j-1)\ell_n+1}^{j\ell_n} \sigma_iY_i}^2}^{1/2}-\frac{1}{b_n}\sum_{j=1}^{b_n}\sigma\rbraces{\frac{j}{b_n}}}\\
 \nonumber + & \abs{\frac{1}{b_n}\sum_{j=1}^{b_n} \sigma\rbraces{\frac{j}{b_n}}- \int_0^1 \sigma(x)\mathrm{d}x}.
\end{align}
The first expression in \eqref{eq: Appendix Proof of consistency kappahat for Lipschitz} can in turn be bounded by
\begin{align*}
& \abs{\frac{1}{b_n}\sum_{j=1}^{b_n}\hat{\sigma}_{j,n}-\frac{1}{b_n}\sum_{j=1}^{b_n}\rbraces{ \frac{1}{\ell_n} \sum_{i=(j-1)\ell_n+1}^{j\ell_n} \sigma_i^2Y_i^2-\rbraces{\frac{1}{\ell_n}\sum_{i=(j-1)\ell_n+1}^{j\ell_n} \sigma_iY_i}^2}^{1/2}}\\
\leq & \frac{1}{b_n}\sum_{j=1}^{b_n} \rbraces{\frac{2}{\ell_n}\sum_{i=(j-1)\ell_n+1}^{j\ell_n} \sigma_i\mu_iY_i}^{1/2}+ \frac{1}{b_n}\sum_{j=1}^{b_n} \rbraces{2\rbraces{\frac{1}{\ell_n} \sum_{i=(j-1)\ell_n+1}^{j\ell_n}\sigma_iY_i}\rbraces{\frac{1}{\ell_n} \sum_{i=(j-1)\ell_n+1}^{j\ell_n}\mu_i}}^{1/2} \\
&+\frac{1}{b_n}\sum_{j=1}^{b_n} \rbraces{\frac{1}{\ell_n} \sum_{i=(j-1)\ell_n+1}^{j\ell_n} \mu_i^2 -\rbraces{\frac{1}{\ell_n} \sum_{i=(j-1)\ell_n+1}^{j\ell_n}\mu_i}^2}^{1/2}.
\end{align*}
For the first of these terms, it holds
\begin{align*}
& \frac{1}{b_n}\sum_{j=1}^{b_n} \E{\abs{\frac{2}{\ell_n}\sum_{i=(j-1)\ell_n+1}^{j\ell_n} \sigma_i\mu_iY_i}^{1/2}}
\leq  \frac{1}{b_n}\sum_{j=1}^{b_n} \rbraces{\E{\abs{\frac{2}{\ell_n}\sum_{i=(j-1)\ell_n+1}^{j\ell_n} \sigma_i\mu_iY_i}^2}}^{1/4}\\
\leq & C \mu_{\sup}^{1/2}\sigma_{\sup}^{1/2} \rbraces{\frac{1}{\ell_n^2} \sum_{i=1}^{\ell_n} \sum_{r=1}^{\ell_n} \abs{\E{Y_iY_r}}}^{1/4}
\leq  C \mu_{\sup}^{1/2}\sigma_{\sup}^{1/2} \frac{1}{\ell_n^{1/4}}\rbraces{1+16\sum_{k=1}^\infty \alpha_Y(k)^{\frac{\delta}{2+\delta}} \norm{Y_1}^2_{2+\delta}}^{1/4} \longrightarrow 0
\end{align*}
by Davydov's covariance inequality. 
Similarly, for the second term, we have 
\begin{align*}
 & \frac{1}{b_n}\sum_{j=1}^{b_n} \E{\abs{2\rbraces{\frac{1}{\ell_n} \sum_{i=(j-1)\ell_n+1}^{j\ell_n}\sigma_iY_i}\rbraces{\frac{1}{\ell_n} \sum_{i=(j-1)\ell_n+1}^{j\ell_n}\mu_i}}^{1/2}}\\
\leq & C \mu_{\sup}^{1/2} \frac{1}{b_n}\sum_{j=1}^{b_n}\E{\abs{\frac{1}{\ell_n} \sum_{i=(j-1)\ell_n+1}^{j\ell_n}\sigma_iY_i}^{1/2}} \\
\leq & C \mu_{\sup}^{1/2}\sigma_{\sup}^{1/2} \frac{1}{\ell_n^{1/4}}\rbraces{1+16\sum_{k=1}^\infty \alpha_Y(k)^{\frac{\delta}{2+\delta}} \norm{Y_1}^2_{2+\delta}}^{1/4} \longrightarrow 0.
\end{align*}
To show the negligibility of the third term, we explicitly require the assumptions made on the mean function $\mu$. We split the problem up into considering a Lipschitz-continuous mean function $\mu_L$ and  a piecewise constant mean function $\mu_J$ (i.e. the jumps). As both parts are additive,
\begin{align*}
&\frac{1}{b_n}\sum_{j=1}^{b_n} \rbraces{\frac{1}{\ell_n} \sum_{i=(j-1)\ell_n+1}^{j\ell_n} \mu_i^2 -\rbraces{\frac{1}{\ell_n} \sum_{i=(j-1)\ell_n+1}^{j\ell_n}\mu_i}^2}^{1/2}\\
= &\frac{1}{b_n}\sum_{j=1}^{b_n} \rbraces{\frac{1}{\ell_n} \sum_{i=(j-1)\ell_n+1}^{j\ell_n} \rbraces{\mu_{L,i}+\mu_{J,i} -\frac{1}{\ell_n} \sum_{r=(j-1)\ell_n+1}^{j\ell_n}(\mu_{L,r}+\mu_{J,r})}^2}^{1/2}\\
\leq & \frac{1}{b_n}\sum_{j=1}^{b_n} \rbraces{\frac{2}{\ell_n} \sum_{i=(j-1)\ell_n+1}^{j\ell_n} \rbraces{\rbraces{\mu_{L,i}-\frac{1}{\ell_n} \sum_{r=(j-1)\ell_n+1}^{j\ell_n}\mu_{L,r}}^2+\rbraces{\mu_{J,i}-\frac{1}{\ell_n} \sum_{r=(j-1)\ell_n+1}^{j\ell_n}\mu_{J,r}}^2}}^{1/2}\\
\leq & \frac{\sqrt{2}}{b_n}\sum_{j=1}^{b_n} \rbraces{\frac{1}{\ell_n} \sum_{i=(j-1)\ell_n+1}^{j\ell_n} \mu_{L,i}^2 -\rbraces{\frac{1}{\ell_n} \sum_{i=(j-1)\ell_n+1}^{j\ell_n}\mu_{L,i}}^2}^{1/2}\\
&+\frac{\sqrt{2}}{b_n}\sum_{j=1}^{b_n} \rbraces{\frac{1}{\ell_n} \sum_{i=(j-1)\ell_n+1}^{j\ell_n} \mu_{J,i}^2 -\rbraces{\frac{1}{\ell_n} \sum_{i=(j-1)\ell_n+1}^{j\ell_n}\mu_{J,i}}^2}^{1/2}.
\end{align*}
Since $\mu_L$ is Lipschitz-continuous, we have 
\begin{align*}
\frac{1}{b_n}\sum_{j=1}^{b_n} \rbraces{\frac{1}{\ell_n} \sum_{i=(j-1)\ell_n+1}^{j\ell_n} \mu_{L,i}^2 -\rbraces{\frac{1}{\ell_n} \sum_{i=(j-1)\ell_n+1}^{j\ell_n}\mu_{L,i}}^2}^{1/2}\leq \sqrt{\frac{2\mu_{\sup}}{b_n}}.
\end{align*}
For the jump part, we obtain 
\begin{align*}
\frac{1}{b_n}\sum_{j=1}^{b_n} \rbraces{\frac{1}{\ell_n} \sum_{i=(j-1)\ell_n+1}^{j\ell_n} \mu_{J,i}^2 -\rbraces{\frac{1}{\ell_n} \sum_{i=(j-1)\ell_n+1}^{j\ell_n}\mu_{J,i}}^2}^{1/2}\leq \frac{k_\mu\Delta_\mu}{b_n}.
\end{align*}

To prove convergence of second term in  \eqref{eq: Appendix Proof of consistency kappahat for Lipschitz},  it suffices to show that 
\begin{align*}
& \E{\abs{\frac{1}{b_n}\sum_{j=1}^{b_n}\rbraces{ \frac{1}{\ell_n} \sum_{i=(j-1)\ell_n+1}^{j\ell_n} \sigma_i^2Y_i^2-\rbraces{\frac{1}{\ell_n}\sum_{i=(j-1)\ell_n+1}^{j\ell_n} \sigma_iY_i}^2}^{1/2}-\frac{1}{b_n}\sum_{j=1}^{b_n}\sigma_{j\ell_n}}}\\
\leq &\frac{1}{b_n}\sum_{j=1}^{b_n} \E{\abs{\rbraces{ \frac{1}{\ell_n} \sum_{i=(j-1)\ell_n+1}^{j\ell_n} \sigma_i^2Y_i^2-\rbraces{\frac{1}{\ell_n}\sum_{i=(j-1)\ell_n+1}^{j\ell_n} \sigma_iY_i}^2}^{1/2}-\rbraces{\sigma_{j\ell_n}^2}^{1/2}}}\\
\leq  &\frac{1}{b_n}\sum_{j=1}^{b_n} \rbraces{\E{\abs{\frac{1}{\ell_n} \sum_{i=(j-1)\ell_n+1}^{j\ell_n} \sigma_i^2Y_i^2-\rbraces{\frac{1}{\ell_n}\sum_{i=(j-1)\ell_n+1}^{j\ell_n} \sigma_iY_i}^2-\sigma_{j\ell_n}^2}}}^{1/2}\\
\leq &\frac{1}{b_n}\sum_{j=1}^{b_n} \left( \rbraces{\E{\abs{\frac{1}{\ell_n}\sum_{i=(j-1)\ell_n+1}^{j\ell_n} \sigma_iY_i}^2}}^{1/2}+ \rbraces{\E{\abs{\frac{1}{\ell_n} \sum_{i=(j-1)\ell_n+1}^{j\ell_n} (\sigma_i^2-\sigma_{j\ell_n}^2)Y_i^2}}}^{1/2} \right.\\
 & \left. + \rbraces{\E{\abs{\sigma_{j\ell_n}^2\frac{1}{\ell_n} \sum_{i=(j-1)\ell_n+1}^{j\ell_n} (Y_i^2-1)}}}^{1/2}  \right)
\end{align*}
converges towards zero. The first of these terms is of order $O(1/\ell_n^{1/2})$, while the second one is of order $O(1/b_n^{1/2})$ due to the piecewise Lipschitz-property of $\sigma$. It remains to prove that
$$\frac{1}{b_n}\sum_{j=1}^{b_n}\rbraces{\E{\abs{\sigma_{j\ell_n}^2\frac{1}{\ell_n} \sum_{i=(j-1)\ell_n+1}^{j\ell_n} (Y_i^2-1)}}}^{1/2}\leq \sigma_{\sup} \rbraces{\E{\abs{\frac{1}{\ell_n} \sum_{i=1}^{\ell_n} (Y_i^2-1)}}}^{1/2}$$
converges towards zero as well. Since $\frac{1}{\ell_n} \sum_{i=1}^{\ell_n} (Y_i^2-1)\pConv 0$, by Theorem 5.4 in \citep{Billingsley.1968}, it suffices to prove uniform integrability to additionally obtain 
$ \E{\abs{\frac{1}{\ell_n} \sum_{i=1}^{\ell_n} (Y_i^2-1)}}\rightarrow 0$. Hence, we need to prove that for some $\eta>0$, it holds 
$$\sup_n \E{\abs{\frac{1}{\ell_n} \sum_{i=1}^{\ell_n} (Y_i^2-1)}^{1+\eta}}<\infty.$$
The Rosenthal-type inequality from Theorem \ref{Literature: Thm Rosenthal} yields
\begin{align*}
\E{\abs{\frac{1}{\ell_n} \sum_{i=1}^{\ell_n} (Y_i^2-1)}^{1+\eta}} \leq \ell_n^{-(1+\eta) } \sum_{i=1}^{\ell_n} \rbraces{\E{\abs{Y_i^2-1}^{1+\eta+\varepsilon}}}^{(1+\eta)/(1+\eta+\varepsilon)} \longrightarrow 0
\end{align*}
for some $\varepsilon>0$ given the conditions $\E{\abs{Y_1}^{2(1+\eta+\varepsilon)}}<\infty$ and $\sum_{k=1}^{\infty} \alpha_{Y^2}(k)^{\varepsilon/(2+\varepsilon)}<\infty$. The first condition is fulfilled, e.g., if we set $\varepsilon=\eta=\delta/4$. The second condition then translates to 
$\sum_{k=1}^{\infty} \alpha_{Y}(k)^{\delta/(8+\delta)}<\infty$
due to $\alpha_{Y^2}(k)\leq \alpha_Y(k)$. Inserting $\alpha_Y(k)\leq C k^{-\rho(2+\delta)(1+\delta)/\delta^2}$ leads to the condition $\rho>\frac{\delta(8+\delta)}{(2+\delta)(1+\delta)}$. Alternatively, given the condition $\E{\abs{Y_1}^{4+2\delta}}<\infty$, one can proceed as in i). 
The third term in  \eqref{eq: Appendix Proof of consistency kappahat for Lipschitz}  is of order $O(1/b_n)$, which can be seen as in i). 
\end{proof}

We now turn towards the behaviour of the estimators $\hat{\kappa}_{\tilde{X}}(n)$ and  $\hat{\kappa}(n)$ under the alternative of a general c\`adl\`ag mean function. 
\begin{lemma} \label{Lemma: convergence estimator kappahat kappaY under A}
 Assume there is  a $0< \delta \leq 1$ and a $\rho>(\delta^2+4\delta)/(2+2\delta)\vee 1$ such that $\mathbb{E}\big({\abs{Y_1}^{2+\delta}}\big)<\infty$ and $\alpha_Y(k)\leq Ck^{-\rho (2+\delta)(1+\delta)/\delta^2}$.  
  Then it holds under the alternative of a general c\`adl\`ag mean function that 
$$\hat{\kappa}_{\tilde{X}}(n)= O_{\Pb}\left(\elln^{1/2}\right).$$
If additionally either $\mathbb{E}\big({\abs{Y_1}^{4+2\delta}}\big)<\infty$ or $\rho>\frac{\delta(8+\delta)}{(2+\delta)(1+\delta)}$, it holds
$$\hat{\kappa}(n)= \frac{\hat{\kappa}_{\tilde{X}}(n)}{\frac{1}{b_n}\sum_{j=1}^{b_n}\hat{\sigma}_{j,n}}= O_{\Pb}\left(\elln^{1/2}\right).$$
\end{lemma}

\begin{proof}[Proof of Lemma \ref{Lemma: convergence estimator kappahat kappaY under A}]
The long run variance estimator $\hat{\kappa}_{\tilde{X}}$ for a non-zero mean function can be split up into the estimator $\hat{\kappa}_{\tilde{X}; \mathbb{H}}(n)$ we obtain under the hypothesis, i.e. for $\mu=0$, and an additional term containing the mean function $\mu$. We have
\begin{align*}
&\abs{\hat{\kappa}_{\tilde{X}}(n)-\kappa_Y\int_0^1\sigma(x)\mathrm{d}x}
\leq  \abs{\hat{\kappa}_{\tilde{X}; \mathbb{H}}(n)-\kappa_Y\int_0^1\sigma(x)\mathrm{d}x}\\
& +  \sqrt{\frac{2c_0}{1+2c_0}}\sqrt{\frac{\pi}{2}}\frac{1}{\bn} \sum_{j=1}^{\bn} \frac{1}{\sqrt{\elln}} \abs{\sum_{i=(j-1)\elln+1}^{j\elln} \rbraces{\mu_i-\frac{1}{2c_0\elln}\rbraces{\sum_{r=(j-c_0-1)\elln+1}^{(j-1)\elln}\mu_r+\sum_{r=j\elln+1}^{(j+c_0)\elln}\mu_r}}}.
\end{align*}
We considered the first difference in Proposition \ref{Prop: Consistency alpha-mix} and showed that it is at most of order $O_{\mathbb{P}}(1)$ (in fact, we have shown more). The second term does not converge towards zero for an arbitrary mean function, we can only bound the expression by $\sqrt{2c_0/(1+2c_0)}\sqrt{2\pi} \mu_{\sup} \elln^{1/2}$. 

To prove the second assertion $\hat{\kappa}(n)=O_{\Pb}\left(\elln^{1/2}\right)$, we have to show that the sample mean of the estimated variances is bounded away from zero.  It holds
\begin{align*}
 &\frac{1}{b_n}\sum_{j=1}^{b_n}\hat{\sigma}_{j,n}
 =  \frac{1}{b_n}\sum_{j=1}^{b_n} \left(\frac{1}{\ell_n} \sum_{i=(j-1)\ell_n+1}^{j\ell_n} \sigma_i^2Y_i^2-\rbraces{\frac{1}{\ell_n}\sum_{i=(j-1)\ell_n+1}^{j\ell_n} \sigma_iY_i}^2  + \frac{2}{\ell_n}\sum_{i=(j-1)\ell_n+1}^{j\ell_n} \sigma_i\mu_iY_i \right. \\
&  \left. - 2\rbraces{\frac{1}{\ell_n} \sum_{i=(j-1)\ell_n+1}^{j\ell_n}\sigma_iY_i}\rbraces{\frac{1}{\ell_n} \sum_{i=(j-1)\ell_n+1}^{j\ell_n}\mu_i} +\frac{1}{\ell_n} \sum_{i=(j-1)\ell_n+1}^{j\ell_n} \mu_i^2- \rbraces{\frac{1}{\ell_n} \sum_{i=(j-1)\ell_n+1}^{j\ell_n}\mu_i}^2  \right)^{1/2}\\
\geq &  \frac{1}{b_n}\sum_{j=1}^{b_n} \left(\frac{1}{\ell_n} \sum_{i=(j-1)\ell_n+1}^{j\ell_n} \sigma_i^2Y_i^2-\rbraces{\frac{1}{\ell_n}\sum_{i=(j-1)\ell_n+1}^{j\ell_n} \sigma_iY_i}^2+\frac{1}{\ell_n} \sum_{i=(j-1)\ell_n+1}^{j\ell_n} \mu_i^2- \rbraces{\frac{1}{\ell_n} \sum_{i=(j-1)\ell_n+1}^{j\ell_n}\mu_i}^2 \right)^{1/2} \\
&  -\frac{1}{b_n}\sum_{j=1}^{b_n} \abs{ \frac{2}{\ell_n}\sum_{i=(j-1)\ell_n+1}^{j\ell_n} \sigma_i\mu_iY_i}^{1/2} - \frac{1}{b_n}\sum_{j=1}^{b_n}\abs{2\rbraces{\frac{1}{\ell_n} \sum_{i=(j-1)\ell_n+1}^{j\ell_n}\sigma_iY_i}\rbraces{\frac{1}{\ell_n} \sum_{i=(j-1)\ell_n+1}^{j\ell_n}\mu_i}}^{1/2} \\
\geq &  \frac{1}{b_n}\sum_{j=1}^{b_n} \left(\frac{1}{\ell_n} \sum_{i=(j-1)\ell_n+1}^{j\ell_n} \sigma_i^2Y_i^2-\rbraces{\frac{1}{\ell_n}\sum_{i=(j-1)\ell_n+1}^{j\ell_n} \sigma_iY_i}^2 \right)^{1/2} \\
&  -\frac{1}{b_n}\sum_{j=1}^{b_n} \abs{ \frac{2}{\ell_n}\sum_{i=(j-1)\ell_n+1}^{j\ell_n} \sigma_i\mu_iY_i}^{1/2} - \frac{1}{b_n}\sum_{j=1}^{b_n}\abs{2\rbraces{\frac{1}{\ell_n} \sum_{i=(j-1)\ell_n+1}^{j\ell_n}\sigma_iY_i}\rbraces{\frac{1}{\ell_n} \sum_{i=(j-1)\ell_n+1}^{j\ell_n}\mu_i}}^{1/2}.
 \end{align*}
The last two terms converge in $L^1$ towards zero, irrespective of the mean function $\mu$, as shown in the proof of Proposition \ref{Prop: convergence estimator kappahat kappaY under H}, while the first was shown to converge in $L^1$ towards $\int_0^1\sigma(x)\mathrm{d}x\geq \sigma_0>0$ given one of the conditions $\rho>\frac{\delta(8+\delta)}{(2+\delta)(1+\delta)}$ or $\E{\abs{Y_1}^{4+2\delta}}<\infty$.
\end{proof}

\subsubsection{Estimating the centring term}

\begin{proof}[Proof of Proposition \ref{Prop: Convergence estimated centring term}]
As shown in the proof of Proposition \ref{Prop: Replacing centring term}, it holds for the Riemann-approximation that 
$$\sqrt{b_n}\abs{\frac{1}{b_n(b_n-1)}\sum_{1\leq j\neq k \leq b_n} \sqrt{\sigma_{j\ell_n}^2+\sigma_{k\ell_n}^2}-\int_0^1\int_0^1 \sqrt{\sigma^2(x)+\sigma^2(y)} \mathrm{d}x\mathrm{d}y}\longrightarrow 0,$$
such that it suffices to prove 
$$\sqrt{b_n}\abs{\frac{1}{b_n(b_n-1)}\sum_{1\leq j\neq k \leq b_n} \sqrt{\hat{\sigma}_{j,n}^2+\hat{\sigma}_{k,n}^2} -  \frac{1}{b_n(b_n-1)}\sum_{1\leq j\neq k \leq b_n} \sqrt{\sigma_{j\ell_n}^2+\sigma_{k\ell_n}^2} }\overset{L^1}{\rightarrow} \;0.$$
We split the sum up via
\begin{align*}
&\sqrt{b_n}\abs{\frac{1}{b_n(b_n-1)}\sum_{1\leq j\neq k \leq b_n} \sqrt{\hat{\sigma}_{j,n}^2+\hat{\sigma}_{k,n}^2} -  \frac{1}{b_n(b_n-1)}\sum_{1\leq j\neq k \leq b_n} \sqrt{\sigma_{j\ell_n}^2+\sigma_{k\ell_n}^2} }\\
\leq & \frac{\sqrt{b_n}}{b_n(b_n-1)}\sum_{1\leq j\neq k \leq b_n} \abs{\sqrt{\hat{\sigma}_{j,n}^2+\hat{\sigma}_{k,n}^2} - \sqrt{\sigma_{j\ell_n}^2+\sigma_{k\ell_n}^2} } \1_{\{\abs{\sigma_{j\ell_n}^2-\hat{\sigma}_{j,n}^2}\leq \sigma_0^2/2\}}\\
& + \frac{\sqrt{b_n}}{b_n(b_n-1)}\sum_{1\leq j\neq k \leq b_n} \abs{\sqrt{\hat{\sigma}_{j,n}^2+\hat{\sigma}_{k,n}^2} - \sqrt{\sigma_{j\ell_n}^2+\sigma_{k\ell_n}^2} } \1_{\{\abs{\sigma_{k\ell_n}^2-\hat{\sigma}_{k,n}^2}\leq \sigma_0^2/2\}}\\
& +  \frac{\sqrt{b_n}}{b_n(b_n-1)}\sum_{1\leq j\neq k \leq b_n} \abs{\sqrt{\hat{\sigma}_{j,n}^2+\hat{\sigma}_{k,n}^2} - \sqrt{\sigma_{j\ell_n}^2+\sigma_{k\ell_n}^2} } \1_{\{\abs{\sigma_{j\ell_n}^2-\hat{\sigma}_{j,n}^2}> \sigma_0^2/2\}} \1_{\{\abs{\sigma_{k\ell_n}^2-\hat{\sigma}_{k,n}^2}>\sigma_0^2/2\}}.
\end{align*}
If $\abs{\sigma_{j\ell_n}^2-\hat{\sigma}_{j,n}^2}\leq \sigma_0^2/2$, we have $\hat{\sigma}_{j,n}^2+\hat{\sigma}_{k,n}^2\geq\sigma_0^2/2$ since $\sigma^2(x)\geq \sigma_0^2>0$ for all $x\in [0,1]$. Hence, for the first term, the square root is Lipschitz-continuous, yielding 
\begin{align*}
& \frac{\sqrt{b_n}}{b_n(b_n-1)}\sum_{1\leq j\neq k \leq b_n}\abs{ \sqrt{\hat{\sigma}_{j,n}^2+\hat{\sigma}_{k,n}^2} - \sqrt{\sigma_{j\ell_n}^2+\sigma_{k\ell_n}^2} } \1_{\{\abs{\sigma_{j\ell_n}^2-\hat{\sigma}_{j,n}^2}\leq \sigma_0^2/2\}}\\
\leq & \frac{\sqrt{b_n}}{b_n(b_n-1)}\sum_{1\leq j\neq k \leq b_n}\rbraces{  \abs{\sigma_{j\ell_n}^2-\hat{\sigma}_{j,n}^2}+ \abs{\sigma_{k\ell_n}^2-\hat{\sigma}_{k,n}^2} } \1_{\{\abs{\sigma_{j\ell_n}^2-\hat{\sigma}_{j,n}^2}\leq \sigma_0^2/2\}}\\
\leq & \frac{\sqrt{b_n}}{b_n(b_n-1)}\sum_{1\leq j\neq k \leq b_n}\rbraces{  \abs{\sigma_{j\ell_n}^2-\hat{\sigma}_{j,n}^2}+ \abs{\sigma_{k\ell_n}^2-\hat{\sigma}_{k,n}^2} }\\
\leq & \frac{2}{\sqrt{b_n}} \sum_{j=1}^{b_n} \abs{\sigma_{j\ell_n}^2-\hat{\sigma}_{j,n}^2}.
\end{align*}

 It holds 
\begin{align*}
& \frac{1}{\sqrt{b_n}} \sum_{j=1}^{b_n} \E{\abs{\sigma_{j\ell_n}^2-\hat{\sigma}_{j,n}^2}}\\
 = & \frac{1}{\sqrt{b_n}} \sum_{j=1}^{b_n} \E{\abs{\sigma_{j\ell_n}^2-\frac{1}{\ell_n} \sum_{i=(j-1)\ell_n+1}^{j\ell_n} \sigma_i^2Y_i^2+\rbraces{\frac{1}{\ell_n}\sum_{i=(j-1)\ell_n+1}^{j\ell_n} \sigma_iY_i}^2}}\\
\leq &  \frac{1}{\sqrt{b_n}} \sum_{j=1}^{b_n} \left(\E{\abs{\sigma_{j\ell_n}^2-\frac{1}{\ell_n} \sum_{i=(j-1)\ell_n+1}^{j\ell_n} \sigma_{j\ell_n}^2Y_i^2 }}\right.\\
& \left. +\E{\abs{\frac{1}{\ell_n} \sum_{i=(j-1)\ell_n+1}^{j\ell_n} \sigma_i^2Y_i^2-\frac{1}{\ell_n} \sum_{i=(j-1)\ell_n+1}^{j\ell_n} \sigma_{j\ell_n}^2Y_i^2}}+\E{\abs{\frac{1}{\ell_n}\sum_{i=(j-1)\ell_n+1}^{j\ell_n} \sigma_iY_i}^2}\right)\\
\leq & \frac{1}{\sqrt{b_n}} \sum_{j=1}^{b_n} \left( \sigma_{\sup}^2\E{\abs{\frac{1}{\ell_n} \sum_{i=(j-1)\ell_n+1}^{j\ell_n} (Y_i^2-1) }}+\E{\abs{\frac{1}{\ell_n} \sum_{i=(j-1)\ell_n+1}^{j\ell_n} (\sigma_i^2-\sigma_{j\ell_n}^2)Y_i^2}}\right.\\
& \left. +\frac{1}{\ell_n^2}\sum_{i=(j-1)\ell_n+1}^{j\ell_n}\sum_{r=(j-1)\ell_n+1}^{j\ell_n} \sigma_{\sup}^2\abs{\cov{Y_i}{Y_r}}\right)\\
\leq & \sqrt{b_n} \left( \sigma_{\sup}^2\E{\abs{\frac{1}{\ell_n} \sum_{i=1}^{\ell_n} (Y_i^2-1) }}+2\sigma_{sup} \frac{1+k_\sigma\Delta_\sigma}{b_n}\E{\abs{\frac{1}{\ell_n} \sum_{i=1}^{\ell_n}Y_i^2}}\right.
  \left. +\sigma_{\sup}^2\frac{1}{\ell_n^2}\sum_{i=1}^{\ell_n}\sum_{r=1}^{\ell_n} \abs{\cov{Y_i}{Y_r}}\right)\\
\leq  & C \sqrt{b_n} \E{\abs{\frac{1}{\ell_n} \sum_{i=1}^{\ell_n} (Y_i^2-1) }^2}^{1/2}+ \frac{C}{\sqrt{b_n}} \E{Y_1^2}
+ C\frac{\sqrt{b_n}}{\ell_n} \rbraces{\E{Y_1^2}+ 2\sum_{k=1}^\infty \abs{\cov{Y_1}{Y_{k+1}}}} \\
\leq  & C\frac{\sqrt{b_n}}{\sqrt{\ell_n}} \rbraces{\E{\abs{Y_1^2-1}^2}+ 2\sum_{k=1}^{\infty} \abs{\cov{Y_1^2-1}{Y_{k+1}^2-1}}}^{1/2}+ \frac{C}{\sqrt{b_n}} \E{Y_1^2}  \\
& +  C\frac{\sqrt{b_n}}{\ell_n} \rbraces{\E{Y_1^2}+16\norm{Y_1}_{2+\delta}^2 \sum_{k=1}^\infty \alpha_{Y}(k)^{\delta/(2+\delta)}}\\
\leq & C \frac{\sqrt{b_n}}{\sqrt{\ell_n}} \rbraces{\E{\abs{Y_1^2-1}^2}+16\norm{Y_1^2-1}_{2+\delta}^2 \sum_{k=1}^{\infty} \alpha_{Y^2}(k)^{\delta/(2+\delta)}}^{1/2}+ \frac{C}{\sqrt{b_n}} \E{Y_1^2}\\
& +  C\frac{\sqrt{b_n}}{\ell_n} \rbraces{\E{Y_1^2}+16\norm{Y_1}_{2+\delta}^2 \sum_{k=1}^\infty \alpha_{Y}(k)^{\delta/(2+\delta)}},
\end{align*}
where we used Davydov's covariance inequality in the last two steps. Since $\alpha_{Y^2}(k)\leq  \alpha_Y(k)$ for all $k\in \N$ and since we assumed $\sum_{k=1}^\infty \alpha_{Y}(k)^{\delta/(2+\delta)}<\infty$ as well as finite $(4+2\delta)$-moments of $Y$, the above expression converges towards zero. 

The second term is identical to the first. 
For the third term, we only have 
\begin{align*}
& \frac{\sqrt{b_n}}{b_n(b_n-1)}\sum_{1\leq j\neq k \leq b_n}\abs{ \sqrt{\hat{\sigma}_{j,n}^2+\hat{\sigma}_{k,n}^2} - \sqrt{\sigma_{j\ell_n}^2+\sigma_{k\ell_n}^2} } \1_{\{\abs{\sigma_{j\ell_n}^2-\hat{\sigma}_{j,n}^2}> \sigma_0^2/2\}} \1_{\{\abs{\sigma_{k\ell_n}^2-\hat{\sigma}_{k,n}^2}> \sigma_0^2/2\}}\\
\leq & \frac{\sqrt{b_n}}{b_n(b_n-1)}\sum_{1\leq j\neq k \leq b_n}\rbraces{\sqrt{\abs{\sigma_{j\ell_n}^2-\hat{\sigma}_{j,n}^2}}+\sqrt{\abs{\sigma_{k\ell_n}^2-\hat{\sigma}_{k,n}^2}} } \1_{\{\abs{\sigma_{j\ell_n}^2-\hat{\sigma}_{j,n}^2}> \sigma_0^2/2\}} \1_{\{\abs{\sigma_{k\ell_n}^2-\hat{\sigma}_{k,n}^2}> \sigma_0^2/2\}}\\
\leq &  \frac{2}{\sqrt{b_n}} \sum_{j=1}^{b_n} \sqrt{\abs{\sigma_{j\ell_n}^2-\hat{\sigma}_{j,n}^2}}\1_{\{\abs{\sigma_{j\ell_n}^2-\hat{\sigma}_{j,n}^2}> \sigma_0^2/2\}}.
\end{align*}
Moreover, 
\begin{align*}
& \frac{1}{\sqrt{b_n}} \sum_{j=1}^{b_n} \E{\sqrt{\abs{\sigma_{j\ell_n}^2-\hat{\sigma}_{j,n}^2}}\1_{\{\abs{\sigma_{j\ell_n}^2-\hat{\sigma}_{j,n}^2}> \sigma_0^2/2\}}}\\
\leq &  \frac{1}{\sqrt{b_n}} \sum_{j=1}^{b_n} \rbraces{\E{\abs{\sigma_{j\ell_n}^2-\hat{\sigma}_{j,n}^2}}}^{1/2}\Pb\rbraces{\abs{\sigma_{j\ell_n}^2-\hat{\sigma}_{j,n}^2}> \sigma_0^2/2}^{1/2}\\
\leq  & \frac{\sqrt{2}}{\sigma_0 }  \frac{1}{\sqrt{b_n}} \sum_{j=1}^{b_n} \rbraces{\E{\abs{\sigma_{j\ell_n}^2-\hat{\sigma}_{j,n}^2}}}^{1/2}\rbraces{\E{\abs{\sigma_{j\ell_n}^2-\hat{\sigma}_{j,n}^2}}}^{1/2}= C \frac{1}{\sqrt{b_n}} \sum_{j=1}^{b_n} \E{\abs{\sigma_{j\ell_n}^2-\hat{\sigma}_{j,n}^2}},  
\end{align*}
which converges towards zero as shown above.
\end{proof}

\begin{lemma}{
\label{Lemma: estimated centring term alternative} Under the alternative of a general c\`adl\`ag function $\mu$, it holds 
$$\frac{1}{b_n(b_n-1)}\sum_{1\leq j\neq k \leq b_n} \sqrt{\hat{\sigma}_{j,n}^2+\hat{\sigma}_{k,n}^2}\cdot \sqrt{\frac{2}{\pi}}=O_{\mathbb{P}}(1).$$
}\end{lemma}

\begin{proof}[Proof of Lemma \ref{Lemma: estimated centring term alternative}]
It holds 
\begin{align*}
& \frac{1}{b_n(b_n-1)}\sum_{1\leq j\neq k \leq b_n} \sqrt{\hat{\sigma}_{j,n}^2+\hat{\sigma}_{k,n}^2} \cdot \sqrt{\frac{2}{\pi}}
\leq C \frac{1}{b_n} \sum_{j=1}^{b_n} \hat{\sigma}_{j,n} 
\leq C \rbraces{ \frac{1}{b_n} \sum_{j=1}^{b_n} \hat{\sigma}_{j,n}^2 }^{1/2}
\end{align*}
and 
\begin{align*}
 \frac{1}{b_n} \sum_{j=1}^{b_n} \hat{\sigma}_{j,n}^2 
\leq & 2 \frac{1}{b_n} \sum_{j=1}^{b_n} \frac{1}{\ell_n}\sum_{i=(j-1)\ell_n+1}^{j\ell_n}\rbraces{\rbraces{\sigma_iY_i -\frac{1}{\ell_n} \sum_{r=(j-1)\ell_n+1}^{j\ell_n} \sigma_r Y_r}^2+ \rbraces{\mu_i -\frac{1}{\ell_n} \sum_{r=(j-1)\ell_n+1}^{j\ell_n} \mu_r}^2}\\
\leq &  C \rbraces{\frac{1}{b_n} \sum_{j=1}^{b_n} \rbraces{\frac{1}{\ell_n}\sum_{i=(j-1)\ell_n+1}^{j\ell_n} \sigma_i^2Y_i^2- \rbraces{\frac{1}{\ell_n}\sum_{i=(j-1)\ell_n+1}^{j\ell_n} \sigma_iY_i}^2 } + 4\mu_{\sup}^2} \\
\leq & C \rbraces{\sigma_{\sup}^2 \frac{1}{n}\sum_{i=1}^n Y_i^2 +4\mu_{\sup}^2} \pConv C\rbraces{\sigma_{\sup}^2\E{Y_1^2}+4\mu_{\sup}^2}.
\end{align*}

\end{proof}

\subsubsection{Estimating $\psi^2$}

\begin{proof}[Proof of Proposition \ref{Prop: hatpsin psin approx H}]
Recall the definition of 
$$\psi_n^2= 4\frac{1}{b_n} \sum_{j=1}^{b_n}\E{\abs{\frac{1}{b_n-1} \sum_{\substack{1\leq k\leq b_n\\k\neq j}} \rbraces{\E{h\rbraces{\sigma_{j\ell_n}Z,\sigma_{k\ell_n}Z'}\middle|Z}-\E{h\rbraces{\sigma_{j\ell_n}Z, \sigma_{k\ell_n}Z'}}}}^2}$$
from the proof of Proposition \ref{Prop: CLT U2}, where we showed that $\abs{\psi_n^2-\psi^2}=O(1/b_n)$. It remains to prove $\abs{\hat{\psi}_n^2-\psi^2_n}\overset{L^1}{\rightarrow} \;0$. Note that for arbitrary $x_i$ and $y_i$ in $\R$, $1\leq i\leq b_n$, it holds

\begin{align*}
& \abs{f_n(x_1, \ldots, x_{b_n})-f_n(y_1, \ldots, y_{b_n})} \\
=& \left|\frac{1}{b_n} \sum_{j=1}^{b_n} \left( \E{\abs{\frac{1}{b_n-1}\sum_{\substack{1\leq k \leq b_n\\ k\neq j}} \rbraces{\E{\abs{x_jZ-x_kZ'}|Z}-\E{\abs{x_jZ-x_kZ'}}}}^2} \right. \right. \\
& \left. \left.- \E{\abs{\frac{1}{b_n-1}\sum_{\substack{1\leq k \leq b_n\\ k\neq j}} \rbraces{\E{\abs{y_jZ-y_kZ'}|Z}-\E{\abs{y_jZ-y_kZ'}}}}^2}\right)\right|\\
= &  \frac{1}{b_n} \sum_{j=1}^{b_n} \mathbb{E}  \bigg(  \bigg|\frac{1}{b_n-1}\sum_{\substack{1\leq k \leq b_n\\ k\neq j}} \left(\E{\abs{x_jZ-x_kZ'}|Z}+\E{\abs{y_jZ-y_kZ'}|Z} \right. \\
 & \left.\left. \left. -\E{\abs{x_jZ-x_kZ'}}-\E{\abs{y_jZ-y_kZ'}}\right)\right. \bigg|\right. \\
 &  \left. \left. \cdot  \bigg |\frac{1}{b_n-1}\sum_{\substack{1\leq k \leq b_n\\ k\neq j}} \left( \E{\abs{x_jZ-x_kZ'}|Z}-\E{\abs{y_jZ-y_kZ'}|Z}  \right.\right.\right.\\
 & \left.\left.\left. -\E{\abs{x_jZ-x_kZ'}}+\E{\abs{y_jZ-y_kZ'}}\right) \right.  \bigg|\right.  \bigg)\\
 \leq &  \frac{1}{b_n} \sum_{j=1}^{b_n} \left.  \bigg (\mathbb{E}  \bigg ( \bigg|\frac{1}{b_n-1}\sum_{\substack{1\leq k \leq b_n\\ k\neq j}} \left(\E{\abs{x_jZ-x_kZ'}|Z}+\E{\abs{y_jZ-y_kZ'}|Z} \right.\right.\\
 & \left.\left. \left. \left. -\E{\abs{x_jZ-x_kZ'}}-\E{\abs{y_jZ-y_kZ'}}\right)\right.  \bigg |^2\right.  \bigg)\right.  \bigg)^{1/2} \\
 &   \cdot \left.  \bigg(\mathbb{E}  \bigg (  \bigg |\frac{1}{b_n-1}\sum_{\substack{1\leq k \leq b_n\\ k\neq j}} \left(\E{\abs{x_jZ-x_kZ'}|Z}-\E{\abs{y_jZ-y_kZ'}|Z} \right. \right. \\
 & \left.\left. \left. \left. -\E{\abs{x_jZ-x_kZ'}}+\E{\abs{y_jZ-y_kZ'}}\right)\right.  \bigg |^2\right.  \bigg)\right.  \bigg)^{1/2}.
\end{align*}
We will deal with both expectations separately, 
\begin{align*}
& \left.  \bigg \lVert \frac{1}{b_n-1}\sum_{\substack{1\leq k \leq b_n\\ k\neq j}} \left(\E{\abs{x_jZ-x_kZ'}|Z}+\E{\abs{y_jZ-y_kZ'}|Z} \right. \right.\\
 & \left.\left.  -\E{\abs{x_jZ-x_kZ'}}-\E{\abs{y_jZ-y_kZ'}}\right)\right.  \bigg \rVert_2\\
\leq &  \bigg \lVert (\abs{x_j}+\abs{y_j}) (E{\abs{Z}}+\abs{Z})+ \frac{1}{b_n-1}\sum_{\substack{1\leq k \leq b_n\\ k\neq j}} 2(\abs{x_k}+\abs{y_k})\E{\abs{Z}}\bigg \rVert _2\\
\leq & C \bigg (\abs{x_j}+\abs{y_j} + \frac{1}{b_n-1} \sum_{\substack{1\leq k\leq b_n \\ k\neq j}} \rbraces{\abs{x_k}+\abs{y_k}}\bigg )
\end{align*}
and, similarly, 
\begin{align*}
& \left. \bigg \lVert \frac{1}{b_n-1}\sum_{\substack{1\leq k \leq b_n\\ k\neq j}} \left(\E{\abs{x_jZ-x_kZ'}|Z}-\E{\abs{y_jZ-y_kZ'}|Z} \right.\right.\\
 & \left.\left.-\E{\abs{x_jZ-x_kZ'}}+\E{\abs{y_jZ-y_kZ'}}\right)\right. \bigg \rVert_2\\
\leq  & \bigg \lVert \frac{1}{b_n-1}\sum_{\substack{1\leq k \leq b_n\\ k\neq j}} \rbraces{\abs{(x_j-y_j)Z}+\E{\abs{(x_k-y_k)Z}} +\E{\abs{(x_j-y_j)Z}}+\E{\abs{(x_k-y_k)Z}}}\bigg \rVert _2\\
\leq & C \bigg ( \abs{x_j-y_j}+ \frac{1}{b_n-1}\sum_{\substack{1\leq k \leq b_n\\ k\neq j}} \abs{x_k-y_k}\bigg ).
\end{align*}
Hence, for $x_i,y_i\geq 0$, one obtains
\begin{align*}
& \abs{f_n(x_1, \ldots, x_{b_n})-f_n(y_1, \ldots, y_{b_n})} \\
 \leq & C \frac{1}{b_n} \sum_{j=1}^{b_n}  \rbraces{x_j+y_j + \frac{1}{b_n-1} \sum_{\substack{1\leq k\leq b_n \\ k\neq j}} \rbraces{x_k+y_k}}
 \rbraces{\abs{x_j-y_j}+ \frac{1}{b_n-1}\sum_{\substack{1\leq k \leq b_n\\ k\neq j}} \abs{x_k-y_k}}\\
 \leq & C\frac{1}{b_n} \sum_{j=1}^{b_n} \abs{x_j^2-y_j^2}+C \rbraces{\frac{b_n^2}{(b_n-1)^2}+2\frac{b_n}{b_n-1}} \rbraces{\frac{1}{b_n}\sum_{j=1}^{b_n} \rbraces{x_j+y_j}}\rbraces{\frac{1}{b_n}\sum_{j=1}^{b_n} \abs{x_j-y_j}}
\end{align*}
and consequently, 
\begin{align*}
& \abs{\hat{\psi}_n^2-\psi_n^2}\\
\leq &C\frac{1}{b_n} \sum_{j=1}^{b_n} \abs{\hat{\sigma}_{j,n}^2-\sigma_{j\ell_n}^2}+C \rbraces{\frac{b_n^2}{(b_n-1)^2}+2\frac{b_n}{b_n-1}} \rbraces{\frac{1}{b_n}\sum_{j=1}^{b_n} \rbraces{\hat{\sigma}_{j,n}+\sigma_{j\ell_n}}}\rbraces{\frac{1}{b_n}\sum_{j=1}^{b_n} \abs{\hat{\sigma}_{j,n}-\sigma_{j\ell_n}}}.
\end{align*}
For the first term, we obtain 
$$\frac{1}{b_n} \sum_{j=1}^{b_n} \E{\abs{\hat{\sigma}_{j,n}^2-\sigma_{j\ell_n}^2}}\rightarrow 0, $$
which can be seen as in  Propositions \ref{Prop: convergence estimator kappahat kappaY under H} ii) and \ref{Prop: Convergence estimated centring term}.
 Turning towards the second term, it holds
\begin{align*}
& \E{ \rbraces{\frac{1}{b_n}\sum_{j=1}^{b_n} \rbraces{\hat{\sigma}_{j,n}+\sigma_{j\ell_n}}} \cdot\rbraces{\frac{1}{b_n}\sum_{j=1}^{b_n} \abs{\hat{\sigma}_{j,n}-\sigma_{j\ell_n}}}}\\
\leq &\rbraces{\frac{1}{b_n}\sum_{j=1}^{b_n} \sigma_{j\ell_n}} \cdot \rbraces{\frac{1}{b_n}\sum_{j=1}^{b_n} \E{\abs{\hat{\sigma}_{j,n}-\sigma_{j\ell_n}}}}
+ \E{\rbraces{\frac{1}{b_n}\sum_{j=1}^{b_n} \hat{\sigma}_{j,n}} \cdot\rbraces{\frac{1}{b_n}\sum_{j=1}^{b_n} \abs{\hat{\sigma}_{j,n}-\sigma_{j\ell_n}}}}\\
\leq & \sigma_{\sup} \frac{1}{b_n}\sum_{j=1}^{b_n} \E{\abs{\hat{\sigma}_{j,n}^2-\sigma_{j\ell_n}^2}^{1/2}} +  \norm{\frac{1}{b_n}\sum_{j=1}^{b_n} \hat{\sigma}_{j,n}}_2\cdot\norm{\frac{1}{b_n}\sum_{j=1}^{b_n} \abs{\hat{\sigma}_{j,n}-\sigma_{j\ell_n}}}_2,
\end{align*}
where we omitted the asymptotically negligible factor $ \rbraces{\frac{b_n^2}{(b_n-1)^2}+2\frac{b_n}{b_n-1}}$. Moreover, it holds 
$$\frac{1}{b_n}\sum_{j=1}^{b_n} \E{\abs{\hat{\sigma}_{j,n}^2-\sigma_{j\ell_n}^2}^{1/2}}\leq \frac{1}{b_n}\sum_{j=1}^{b_n} \rbraces{\E{\abs{\hat{\sigma}_{j,n}^2-\sigma_{j\ell_n}^2}}}^{1/2}\rightarrow 0,$$
which can again be shown analogously to  Propositions \ref{Prop: convergence estimator kappahat kappaY under H} ii) and \ref{Prop: Convergence estimated centring term}.
The first factor of the second term above is bounded due to 
 \begin{align*}
\norm{\frac{1}{b_n}\sum_{j=1}^{b_n} \hat{\sigma}_{j,n}}_2
 \leq \frac{1}{b_n}\sum_{j=1}^{b_n} \norm{ \hat{\sigma}_{j,n}}_2 
 \leq  \frac{1}{b_n}\sum_{j=1}^{b_n} \rbraces{\E{\frac{1}{\ell_n}\sum_{i=(j-1)\ell_n+1}^{j\ell_n} \sigma_i^2Y_i^2}}^{1/2} \leq \sigma_{\sup} \norm{Y_1}_2, 
 \end{align*}
 whereas the second factor converges towards zero as it can once more be bounded by 
 \begin{align*}
\norm{\frac{1}{b_n}\sum_{j=1}^{b_n} \abs{\hat{\sigma}_{j,n}-\sigma_{j\ell_n}}}_{2}
\leq \frac{1}{b_n}\sum_{j=1}^{b_n} \norm{\hat{\sigma}_{j,n}-\sigma_{j\ell_n}}_2
 \leq \frac{1}{b_n}\sum_{j=1}^{b_n} \rbraces{\E{ \abs{\hat{\sigma}_{j,n}^2-\sigma_{j\ell_n}^2}}}^{1/2}.
 \end{align*}
\end{proof}

\begin{lemma}{
\label{Lemma: hatpsin under A}
 Under the alternative of a general c\`adl\`ag function $\mu$, it holds 
$$ \hat{\psi}_n=O_{\mathbb{P}}(1).$$
}\end{lemma}

\begin{proof}[Proof of Lemma \ref{Lemma: hatpsin under A}]
For arbitrary $x_i\geq 0$, $1\leq i\leq b_n$, it holds
\begin{align*}
& f_n(x_1, \ldots, x_{b_n}) 
\leq \frac{1}{b_n} \sum_{j=1}^{b_n}\E{\abs{\frac{1}{b_n-1} \sum_{\substack{1\leq k\leq b_n\\k\neq j}} 2x_k \E{\abs{Z}}+x_j(\abs{Z}+\E{\abs{Z}})}^2}\\
\leq & 8 \E{\abs{Z}}^2 \rbraces{\frac{1}{b_n-1} \sum_{k=1}^{b_n} x_k}^2  + \frac{2}{b_n}\sum_{j=1}^{b_n} x_j^2 \rbraces{\E{Z^2}+3\rbraces{\E{\abs{Z}}}^2}\\
\leq & C \rbraces{\rbraces{\frac{1}{b_n-1} \sum_{k=1}^{b_n} x_k}^2+ \frac{1}{b_n}\sum_{j=1}^{b_n} x_j^2}.
\end{align*}
Hence, 
$$ \hat{\psi}_n \leq C  \rbraces{\rbraces{\frac{1}{b_n-1} \sum_{j=1}^{b_n} \hat{\sigma}_{j,n}}^2+ \frac{1}{b_n}\sum_{j=1}^{b_n} \hat{\sigma}_{j,n}^2}.$$
The boundedness of these expressions has already been checked in the proof of Lemma \ref{Lemma: estimated centring term alternative}.

\end{proof}

\subsection{Proof of Theorems \ref{Cor: Asymptotic test H}, \ref{Cor: Asymptotic test A} and \ref{Cor: Asymptotic test, sigma=c}}
\begin{proof}[Proof of Theorem \ref{Cor: Asymptotic test H}]
We can rewrite the test statistic as
\begin{align*}
& \sqrt{b_n}\frac{1}{\hat{\psi}_n}\rbraces{\frac{\sqrt{\ell_n}}{\hat{\kappa}(n)}U(n) -  \frac{1}{b_n(b_n-1)}\sum_{1\leq j\neq k \leq b_n} \sqrt{\hat{\sigma}_j^2+\hat{\sigma}_{k,n}^2} \sqrt{\frac{2}{\pi}}}\\
= & \frac{\psi}{\hat{\psi}_n} \left(\frac{\sqrt{b_n}}{\psi} \frac{\sqrt{\ell_n}U(n)}{\kappa_Y} \rbraces{\frac{\kappa_Y-\hat{\kappa}(n)}{\hat{\kappa}(n)}} + \frac{\sqrt{b_n}}{\psi} \rbraces{\frac{\sqrt{\ell_n}}{\kappa_Y}U(n)-\int_0^1\int_0^1 \sqrt{\sigma^2(x)+\sigma^2(y)} \mathrm{d}x\mathrm{d}y \sqrt{\frac{2}{\pi}} }  \right. \\
 & \left.+  \frac{\sqrt{b_n}}{\psi}\rbraces{\int_0^1\int_0^1 \sqrt{\sigma^2(x)+\sigma^2(y)} \mathrm{d}x\mathrm{d}y \sqrt{\frac{2}{\pi}}-\frac{1}{b_n(b_n-1)}\sum_{1\leq j\neq k \leq b_n} \sqrt{\hat{\sigma}_{j,n}^2+\hat{\sigma}_{k,n}^2} \sqrt{\frac{2}{\pi}}} \right).
\end{align*}
The first of these terms converges in probability towards zero due to Theorem \ref{THM: LLN} and Proposition \ref{Prop: convergence estimator kappahat kappaY under H}. The second term converges in distribution towards the desired standard normal limit due to Theorem \ref{Thm: Central limit theorem}, while the third one converges in $L^1$ towards zero due to Proposition \ref{Prop: Convergence estimated centring term}. Moreover,  $ \psi/\hat{\psi}_n$ converges in $L^1$ towards one due to Proposition \ref{Prop: hatpsin psin approx H}. Since $L^1$-convergence implies convergence in probability, the assertion then follows by Slutzky's lemma. 
\end{proof}

\begin{proof}[Proof of Theorem \ref{Cor: Asymptotic test A}]
The proof follows from a combination of Theorem \ref{Thm: Behaviour under A}, Lemmas \ref{Lemma: estimated centring term alternative} and \ref{Lemma: hatpsin under A} and either Proposition \ref{Prop: convergence estimator kappahat kappaY under H} or Lemma \ref{Lemma: convergence estimator kappahat kappaY under A}, depending on the form of $\mu$.
\end{proof}

\begin{proof}[Proof of Theorem \ref{Cor: Asymptotic test, sigma=c}]
The proof is a mere consequence of Theorem \ref{THM: LLN}, Theorem \ref{Thm: Central limit theorem} and Corollary  \ref{Cor: Convergence speed LRV esti}.
\end{proof}
\pagebreak

\section{Auxiliary results from the literature}
This section collects some key results from the literature that are essential tools for our proofs. 

\begin{theorem}[Theorem 10.7 in Bradley \citep{Bradley.2007}]
Let $(X_i)_{i\in \mathbb{Z}}$ be a strictly stationary, $\alpha$-mixing sequence of random variables such that $\E{X_0}=0$. Suppose that for some $\delta>0$, one has that $\E{\abs{X_0}^{2+\delta}}<\infty$ and that the mixing-coefficients satisfy $\sum_{k=1}^{\infty} \alpha(k)^{\delta/(2+\delta)}<\infty$.
\begin{enumerate}
\item Then $\kappa^2:= \E{X_0^2}+ 2 \sum_{k=1}^\infty \E{X_0X_k}$ exists in $[0,\infty)$ and the sum is absolutely convergent.
\item If also $\kappa^2>0$, then $\sum_{i=1}^n X_i/(\sqrt{n}\kappa) \distConv \NoD{0}{1}$ as $n\rightarrow \infty$.
\end{enumerate}
\end{theorem}
\begin{theorem}[Davydov's covariance inequality; see, Theorem 3, Section 1.2, in  \citep{Doukhan.1994}]{
\label{Appendix: Auxiliary: THM Davydov}
Let $X$ and $Y$ be two random variables that are measurable with respect to the $\sigma$-fields $\mathcal{A}$ and $\mathcal{B}$, respectively. Then it holds

$$\abs{\cov{X}{Y}}\leq 8 \alpha(\mathcal{A},\mathcal{B})^{1/r} \|X\|_p \|Y\|_q,$$
for any $p,q,r\geq 1$ such that $\frac{1}{p}+\frac{1}{q}+\frac{1}{r}=1$.
}\end{theorem}

\begin{theorem}[Theorem 1, Section 1.1, in  \citep{Doukhan.1994}]{
\label{Literature: Doukhan Thm 1 Sec 1.1}
Let $(\mathcal{A}_n)_{n\in \N}$ and $(\mathcal{B}_n)_{n\in\N}$ be two sequences of $\sigma$-fields such that $(\mathcal{A}_n\vee \mathcal{B}_n)_{n\in \N}$ are independent. Then it holds
$$\alpha(\bigvee_{n=1}^\infty \mathcal{A}_n,\bigvee_{n=1}^\infty \mathcal{B}_n) \leq \sum_{n=1}^\infty \alpha(\mathcal{A}_n, \mathcal{B}_n).$$
}\end{theorem}

\begin{theorem}[Theorem 2 in Peligrad \citep{Peligrad.2002}]
\label{Literature: Thm Peligrad}
Let $X_1, \ldots, X_n$ be a sequence of real-valued integrable random variables  on some probability space $(\Omega, \mathcal{F}, \Pb)$. Then, one can redefine the sequence on a possibly richer probability space together with an independent sequence of random variables $X_1', \ldots, X_n'$ such that for each $i$, $X_i'$ has the same distribution as $X_i$, $X_i'$ is independent of the $\sigma$-field generated by $X_1, \ldots, X_{i-1}$, and
$$\E{\abs{X_i-X_i'}}\leq 4 \int_0^{\alpha(i)} Q_{\abs{X_i}}(u)\mathrm{d}u,$$

where $\alpha(i):=\alpha(\sigma(X_1, \ldots, X_{i-1}), \sigma(X_i))$ and $Q_{\abs{X_i}}(u)=\inf \{x\in \R: \Pb(\abs{X_i}>x)\leq u\}$.
\end{theorem}

Frequently in our proofs, we require inequalities bounding the moments of a sum  of some random variables.
\begin{lemma}[$c_r$-inequality]{
\label{Literature: Lemma cr-ineq}
For two random variables $X$ and $Y$ with existing $r$-th moments for some $r> 0$, it holds
\begin{equation*}
\E{\abs{X+Y}^r} \leq \max\rbraces{ 1,2^{r-1} } \cdot \rbraces{\E{\abs{X}^r}+ \E{\abs{Y}^r}}.
\end{equation*}
}\end{lemma}

\begin{theorem}[Rosenthal-type inequality; see, Theorem 2, Section 1.4, in \citep{Doukhan.1994}] 
\label{Literature: Thm Rosenthal}
Let $(X_i)_{i\in \N}$ be an $\alpha$-mixing sequence of random variables and let $T$ be a finite subset of $\N$ such that $\E{X_t}=0$ for all $t\in T$.
Assume there exists an $\varepsilon>0$ and a constant $c \in 2\N$ with $c\geq \tau$ such that
$$\sum_{k=1}^\infty \rbraces{k+1}^{c-2} \alpha(k)^{\varepsilon/(c+\varepsilon)}<\infty\quad \text{as well as } \quad
 \E{\abs{X_t}^{\tau+\varepsilon}}<\infty$$ for some $\tau>0$ and all $t\in T$. Then there exists a constant $C$ depending only on $\tau$ and the mixing coefficients $\alpha(k)$ of $X$ such that
 $$\E{\abs{\sum_{t\in T}X_t}^\tau}\leq C D(\tau, \varepsilon,T),$$
 where
 $$
 D(\tau, \varepsilon,T) := \begin{cases}
L(\tau,0,T) \quad \text{ for } 0<\tau\leq 1, \varepsilon\geq 0,\\
L(\tau,\varepsilon,T) \quad \text{ for } 1<\tau \leq 2, \varepsilon>0,\\
\max\rbraces{L(\tau,\varepsilon,T), (L(2,\varepsilon,T))^{\tau/2}} \quad \text{ for } \tau>2, \varepsilon>0,\\
\end{cases}
$$
with $$L(\mu, \varepsilon,T):= \sum_{t\in T} \rbraces{\E{\abs{X_t}^{\mu+\varepsilon}}}^{\mu/(\mu+\varepsilon)}.$$
\end{theorem}

\begin{theorem}[Theorem 1 in Yokoyama \citep{Yokoyama.1980}] 
\label{Literature: Thm Yokoyama}
Let $(X_i)_{i\in \N}$ be a strictly stationary, $\alpha$-mixing sequence of random variables such that $\E{X_1}=\mu$.  Assume there exist constants $\delta$ with $0<\delta\leq \infty$ and  $t$ with $2\leq t <2+\delta$ such that $$\E{\abs{X_1}^{2+\delta}}<\infty \quad \text{and} \quad \sum_{k=1}^{\infty} k^{t/2-1}\alpha(k)^{\rbraces{2+\delta-t}/(2+\delta)}<\infty.$$
Then it holds
\begin{equation*}
\E{\abs{\sum_{i=1}^n(X_i-\mu)}^t} \leq C n^{t/2}
\end{equation*}
\end{theorem}

The next two theorems enable us to control the difference between some distribution function $F$ and the standard normal distribution $\Phi$.

\begin{theorem}[Theorem 9, Chapter V, in Petrov \citep{Petrov.1975}]
\label{Literature: Petrov}
Let $\Phi$ denote the distribution function of the  standard normal distribution, let $F$ be another arbitrary distribution function and define $\Delta:= \sup_x \abs{F(x)-\Phi(x)}$. Suppose that $0<\Delta<1/\sqrt{e}$ and that $F$ has finite absolute moments of order $p$ for some $p>0$. Then there exists a constant $C_p$ depending only on $p$ such that
\begin{equation*}
\abs{F(x)-\Phi(x)} \leq \frac{C_p \Delta\log\rbraces{\frac{1}{\Delta}}^{p/2}+\lambda_p}{1+\abs{x}^p}
\end{equation*}
for all $x\in \R$, where
\begin{equation*}
\lambda_p= \abs{\int_{-\infty}^\infty\abs{x}^p\mathrm{d}F(x)-\int_{-\infty}^\infty\abs{x}^p\mathrm{d}\Phi(x)}.
\end{equation*}
\end{theorem}

\begin{theorem}[Theorem 1 in Tikhomirov \citep{Tikhomirov.1980}]
\label{Literature: Tikhomirov}
Let $(X_i)_{i\in \N}$ be a strictly stationary, $\alpha$-mixing sequence of random variables with mean zero and finite variance. Let $\Phi$ denote the distribution function of the standard normal distribution and define $$F_n(x):= \Pb\rbraces{\frac{\sum_{i=1}^nX_i}{\sqrt{\Var{\sum_{i=1}^n X_i}}}\leq x}\quad  \text{as well as} \quad \Delta_n:=\sup_x\abs{F_n(x)-\Phi(x)}.$$
Suppose that there exist constants $C_1>0$ and $\rho>1$ such that $$\alpha_X(k)\leq C_1k^{-\rho(2+\delta)(1+\delta)/\delta^2}$$ holds for all $k\in \N$ and some  $0<\delta\leq 1$ such that $\E{\abs{X_1}^{2+\delta}}<\infty$. Then,
$$\kappa^2=\E{X_1^2}+ 2\sum_{k=2}^\infty \E{X_1X_k}<\infty$$
 and if $\kappa^2>0$,  there exists a constant $C_2$ depending solely on $C_1, \rho $ and $\delta$ such that
\begin{equation*}
\Delta_n\leq C_2 n^{-(\delta/2)(\rho-1)/(\rho+1)}.
\end{equation*}
\end{theorem}

Lastly, we state a very useful result concerning the properties of a c\`adl\`ag function.

\begin{lemma}[Lemma 1, Section 14, in \citep{Billingsley.1968}]{
Let $f:[0,1]\rightarrow \R$ be a c\`adl\`ag function (right-continuous with left-hand limits). Then for every $\varepsilon>0$, there exist  points $0=t_0<t_1< ... <t_r=1$ such that
$$\sup_{s,t\in[t_{i-1},t_i)} \abs{f(s)-f(t)}<\varepsilon$$
for all $i=1, ..., r$.
}\end{lemma}
As pointed out  in \citep{Billingsley.1968}, this clearly implies that there are only finitely many jumps that  exceed a given positive number. Moreover, every c\`adl\`ag function $f$ is bounded on $[0,1]$ and has at most countably many discontinuities.

\pagebreak
\section{Additional results from the simulation study and the data examples}
\subsection{Performance of the original test from Theorem \ref{Cor: Asymptotic test H} for a non-stationary variance}

Table \ref{Table: Theoretical Rejection Rates Variance varies with 0.6} shows the simulated rejection probabilities  for $n=500$ and for different variance functions with changes of magnitude $\Theta_\sigma=0.6\sqrt{2}\approx 0.85$. 

\begin{table}[H]
\centering
    \caption{
    Simulated rejection probabilities of the original test at the nominal significance level $\alpha=0.05$ for the sample size $n = 500$ under the null hypothesis $\mathbb{H}$ and various local alternatives $\mathbb{A}1$ to $\mathbb{A}5$ with effect sizes of magnitude $n^{-1/2}$ and for different data-generating processes and different variance functions.}
    \label{Table: Theoretical Rejection Rates Variance varies with 0.6}
    \begin{tabular}{|c|cccccc|}
       \hline
        &  N(0,1) & Exp(1) & AR(1), 0.4 & AR(1), 0.7 & ARMA(2,2) & GARCH(1,1) \\ \hline
               \multicolumn{7}{|c|}{$\sigma(x)=(1-0.3\sqrt{2})+0.6\sqrt{2}\cdot x$} \\\hline 
       $ \mathbb{H}$&0.058 & 0.052 & 0.080 & 0.113 & 0.057 & 0.067 \\
        $\mathbb{A}_1$ &0.434 & 0.422 & 0.488 & 0.596 & 0.430 & 0.452 \\
       $\mathbb{A}_2$  &0.427 & 0.423 & 0.473 & 0.584 & 0.433 & 0.436 \\
       $\mathbb{A}_3$  &0.908 & 0.922 & 0.931 & 0.959 & 0.910 & 0.912 \\
       $\mathbb{A}_4$  &0.791 & 0.779 & 0.825 & 0.888 & 0.784 & 0.787 \\
       $\mathbb{A}_5$  &0.341 & 0.304 & 0.394 & 0.479 & 0.330 & 0.339 \\\hline
          \multicolumn{7}{|c|}{$\sigma(x)=1+0.3\sqrt{2} \cdot \sin(4\pi x) $} \\\hline 
     $ \mathbb{H}$  &0.076 & 0.072 & 0.101 & 0.152 & 0.082 & 0.088 \\
        $\mathbb{A}_1$ &0.435 & 0.452 & 0.488 & 0.590 & 0.439 & 0.464 \\
       $\mathbb{A}_2$  &0.433 & 0.435 & 0.481 & 0.606 & 0.438 & 0.443 \\
       $\mathbb{A}_3$  &0.902 & 0.898 & 0.926 & 0.958 & 0.899 & 0.902 \\
       $\mathbb{A}_4$  &0.758 & 0.772 & 0.804 & 0.878 & 0.777 & 0.764 \\
       $\mathbb{A}_5$  &0.354 & 0.331 & 0.384 & 0.487 & 0.339 & 0.340 \\\hline
              \multicolumn{7}{|c|}{$\sigma(x)=   (1- 0.3\sqrt{2})\cdot\1_{\{0\leq x< 1/2\}}+ (1+0.3\sqrt{2})\cdot\1_{\{1/2\leq x\leq 1\}}$} \\\hline   
      $ \mathbb{H} $&0.053 & 0.056 & 0.070 & 0.107 & 0.058 & 0.058 \\
       $\mathbb{A}_1$  &0.380 & 0.355 & 0.422 & 0.531 & 0.376 & 0.386 \\
       $\mathbb{A}_2$  &0.330 & 0.346 & 0.378 & 0.484 & 0.340 & 0.342 \\
        $\mathbb{A}_3$ &0.850 & 0.872 & 0.866 & 0.920 & 0.840 & 0.850 \\
        $\mathbb{A}_4$ &0.682 & 0.675 & 0.732 & 0.819 & 0.680 & 0.690 \\
       $\mathbb{A}_5$  &0.240 & 0.198 & 0.260 & 0.368 & 0.249 & 0.234 \\\hline
    \end{tabular}
\end{table}
\pagebreak

\subsection{Performance of the simplified test from Theorem \ref{Cor: Asymptotic test, sigma=c}}
Table \ref{Table: Theoretical Results simplified procedure} reports the simulated rejection rates for the simplified test.  To facilitate a comparison, we additionally report for both tests the size-corrected empirical power in Table \ref{Table: Empirical Results both procedures}, where we use the empirical 95\% percentile of the respective test results for the same distribution and the same sample size as critical values. Moreover, Figure \ref{Figure: Asymptotic behaviour simplified procedure, empirical} shows the size-corrected empirical rejection rates of the simplified test as a function of the sample size. 

\begin{table}[H]
\centering
    \caption{Simulated rejection probabilities of the simplified test at the nominal significance level $\alpha=0.05$ for the sample sizes $n = 500, 2000$ under the null hypothesis $\mathbb{H}$ and various local alternatives $\mathbb{A}1$ to $\mathbb{A}5$ with effect sizes of magnitude $n^{-1/2}$ and for different data-generating processes.}
     \label{Table: Theoretical Results simplified procedure}
    \begin{tabular}{|c|cccccc|}
       \hline
         & N(0,1) & Exp(1) & AR(1), 0.4 & AR(1), 0.7 & ARMA(2,2) & GARCH(1,1) \\ \hline    
          \multicolumn{7}{|c|}{$n=500$} \\\hline     				
      $\mathbb{H}$ &  0.081 & 0.078 & 0.102 & 0.150 & 0.078 & 0.090 \\
     $\mathbb{A}_1$ &   0.496 & 0.497 & 0.552 & 0.643 & 0.508 & 0.528 \\
      $\mathbb{A}_2$ &  0.500 & 0.532 & 0.561 & 0.668 & 0.516 & 0.537 \\
      $\mathbb{A}_3$ &  0.942 & 0.932 & 0.947 & 0.973 & 0.939 & 0.931 \\
      $\mathbb{A}_4$ &  0.823 & 0.830 & 0.855 & 0.908 & 0.831 & 0.829 \\
      $\mathbb{A}_5$ &  0.432 & 0.418 & 0.479 & 0.569 & 0.445 & 0.438 \\\hline
        \multicolumn{7}{|c|}{$n=2000$} \\\hline   
      $\mathbb{H}$ &  0.077 & 0.071 & 0.084 & 0.137 & 0.083 & 0.082 \\
    $\mathbb{A}_1$ &    0.417 & 0.405 & 0.458 & 0.507 & 0.405 & 0.426 \\
    $\mathbb{A}_2$ &    0.604 & 0.623 & 0.641 & 0.720 & 0.618 & 0.632 \\
    $\mathbb{A}_3$ &    0.930 & 0.925 & 0.938 & 0.959 & 0.916 & 0.929 \\
    $\mathbb{A}_4$ &    0.815 & 0.816 & 0.849 & 0.887 & 0.827 & 0.834 \\
    $\mathbb{A}_5$ &    0.860 & 0.851 & 0.877 & 0.913 & 0.845 & 0.862 \\\hline 
    \end{tabular}
\end{table}
\begin{table}[H]
    \centering
    \caption{
    Simulated rejection probabilities of the original and the simplified test at the significance level $\alpha=0.05$ for the sample sizes $n = 500, 2000$ under the null hypothesis $\mathbb{H}$ and various local alternatives $\mathbb{A}1$ to $\mathbb{A}5$ with effect sizes of magnitude $n^{-1/2}$ and for different data-generating processes. For the alternatives, a size-correction at significance level $\alpha=0.05$ has been conducted.}
    \label{Table: Empirical Results both procedures}
    \begin{tabular}{|c|cccccc|}
       \hline
          &N(0,1) & Exp(1) & AR(1), 0.4 & AR(1), 0.7 & ARMA(2,2) & GARCH(1,1) \\ \hline  
          \multicolumn{7}{|c|}{Original test} \\\hline  
         \multicolumn{7}{|c|}{$n=500$} \\\hline     
    $\mathbb{H}$  &  0.079 & 0.074 & 0.085 & 0.148 & 0.072 & 0.075 \\
   $\mathbb{A}_1$     &0.407 & 0.442 & 0.426 & 0.450 & 0.425 & 0.428 \\
    $\mathbb{A}_2$      &0.412 & 0.426 & 0.434 & 0.420 & 0.432 & 0.421 \\
    $\mathbb{A}_3$    &0.901 & 0.908 & 0.917 & 0.916 & 0.896 & 0.908 \\
    $\mathbb{A}_4$     &0.748 & 0.772 & 0.776 & 0.772 & 0.765 & 0.765 \\
    $\mathbb{A}_5$    &0.298 & 0.322 & 0.313 & 0.236 & 0.325 & 0.311 \\   \hline         
        \multicolumn{7}{|c|}{$n=2000$} \\\hline  
    $\mathbb{H}$  &  0.073 & 0.074 & 0.082 & 0.122 & 0.070 & 0.082 \\
     $\mathbb{A}_1$    & 0.334 & 0.347 & 0.362 & 0.339 & 0.349 & 0.316 \\
    $\mathbb{A}_2$    &0.534 & 0.539 & 0.540 & 0.540 & 0.557 & 0.536 \\
    $\mathbb{A}_3$     &0.900 & 0.905 & 0.899 & 0.904 & 0.895 & 0.896 \\
    $\mathbb{A}_4$    &0.772 & 0.761 & 0.780 & 0.790 & 0.772 & 0.755 \\
      $\mathbb{A}_5$    &0.802 & 0.802 & 0.816 & 0.811 & 0.819 & 0.787 \\\hline            				
 \multicolumn{7}{|c|}{Simplified test} \\\hline  
  \multicolumn{7}{|c|}{$n=500$} \\\hline           				
    $\mathbb{H}$    &0.081 & 0.078 & 0.102 & 0.150 & 0.078 & 0.090 \\
     $\mathbb{A}_1$     &0.407 & 0.410 & 0.429 & 0.446 & 0.437 & 0.428 \\
     $\mathbb{A}_2$   &0.387 & 0.425 & 0.413 & 0.406 & 0.423 & 0.413 \\
     $\mathbb{A}_3$    &0.901 & 0.897 & 0.903 & 0.918 & 0.913 & 0.892 \\
    $\mathbb{A}_4$    &0.752 & 0.763 & 0.768 & 0.773 & 0.770 & 0.753 \\
      $\mathbb{A}_5$    &0.298 & 0.296 & 0.295 & 0.233 & 0.349 & 0.304 \\\hline
                \multicolumn{7}{|c|}{$n=2000$} \\\hline  
     $\mathbb{H}$   &0.077 & 0.071 & 0.084 & 0.137 & 0.083 & 0.082 \\
    $\mathbb{A}_1$      &0.346 & 0.338 & 0.354 & 0.339 & 0.320 & 0.328 \\
    $\mathbb{A}_2$    &0.534 & 0.562 & 0.541 & 0.548 & 0.527 & 0.530 \\
     $\mathbb{A}_3$    &0.902 & 0.906 & 0.899 & 0.902 & 0.882 & 0.890 \\
     $\mathbb{A}_4$   &0.767 & 0.765 & 0.786 & 0.766 & 0.770 & 0.759 \\
     $\mathbb{A}_5$     &0.814 & 0.810 & 0.814 & 0.811 & 0.790 & 0.792 \\ \hline

    \end{tabular}
\end{table}

\begin{center}
\begin{figure}[H]
 \begin{subfigure}[c]{0.49\textwidth}
  \includegraphics[width=\textwidth,  trim={0cm 0cm 0.8cm 1.9cm},clip]{Pictures/ConstVar_size_wachsendesn.pdf}
 \end{subfigure}
   \begin{subfigure}[c]{0.49\textwidth}
  \includegraphics[width=\textwidth,  trim={0cm 0cm 0.8cm 1.9cm},clip]{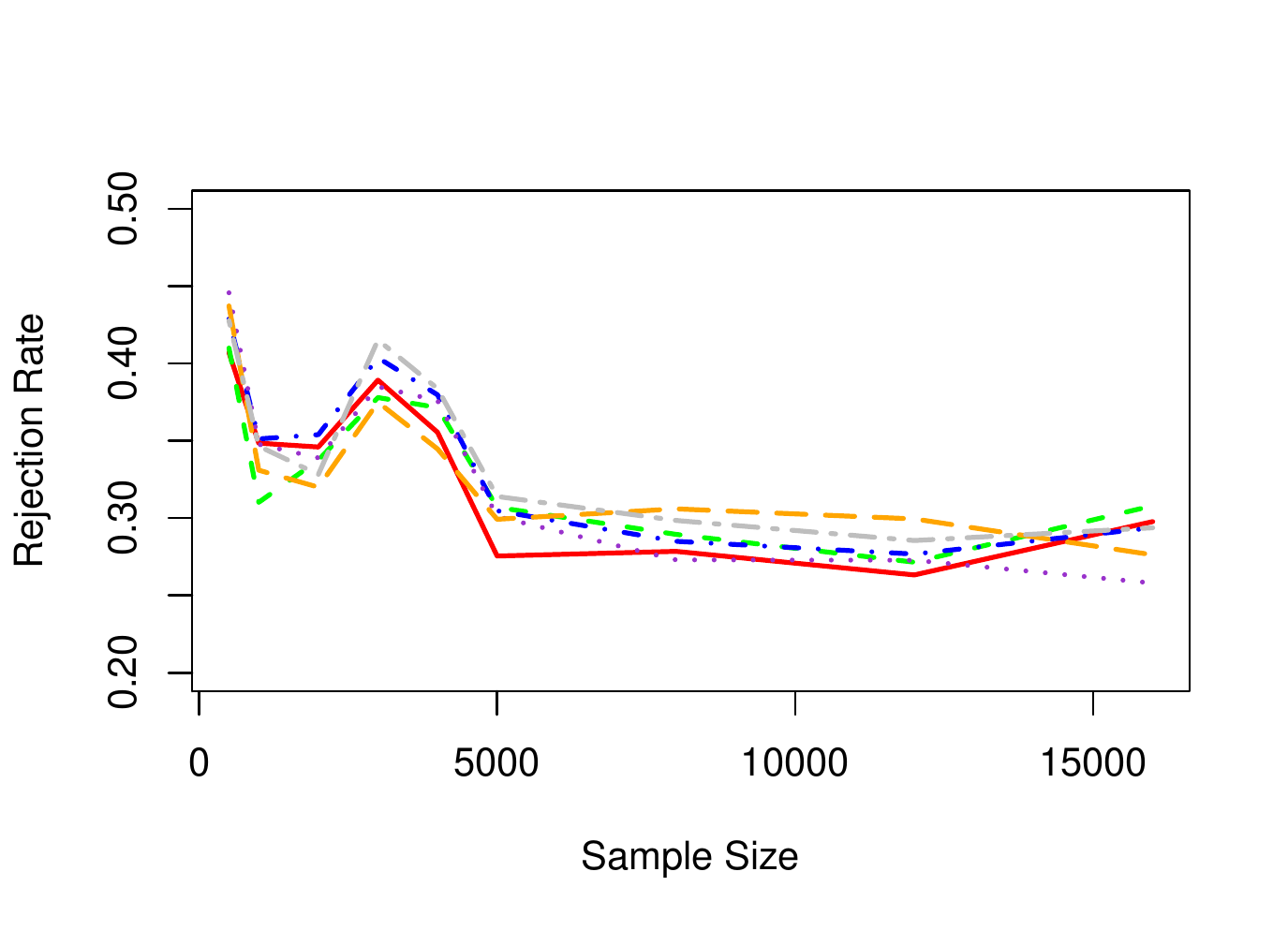}
  \end{subfigure}
   \begin{subfigure}[c]{0.49\textwidth}
  \includegraphics[width=\textwidth,  trim={0cm 0cm 0.8cm 1.9cm},clip]{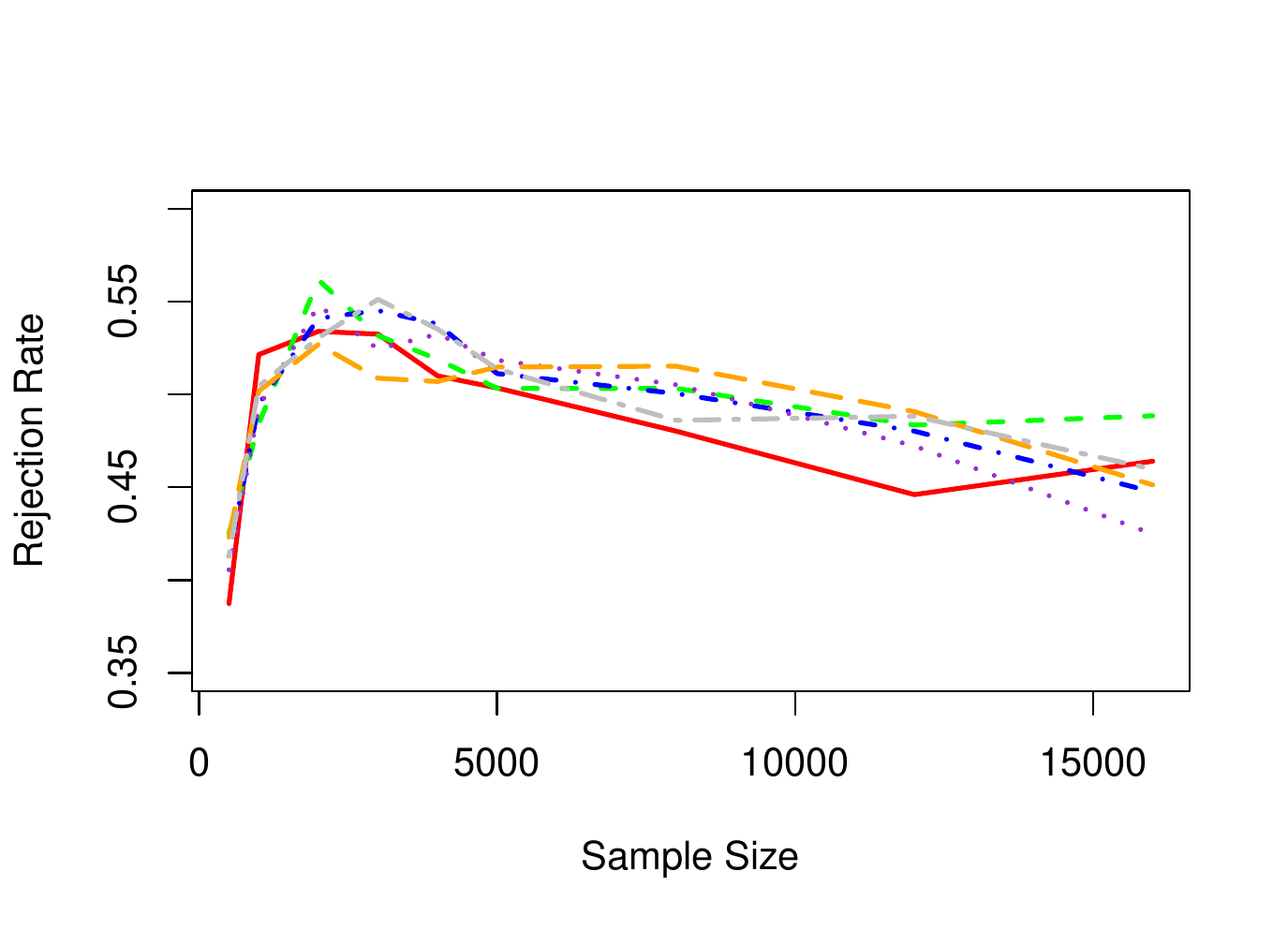}
  \end{subfigure}
   \begin{subfigure}[c]{0.49\textwidth}
  \includegraphics[width=\textwidth,  trim={0cm 0cm 0.8cm 1.9cm},clip]{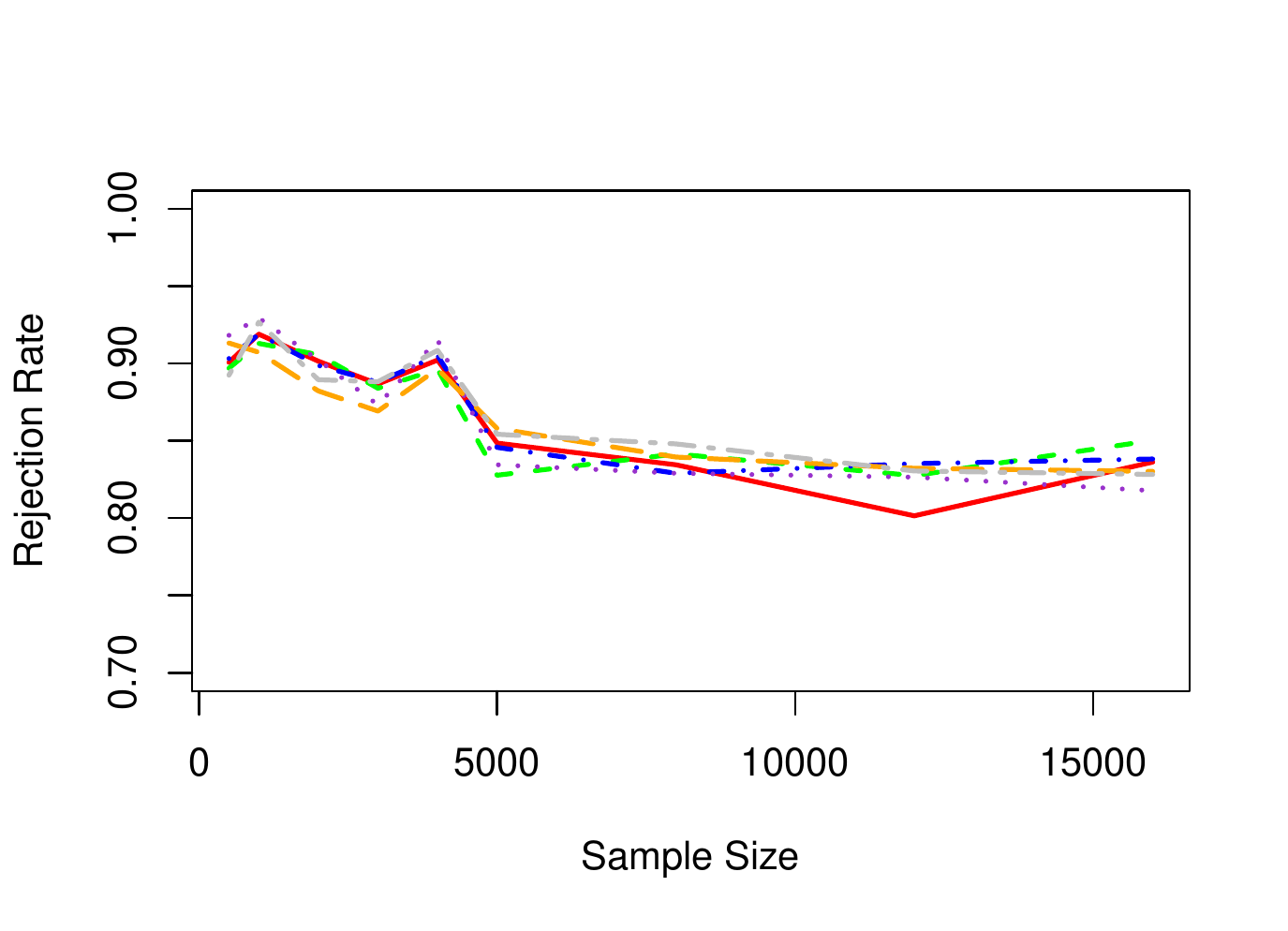}
  \end{subfigure}
     \begin{subfigure}[c]{0.49\textwidth}
  \includegraphics[width=\textwidth,  trim={0cm 0cm 0.8cm 1.9cm},clip]{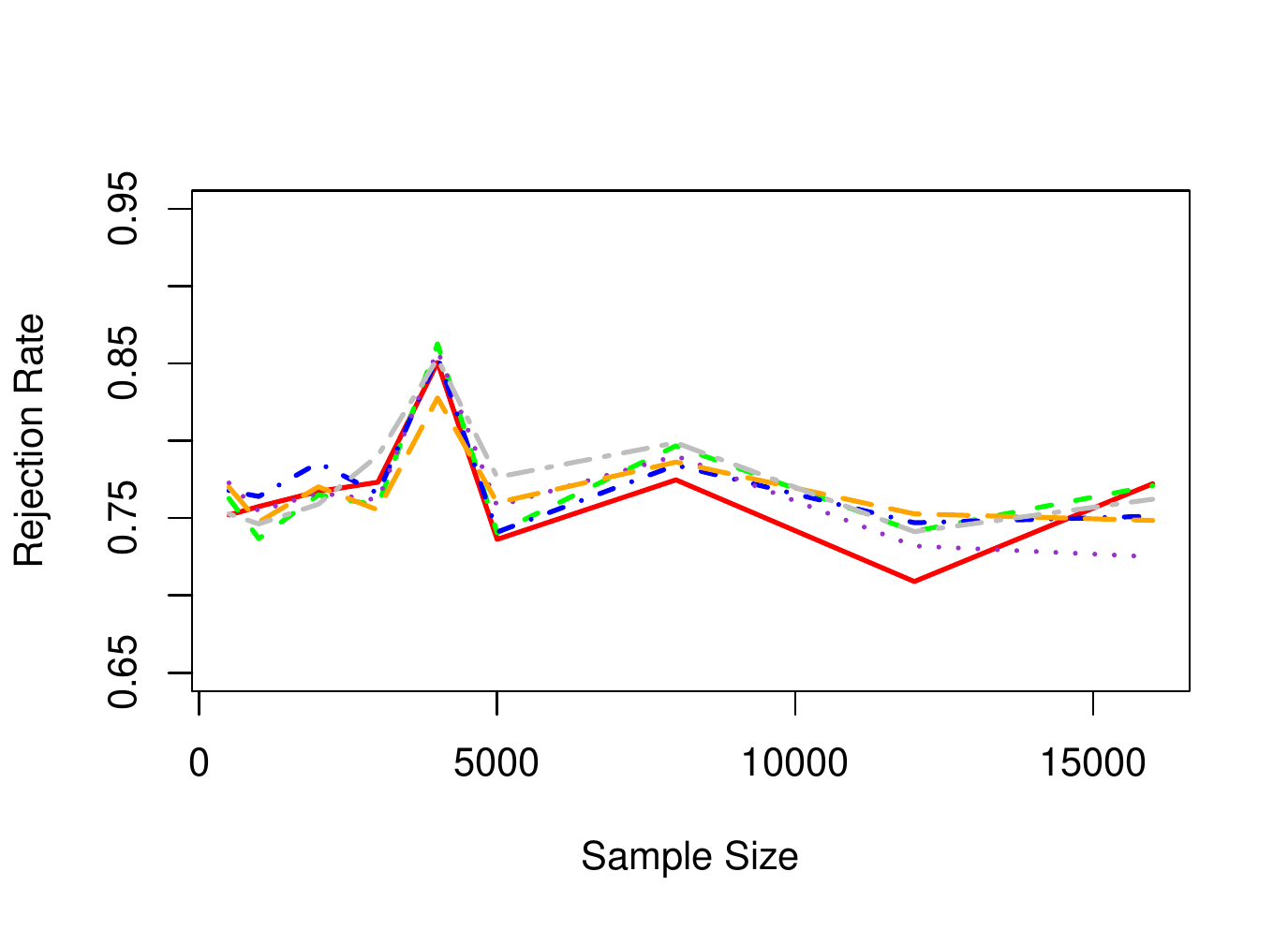}
  \end{subfigure}
     \begin{subfigure}[c]{0.49\textwidth}
  \includegraphics[width=\textwidth,  trim={0cm 0cm 0.8cm 1.9cm},clip]{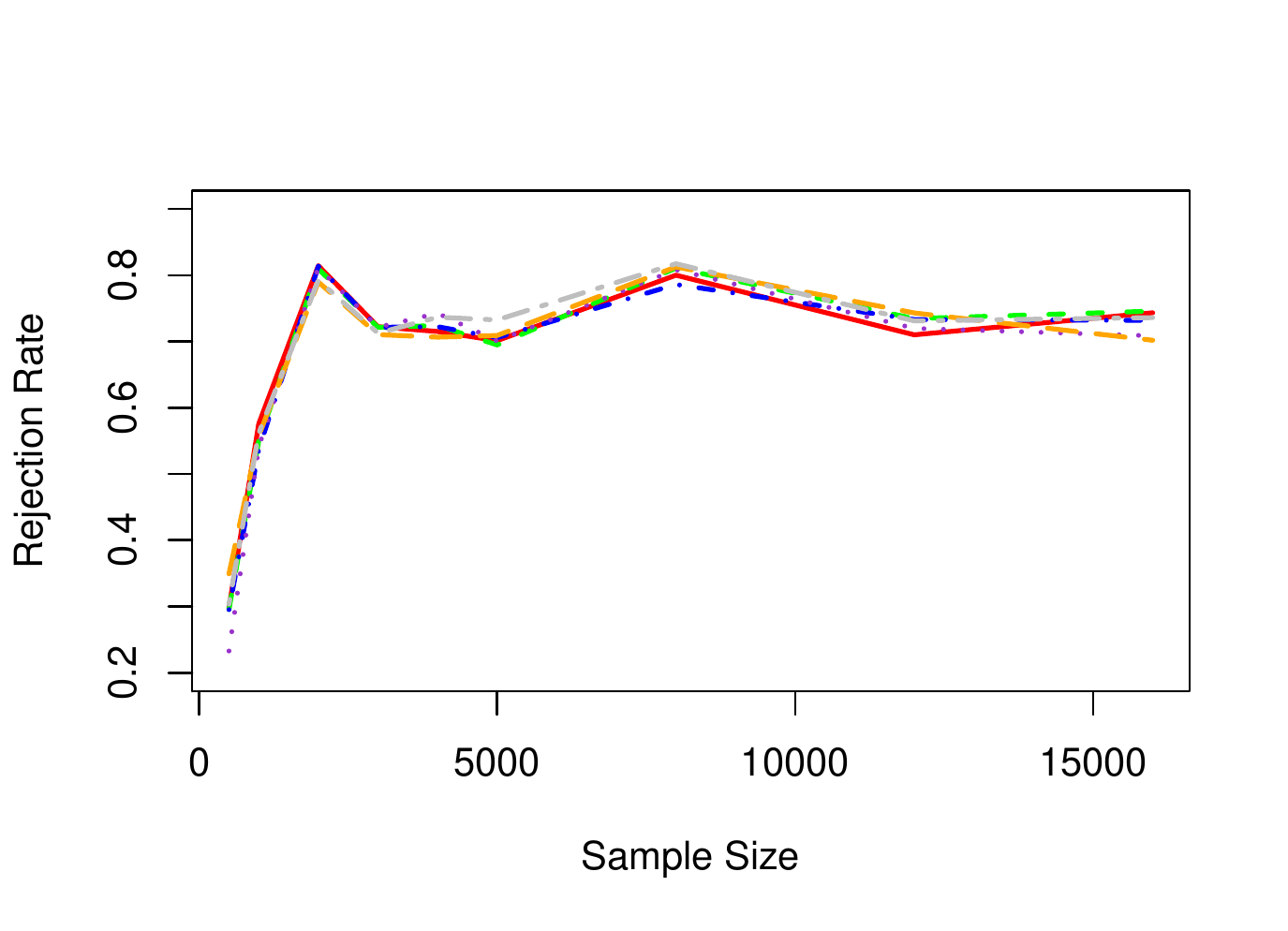}
  \end{subfigure}
   \begin{subfigure}[c]{\textwidth}

  \includegraphics[width=\textwidth, trim={0cm 7cm 0cm 6.5cm},clip]{Pictures/Legend_horizontal_thick.pdf}
  \end{subfigure}
  \caption{Empirical rejection rates of the simplified test at the significance level $\alpha=0.05$ under the hypothesis (top left) and the (local) alternatives $\mathbb{A}_1$ (top right) -$\mathbb{A}_5$ (bottom right) as a function of the sample size for different distributions of the data-generating process. For  the alternatives, a size-correction
at significance level $\alpha=0.05$ has been conducted.} 
\label{Figure: Asymptotic behaviour simplified procedure, empirical}
\end{figure}
\end{center}

\pagebreak

\subsection{Performance of the long run variance estimator}
Table \ref{Table: LRV estimator Bias RMSE both qs} shows the performance of the estimator $\hat{\kappa}_{\tilde{X}}$ for different choices $q=0.3$ and $q=0.4$ of the block length $\elln=n^q$. 

\begin{table}[H]
\caption{Simulated Bias and RMSE of $\hat{\kappa}_{\tilde{X}}$ for the mean functions represented by $\mathbb{H}$ and the fixed alternatives $\mathbb{A}'_1$-$\mathbb{A}'_5$, the sample lengths $n=500, 2000$, and for the parameter choices $q=0.3, 0.4$ with $\elln=n^q$ and $c_0=10$. The observations are standardized to yield a theoretical long run variance $\kappa_Y^2 =1$.}
\label{Table: LRV estimator Bias RMSE both qs}
{\footnotesize
\begin{tabular}{|c|rc|rc|rc|rc|rc|rc|}
    \hline
& \multicolumn{2}{|c|}{\multirow{2}{*}{N(0,1)}} & \multicolumn{2}{c|}{\multirow{2}{*}{Exp(1)}} &  \multicolumn{2}{c|}{\multirow{2}{*}{AR(1), 0.4}} &  \multicolumn{2}{c|}{\multirow{2}{*}{AR(1), 0.7}} &  \multicolumn{2}{c|}{\multirow{2}{*}{ARMA(2,2)}} &  \multicolumn{2}{c|}{\multirow{2}{*}{GARCH(1,1)}} \\
  & \multicolumn{2}{|c|}{} & \multicolumn{2}{c|}{}  & \multicolumn{2}{c|}{} & \multicolumn{2}{c|}{} & \multicolumn{2}{c|}{} & \multicolumn{2}{c|}{} \\
 & \multicolumn{1}{|c}{Bias} & RMSE &  \multicolumn{1}{c}{Bias} & RMSE &  \multicolumn{1}{c}{Bias} & RMSE & \multicolumn{1}{c}{Bias} & RMSE & \multicolumn{1}{c}{Bias} & RMSE & \multicolumn{1}{c}{Bias} & RMSE 
  \\ \hline
 \multicolumn{13}{|c|}{$n=500$, $q=0.4$}\\\hline
   $\mathbb{H}$ &   0.003 & 0.125 & -0.003 & 0.129 & -0.038 & 0.123 & -0.121 & 0.163 & -0.001 & 0.122 & -0.017 & 0.137 \\
   $\mathbb{A}'_1$ &     0.005 & 0.121 & -0.002 & 0.128 & -0.039 & 0.125 & -0.121 & 0.165 & 0.000 & 0.124 & -0.014 & 0.135 \\
   $\mathbb{A}'_2$ &     0.119 & 0.179 & 0.110 & 0.177 & 0.082 & 0.156 & 0.006 & 0.128 & 0.108 & 0.175 & 0.101 & 0.176 \\
   $\mathbb{A}'_3$  &   0.048 & 0.137 & 0.041 & 0.138 & 0.007 & 0.123 & -0.076 & 0.140 & 0.043 & 0.135 & 0.028 & 0.141 \\
    $\mathbb{A}'_4$ &   0.101 & 0.165 & 0.094 & 0.164 & 0.062 & 0.137 & -0.010 & 0.121 & 0.092 & 0.157 & 0.085 & 0.160 \\
   $\mathbb{A}'_5$ &    0.282 & 0.324 & 0.284 & 0.328 & 0.252 & 0.298 & 0.190 & 0.247 & 0.278 & 0.320 & 0.276 & 0.319 \\
 \hline
  \multicolumn{13}{|c|}{$n=2000$, $q=0.4$}\\\hline
 $\mathbb{H}$ &       0.002 & 0.079 & -0.002 & 0.081 & -0.025 & 0.081 & -0.072 & 0.102 & -0.001 & 0.079 & -0.016 & 0.086 \\
 $\mathbb{A}'_1$ &      0.003 & 0.079 & -0.004 & 0.081 & -0.021 & 0.081 & -0.071 & 0.103 & -0.002 & 0.077 & -0.013 & 0.083 \\
  $\mathbb{A}'_2$ &   0.022 & 0.080 & 0.020 & 0.083 & 0.000 & 0.076 & -0.052 & 0.089 & 0.019 & 0.081 & 0.009 & 0.084 \\
  $\mathbb{A}'_3$ &    0.034 & 0.086 & 0.028 & 0.088 & 0.006 & 0.078 & -0.041 & 0.086 & 0.027 & 0.083 & 0.016 & 0.085 \\
  $\mathbb{A}'_4$ &    0.059 & 0.101 & 0.050 & 0.098 & 0.036 & 0.087 & -0.012 & 0.080 & 0.055 & 0.098 & 0.042 & 0.095 \\
  $\mathbb{A}'_5$ &    0.128 & 0.153 & 0.124 & 0.151 & 0.106 & 0.134 & 0.059 & 0.100 & 0.124 & 0.150 & 0.118 & 0.147 \\\hline
  \multicolumn{13}{|c|}{$n=500$, $q=0.3$}\\\hline
   $\mathbb{H}$ &      0.003 & 0.086 & -0.012 & 0.093 & -0.085 & 0.116 & -0.240 & 0.250 & -0.008 & 0.085 & -0.023 & 0.102 \\
    $\mathbb{A}'_1$ &      0.002 & 0.087 & -0.011 & 0.095 & -0.086 & 0.117 & -0.242 & 0.252 & -0.008 & 0.087 & -0.021 & 0.104 \\
     $\mathbb{A}'_2$ &     0.010 & 0.086 & -0.001 & 0.095 & -0.074 & 0.108 & -0.226 & 0.237 & 0.001 & 0.086 & -0.010 & 0.102 \\
    $\mathbb{A}'_3$ &      0.013 & 0.088 & -0.004 & 0.095 & -0.073 & 0.108 & -0.227 & 0.238 & 0.003 & 0.086 & -0.010 & 0.101 \\
    $\mathbb{A}'_4$ &      0.024 & 0.092 & 0.008 & 0.095 & -0.062 & 0.103 & -0.212 & 0.225 & 0.016 & 0.088 & 0.001 & 0.100 \\
    $\mathbb{A}'_5$ &      0.046 & 0.099 & 0.034 & 0.099 & -0.038 & 0.090 & -0.182 & 0.197 & 0.038 & 0.095 & 0.025 & 0.105 \\ \hline
          \multicolumn{13}{|c|}{$n=2000$, $q=0.3$}\\\hline
   $\mathbb{H}$ &           0.001 & 0.052 & -0.009 & 0.056 & -0.058 & 0.076 & -0.168 & 0.174 & -0.011 & 0.052 & -0.021 & 0.062 \\
    $\mathbb{A}'_1$ &      0.000 & 0.052 & -0.008 & 0.057 & -0.057 & 0.077 & -0.168 & 0.174 & -0.010 & 0.053 & -0.021 & 0.061 \\
    $\mathbb{A}'_2$ &     0.001 & 0.052 & -0.007 & 0.056 & -0.055 & 0.074 & -0.167 & 0.173 & -0.009 & 0.053 & -0.022 & 0.062 \\
    $\mathbb{A}'_3$ &     0.007 & 0.054 & -0.002 & 0.057 & -0.050 & 0.071 & -0.161 & 0.167 & -0.004 & 0.053 & -0.015 & 0.063 \\
    $\mathbb{A}'_4$ &     0.015 & 0.055 & 0.005 & 0.057 & -0.043 & 0.067 & -0.151 & 0.158 & 0.002 & 0.052 & -0.007 & 0.061 \\
    $\mathbb{A}'_5$ &     0.025 & 0.060 & 0.016 & 0.059 & -0.032 & 0.060 & -0.142 & 0.149 & 0.015 & 0.055 & 0.004 & 0.059 \\    \hline    
    \end{tabular}}
\end{table}

\pagebreak
\subsection{Data example: Central England temperature data 1659-2020}
Figure  \ref{Fig: Temperature plot and CPS with pw constant mean function} shows the estimated change points in the mean of the annual central England temperatures for the significance level $\alpha=0.05$. Additionally, it depicts the estimated piecewise constant mean function one obtains by computing the sample mean on each segment between two subsequent change points.

 \begin{figure}[H]
 \includegraphics[width=\textwidth, trim={0 0 0 1.3cm}, clip]{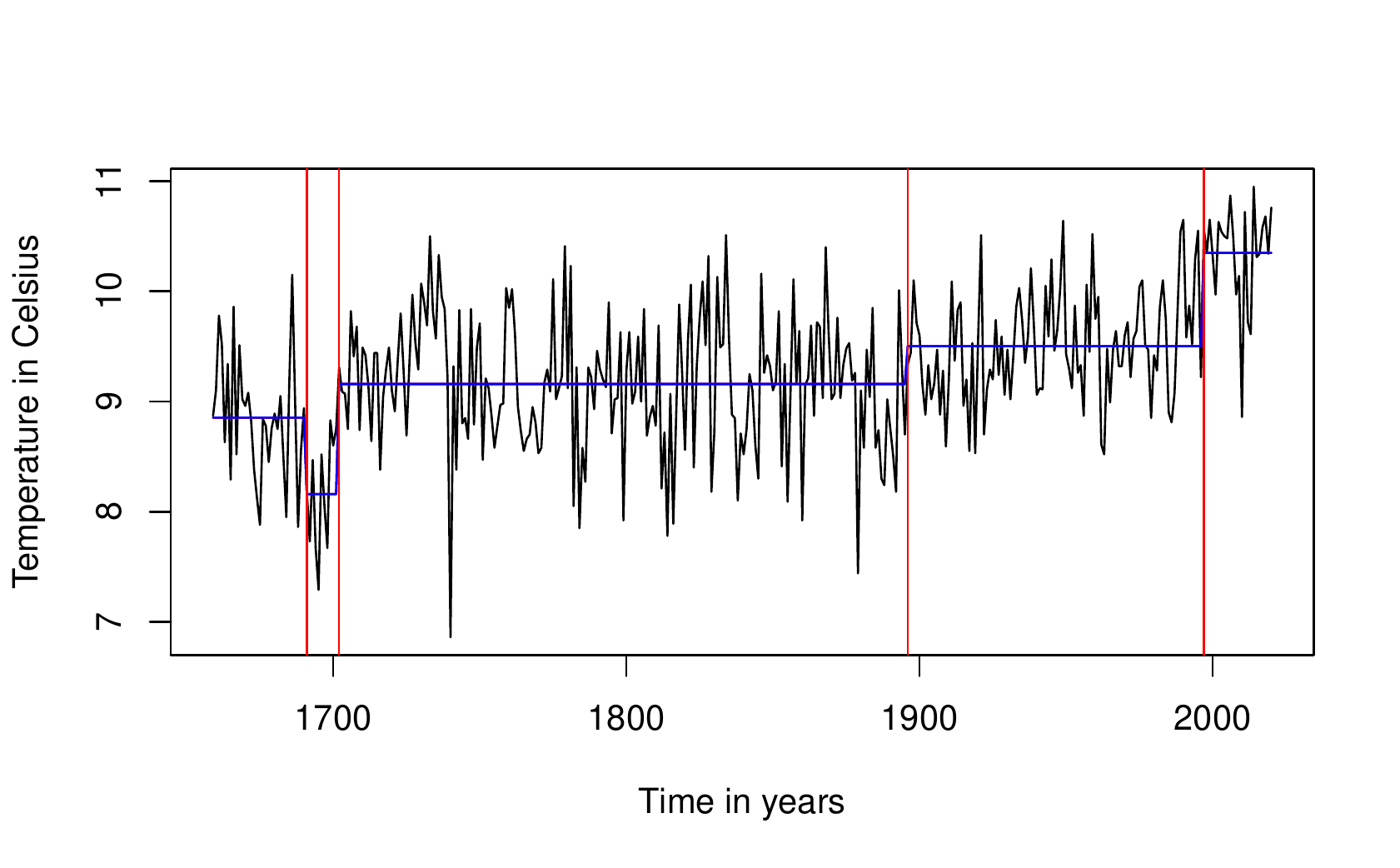}
\vspace*{-1cm}\caption{Annual central England temperatures from 1659 to 2020. Detected mean change points for $\alpha=0.05$ are marked by the red vertical lines. The estimated piecewise constant mean function is depicted by the blue line.}
\label{Fig: Temperature plot and CPS with pw constant mean function}
\end{figure}

%
%
%
%
%
%
%
%
%
%

\end{document}